\documentclass[12pt,a4paper,twoside]{article}

\usepackage{a4wide, amssymb, amsmath, amsthm, graphics, comment, xspace, enumerate}
\usepackage{graphicx}
\usepackage[a4paper,colorlinks=true,citecolor=blue,urlcolor=blue,linkcolor=blue,bookmarksopen=true,unicode=true,
          pdffitwindow=true]{hyperref}
\usepackage{algorithmic, algorithm}
\input amssym.def
\input amssym

\allowdisplaybreaks

\newcommand{\NN}{\mathbb N}

\newcommand{\CC}{\mathbb C}
\newcommand{\RR}{\mathbb R}
\newcommand{\ZZ}{\mathbb Z}
\newcommand{\ds}{\displaystyle}
\newcommand{\EE}{\mathcal E}
\newcommand{\DD}{\mathcal D}
\newcommand{\SSS}{\mathcal S}

\newcommand{\supp}{\mathrm{supp\,}}

\newcommand{\Op}{\mathrm{Op}}

\newtheorem{theorem} {Theorem}[section]
\newtheorem{proposition}{Proposition}[section]
\newtheorem{lemma}{Lemma}[section]
\newtheorem{corollary}{Corollary}[section]
\newtheorem{definition}{Definition}[section]
\newtheorem{remark}{Remark}[section]
\newcommand{\beq}{\begin{eqnarray}}
\newcommand{\eeq}{\end{eqnarray}}

\newcommand{\beqs}{\begin{eqnarray*}}
\newcommand{\eeqs}{\end{eqnarray*}}
\begin{document}

\title{Pseudodifferential operators of infinite order in spaces of tempered ultradistributions}
\author{Bojan Prangoski}
\date{}
\maketitle

\begin{abstract}
Specific global symbol classes and corresponding pseudodifferential operators of infinite order that act continuously on the space of tempered ultradistributions of Beurling and Roumieu type are constructed. For these classes, symbolic calculus is developed.
\end{abstract}

\noindent \textbf{Mathematics Subject Classification} 47G30, 46F05\\
\textbf{Keywords} ultradistributions, pseudodifferential operators

\section{Introduction}

Pseudodifferential operators that act continuously on Gevrey classes were vastly studied during the years. A lot of local symbol classes that give rise to such operators (both of finite and infinite order) were constructed by many authors. Also, global symbol classes and corresponding operators (of finite and infinite order), as well as their symbolic calculus were developed in \cite{C1}, \cite{C2}, \cite{C3}, \cite{C4}, \cite{C5}, \cite{C6} (see also \cite{NR}). The functional frame in which those were studied are the Gelfand - Shilov spaces of Roumieu type. The symbol classes developed there are well suited for studying polyhomogeneous operators. In this paper we develop a global calculus for some classes of pseudodifferential operators of infinite order. The functional frame in which the considered symbol classes and the corresponding pseudodifferential operators will be studied is going to be Komatsu ultradistributions, more precisely the spaces of tempered ultradistributions of Beurling and Roumieu type. Our symbol classes are similar to those in \cite{C3} and \cite{C4}, but the weights that control the growth of the derivatives of the symbols are constructed in such way that they give well suited environment for studying Anti-Wick and Weyl operators on the space of tempered ultradistributions. In this paper, we develop calculus for our symbol classes, i.e. we proof results about change of quantization, composition of operators and asymptotic expansion of the symbol of the transposed operator.\\
\indent The paper is organized as follows:\\
\indent 1. \textbf{Preliminaries}. Definition and basic facts are given concerning test spaces and corresponding spaces of ultradistributions. Some facts are cited from \cite{BojanL} which will be needed for the next sections. Also, a kernel theorem is proven for the space of tempered ultradistributions.\\
\indent 2. \textbf{Definition and basic properties of the symbol classes}. The definition of the symbol classes is given as well as their basic topological properties. Pseudodifferential operators $\Op_{\tau}(a)$, arising from $\tau$ - quantization of the symbol $a$ are studied. A theorem that gives the hypocontinuity of the mapping $(a,u)\mapsto \Op_{\tau}(a)u$, for $u$ in the test space, is proven.\\
\indent 3. \textbf{Symbolic calculus}. The space of asymptotic expansion is defined. Results, concerning change of quantization, composition of operators and asymptotic expansion of the symbol of the transposed operator are proven.

\section{Preliminaries}

The sets of natural, integer, positive integer, real and complex numbers are denoted by $\NN$, $\ZZ$, $\ZZ_+$, $\RR$, $\CC$. We use the symbols for $x\in \RR^d$: $\langle x\rangle =(1+|x|^2)^{1/2} $,
$D^{\alpha}= D_1^{\alpha_1}\ldots D_n^{\alpha_d},\quad D_j^
{\alpha_j}={i^{-1}}\partial^{\alpha_j}/{\partial x}^{\alpha_j}$, $\alpha=(\alpha_1,\alpha_2,\ldots,\alpha_d)\in\NN^d$. If $z\in\CC^d$, by $z^2$ we will denote $z^2_1+...+z^2_d$. Note that, if $x\in\RR^d$, $x^2=|x|^2$.\\
\indent Following \cite{Komatsu1}, we denote by $M_{p}$ a sequence of positive numbers $M_0=1$ so that:\\
\indent $(M.1)$ $M_{p}^{2} \leq M_{p-1} M_{p+1}, \; \; p \in\ZZ_+$;\\
\indent $(M.2)$ $\ds M_{p} \leq c_0H^{p} \min_{0\leq q\leq p} \{M_{p-q} M_{q}\}$, $p,q\in \NN$, for some $c_0,H\geq1$;\\
\indent $(M.3)$  $\ds\sum^{\infty}_{p=q+1}   \frac{M_{p-1}}{M_{p}}\leq c_0q \frac{M_{q}}{M_{q+1}}$, $q\in \ZZ_+$,\\
although in some assertions we could assume the weaker ones $(M.2)'$ and $(M.3)'$ (see \cite{Komatsu1}). For a multi-index $\alpha\in\NN^d$, $M_{\alpha}$ will mean $M_{|\alpha|}$, $|\alpha|=\alpha_1+...+\alpha_d$. Recall,  $m_p=M_p/M_{p-1}$, $p\in\ZZ_+$ and the associated function for the sequence $M_{p}$ is defined by
\beqs
M(\rho)=\sup  _{p\in\NN}\log_+   \frac{\rho^{p}}{M_{p}} , \; \; \rho > 0.
\eeqs
It is non-negative, continuous, monotonically increasing function, which vanishes for sufficiently small $\rho>0$ and increases more rapidly then $(\ln \rho)^p$ when $\rho$ tends to infinity, for any $p\in\NN$.\\
\indent Let $U\subseteq\RR^d$ be an open set and $K\subset\subset U$ (we will use always this notation for a compact subset of an open set). Then $\EE^{\{M_p\},h}(K)$ is the space of all $\varphi\in \mathcal{C}^{\infty}(U)$ which satisfy $\ds\sup_{\alpha\in\NN^d}\sup_{x\in K}\frac{|D^{\alpha}\varphi(x)|}{h^{\alpha}M_{\alpha}}<\infty$ and $\DD^{\{M_p\},h}_K$ is the space of all $\varphi\in \mathcal{C}^{\infty}\left(\RR^d\right)$ with supports in $K$, which satisfy $\ds\sup_{\alpha\in\NN^d}\sup_{x\in K}\frac{|D^{\alpha}\varphi(x)|}{h^{\alpha}M_{\alpha}}<\infty$;
$$
\EE^{(M_p)}(U)=\lim_{\substack{\longleftarrow\\ K\subset\subset U}}\lim_{\substack{\longleftarrow\\ h\rightarrow 0}} \EE^{\{M_p\},h}(K),\,\,\,\,
\EE^{\{M_p\}}(U)=\lim_{\substack{\longleftarrow\\ K\subset\subset U}}
\lim_{\substack{\longrightarrow\\ h\rightarrow \infty}} \EE^{\{M_p\},h}(K),
$$
\beqs
\DD^{(M_p)}(U)=\lim_{\substack{\longrightarrow\\ K\subset\subset U}}\lim_{\substack{\longleftarrow\\ h\rightarrow 0}} \DD^{\{M_p\},h}_K,\,\,\,\,
\DD^{\{M_p\}}(U)=\lim_{\substack{\longrightarrow\\ K\subset\subset U}}\lim_{\substack{\longrightarrow\\ h\rightarrow \infty}} \DD^{\{M_p\},h}_K.
\eeqs
The spaces of ultradistributions and ultradistributions with compact support of Beurling and Roumieu type are defined as the strong duals of $\DD^{(M_p)}(U)$ and $\EE^{(M_p)}(U)$, resp. $\DD^{\{M_p\}}(U)$ and $\EE^{\{M_p\}}(U)$. For the properties of these spaces, we refer to \cite{Komatsu1}, \cite{Komatsu2} and \cite{Komatsu3}. In the future we will not emphasize the set $U$ when $U=\RR^d$. Also, the common notation for the symbols $(M_{p})$ and $\{M_{p}\} $ will be *.\\
\indent For $f\in L^{1} $, its Fourier transform is defined by
$(\mathcal{F}f)(\xi ) = \int_{{\RR^d}} e^{-ix\xi}f(x)dx$, $\xi \in {\RR^d}$.

By $\mathfrak{R}$ is denoted a set of positive sequences which monotonically increases to infinity. For $(r_p)\in\mathfrak{R}$, consider the sequence $N_0=1$, $N_p=M_p\prod_{j=1}^{p}r_j$, $p\in\ZZ_+$. One easily sees that this sequence satisfies $(M.1)$ and $(M.3)'$ and its associated function will be denoted by $N_{r_p}(\rho)$, i.e. $\ds N_{r_{p}}(\rho )=\sup_{p\in\NN} \log_+ \frac{\rho^{p }}{M_p\prod_{j=1}^{p}r_j}$, $\rho > 0$. Note, for given $(r_{p})$ and every $k > 0 $ there is $\rho _{0} > 0$ such that $\ds N_{r_{p}} (\rho ) \leq M(k \rho )$, for $\rho > \rho _{0}$. In \cite{BojanL} the following lemmas are proven (for the definition of subordinate function see \cite{Komatsu1}).
\begin{lemma}\label{15}
let $g:[0,\infty)\longrightarrow[0,\infty)$ be an increasing function that satisfies the following estimate: for every $L>0$ there exists $C>0$ such that $g(\rho)\leq M(L\rho)+\ln C$. Then there exists subordinate function $\epsilon(\rho)$ such that $g(\rho)\leq M(\epsilon(\rho))+\ln C'$, for some constant $C'>1$.
\end{lemma}
\begin{lemma}\label{pst}
Let $(k_p)\in\mathfrak{R}$. There exists $(k'_p)\in\mathfrak{R}$ such that $k'_p\leq k_p$ and
\beqs
\prod_{j=1}^{p+q}k'_j\leq 2^{p+q}\prod_{j=1}^{p}k'_j\cdot\prod_{j=1}^{q}k'_j, \mbox{ for all }p,q\in\ZZ_+.
\eeqs
\end{lemma}
\noindent Hence, for every $(k_p)\in\mathfrak{R}$, we can find $(k'_p)\in\mathfrak{R}$, as lemma \ref{pst}, such that $N_{k_p}(\rho)\leq N_{k'_p}(\rho)$, $\rho>0$ and the sequence $N_0=1$, $N_p=M_p\prod_{j=1}^{p}k'_j$, $p\in\ZZ_+$, satisfies $(M.2)$ if $M_p$ does.\\
\indent From now on, we always assume that $M_p$ satisfies $(M.1)$, $(M.2)$ and $(M.3)$. It is said that $P(\xi ) =\sum _{\alpha \in \NN^d}c_{\alpha } \xi^{\alpha}$, $\xi \in \RR^d$, is an ultrapolynomial of the class $(M_{p})$, resp. $\{M_{p}\}$, whenever the coefficients $c_{\alpha }$ satisfy the estimate $|c_{\alpha }|  \leq C L^{|\alpha| }/M_{\alpha}$, $\alpha \in \NN^d$ for some $L > 0$ and $C>0$, resp. for every $L > 0 $ and some $C_{L} > 0$. The corresponding operator $P(D)=\sum_{\alpha} c_{\alpha}D^{\alpha}$ is an ultradifferential operator of the class $(M_{p})$, resp. $\{M_{p}\}$ and they act continuously on $\EE^{(M_p)}(U)$ and $\DD^{(M_p)}(U)$, resp. $\EE^{\{M_p\}}(U)$ and $\DD^{\{M_p\}}(U)$ and the corresponding spaces of ultradistributions. In \cite{BojanL} a special class of ultrapolynomials of class * were constructed. We summarize the results obtained there in the following proposition.
\begin{proposition}\label{orn}
Let $c>0$ and $k>0$, resp. $c>0$ and $(k_p)\in\mathfrak{R}$ are arbitrary but fixed. Then there exist $l>0$ and $q\in\ZZ_+$, resp. there exist $(l_p)\in\mathfrak{R}$ and $q\in\ZZ_+$ such that $\ds P_l(z)=\prod_{j=q}^{\infty}\left(1+\frac{z^2}{l^2 m_j^2}\right)$, resp. $\ds P_{l_p}(z)=\prod_{j=q}^{\infty}\left(1+\frac{z^2}{l_j^2 m_j^2}\right)$, is an entire function that doesn't have zeroes on the strip $W=\RR^d+i\{y\in\RR^d||y_j|\leq c,\,j=1,...,d\}$. $P_l(x)$, resp. $P_{l_p}(x)$, is an ultrapolynomial of class *. Moreover $|P_l(z)|\geq \tilde{C}e^{M(|z|/k)}$, resp. $|P_{l_p}(z)|\geq \tilde{C}e^{N_{k_p}(|z|)}$, $z\in W$, for some $\tilde{C}>0$ and $\ds\left|\partial^{\alpha}_x\frac{1}{P_l(x)}\right|\leq C\cdot\frac{\alpha!}{r^{|\alpha|}}e^{-M\left(|x|/k\right)}$, resp. $\ds\left|\partial^{\alpha}_x\frac{1}{P_{l_p}(x)}\right|\leq C\cdot\frac{\alpha!}{r^{|\alpha|}}e^{-N_{k_p}(|x|)}$, $x\in\RR^d$, $\alpha\in\NN^d$, where $C$ depends on $k$ and $l$, resp. $(k_p)$ and $(l_p)$, and $M_p$; $r\leq c$ arbitrary but fixed.
\end{proposition}
\indent We denote by $\SSS^{M_{p},m}_{2} \left(\RR^d\right)$, $m > 0$, the space of all smooth functions $\varphi$ which satisfy
\beqs
\sigma_{m,2}(\varphi ): = \left( \sum_{\alpha,\beta\in\NN^d} \int_{\RR^d} \left|\frac{m^{|\alpha|+|\beta|}\langle x\rangle^{|\alpha|}D^{\beta}\varphi(x)}{M_{\alpha}M_{\beta}}\right| ^{2} dx \right) ^{1/2}<\infty,
\eeqs
supplied with the topology induced by the norm $\sigma _{m,2}$. The spaces $\SSS'^{(M_{p})}$ and $\SSS'^{\{M_{p}\}}$ of tempered ultradistributions of Beurling and Roumieu type respectively, are defined as the strong duals of the spaces $\ds\SSS^{(M_{p})}=\lim_{\substack{\longleftarrow\\ m\rightarrow\infty}}\SSS^{M_{p},m}_{2}\left(\RR^d\right)$ and $\ds\SSS^{\{M_{p}\}}=\lim_{\substack{\longrightarrow\\ m\rightarrow 0}}\SSS^{M_{p},m}_{2}\left(\RR^d\right)$, respectively. In \cite{PilipovicK} (see also \cite{PilipovicU}) it is proved that the sequence of norms $\sigma_{m,2}$, $m > 0$, is equivalent with the sequences of norms $\|\cdot\|_{m}$, $m > 0$, where $\ds \|\varphi\|_m:=\sup_{\alpha\in \NN^d}\frac{m^{|\alpha|}\| D^{\alpha}\varphi(\cdot) e^{M(m|\cdot|)}\|_{L_{\infty}}}{M_{\alpha }}$. If we denote by $\SSS^{M_p,m}_{\infty}\left(\RR^d\right)$ the space of all infinitely differentiable functions on $\RR^d$ for which the norm $\|\cdot\|_m$ is finite (obviously it is a Banach space), then $\ds\SSS^{(M_p)}\left(\RR^d\right)=\lim_{\substack{\longleftarrow\\ m\rightarrow\infty}} \SSS^{M_p,m}_{\infty}\left(\RR^d\right)$ and $\ds\SSS^{\{M_p\}}\left(\RR^d\right)=\lim_{\substack{\longrightarrow\\ m\rightarrow 0}} \SSS^{M_p,m}_{\infty}\left(\RR^d\right)$. Also, for $m_2>m_1$, the inclusion $\SSS^{M_p,m_2}_{\infty}\left(\RR^d\right)\longrightarrow\SSS^{M_p,m_1}_{\infty}\left(\RR^d\right)$ is a compact mapping. So, $\SSS^*\left(\RR^d\right)$ is a $(FS)$ - space in $(M_p)$ case, resp. a $(DFS)$ - space in the $\{M_p\}$ case. Moreover, they are nuclear spaces. In \cite{PilipovicK} (see also \cite{PilipovicT}) it is proved that $\ds\SSS^{\{M_{p}\}} = \lim_{\substack{\longleftarrow\\ (r_{i}), (s_{j}) \in \mathfrak{R}}}\SSS^{M_{p}}_{(r_{p}),(s_{q})}$, where $\ds\SSS^{M_{p}}_{(r_{p}),(s_{q})}=\left\{\varphi \in \mathcal{C}^{\infty} \left(\RR^d\right)|\|\varphi\|_{(r_{p}),(s_{q})}<\infty\right\}$ and $\ds\|\varphi\|_{(r_{p}),(s_{q})} =\sup_{\alpha\in \NN^d}\frac{\left\|D^{\alpha}\varphi(x)e^{N_{s_p}(|x|)}\right\|_{L^{\infty}}} {M_{\alpha}\prod^{|\alpha|}_{p=1}r_{p}}$. Also, the Fourier transform is a topological automorphism of $\SSS^*$ and of $\SSS'^*$.\\
\indent We need the following kernel theorem for $\SSS'^*$. The $(M_p)$ case was already considered in \cite{LPK} (the authors used the characterization of Fourier-Hermite coefficients of the elements of the space in the proof of the kernel theorem).

\begin{proposition}\label{ktr}
The following isomorphisms of locally convex spaces hold
\beqs
&{}&\SSS^*\left(\RR^{d_1}\right)\hat{\otimes}\SSS^*\left(\RR^{d_2}\right)\cong\SSS^*\left(\RR^{d_1+d_2}\right)\cong \mathcal{L}_b\left(\SSS'^*\left(\RR^{d_1}\right),\SSS^*\left(\RR^{d_2}\right)\right),\\
&{}&\SSS'^*\left(\RR^{d_1}\right)\hat{\otimes}\SSS'^*\left(\RR^{d_2}\right)\cong\SSS'^*\left(\RR^{d_1+d_2}\right)\cong \mathcal{L}_b\left(\SSS^*\left(\RR^{d_1}\right),\SSS'^*\left(\RR^{d_2}\right)\right).
\eeqs
\end{proposition}
\begin{proof} Note that $\SSS^*\left(\RR^{d_1}\right)\otimes\SSS^*\left(\RR^{d_2}\right)$ is dense in $\SSS^*\left(\RR^{d_1+d_2}\right)$. This is true because of the continuous and dense inclusion $\DD^*\left(\RR^{d_1+d_2}\right)\longrightarrow\SSS^*\left(\RR^{d_1+d_2}\right)$ and because $\DD^*\left(\RR^{d_1}\right)\otimes\DD^*\left(\RR^{d_2}\right)$ is dense in $\DD^*\left(\RR^{d_1+d_2}\right)$ (see theorem 2.1 of \cite{Komatsu2}). We need to prove that $\SSS^*\left(\RR^{d_1+d_2}\right)$ induces on $\SSS^*\left(\RR^{d_1}\right)\otimes\SSS^*\left(\RR^{d_2}\right)$ the topology $\pi=\epsilon$ (the $\pi$ and the $\epsilon$ topologies are the same because $\SSS^*$ is nuclear). Because the bilinear mapping $(\varphi,\psi)\mapsto \varphi\otimes\psi$, $\SSS^*\left(\RR^{d_1}\right)\times\SSS^*\left(\RR^{d_1+d_2}\right)\longrightarrow \SSS^*\left(\RR^{d_1+d_2}\right)$ is separately continuous it follows that it is continuous. This is true in the $(M_p)$ case because $\SSS^{(M_p)}$ is $(FS)$-space (hence a $F$ - space) and it is true in the $\{M_p\}$ case because $\SSS^{\{M_p\}}$ is $(DFS)$ - space (hence a barreled $(DF)$ - space). The continuity of this bilinear mapping proves that the inclusion $\SSS^*\left(\RR^{d_1}\right)\otimes_{\pi}\SSS^*\left(\RR^{d_2}\right)\longrightarrow \SSS^*\left(\RR^{d_1+d_2}\right)$ is continuous, hence the topology $\pi$ is stronger than the induced one. Let $A'$ and $B'$ be equicontinuous subsets of $\SSS^*\left(\RR^{d_1}\right)$ and $\SSS^*\left(\RR^{d_2}\right)$, respectively. There exist $h>0$ and $C>0$ such that $\ds \sup_{T\in A'}|\langle T,\varphi\rangle|\leq C\|\varphi\|_h$ and $\ds \sup_{F\in B'}|\langle F,\psi\rangle|\leq C\|\psi\|_h$ in the $(M_p)$ case, resp. there exist $(k_p),(k'_p)\in\mathfrak{R}$ and $C>0$ such that $\ds \sup_{T\in A'}|\langle T,\varphi\rangle|\leq C\|\varphi\|_{(k_p),(k'_p)}$ and $\ds \sup_{F\in B'}|\langle F,\psi\rangle|\leq C\|\psi\|_{(k_p),(k'_p)}$ in the $\{M_p\}$ case. We consider first the $\{M_p\}$ case. By lemma \ref{pst}, without losing generality we can assume that $\prod_{j=1}^{p+q} k_j\leq 2^{p+q}\prod_{j=1}^p k_j \prod_{j=1}^q k_j$, $p\in\ZZ_+$ and the same for $(k'_j)$. Put $r_j=k_j/(2H)$ and $r'_j=k'_j/(2H)$, $j\in\ZZ_+$. For all $T\in A'$ and $F\in B'$, we have
\beqs
|\langle T_x\otimes F_y,\chi(x,y)\rangle|&=&|\langle F_y,\langle T_x,\chi(x,y)\rangle\rangle|\leq C\sup_{y,\beta}\frac{|\langle T_x, D^{\beta}_y\chi(x,y)\rangle|e^{N_{k'_p}(|y|)}}{M_{\beta}\prod_{j=0}^{|\beta|}k_j}\\
&\leq& C^2\sup_{x,y,\alpha,\beta} \frac{\left|D^{\alpha}_xD^{\beta}_y\chi(x,y)\right|e^{N_{k'_p}(|x|)}e^{N_{k'_p}(|y|)}} {M_{\alpha}M_{\beta}\prod_{j=0}^{|\alpha|}k_j\prod_{j=0}^{|\beta|}k_j}\\
&\leq&c_0^2C^2\sup_{x,y,\alpha,\beta}\frac{\left|D^{\alpha}_xD^{\beta}_y\chi(x,y)\right|e^{N_{r'_j}(|(x,y)|)}} {M_{\alpha+\beta}\prod_{j=0}^{|\alpha|+|\beta|}r_j}=c_0^2C^2\|\chi\|_{(r_p),(r'_p)},
\eeqs
where, in the third inequality we used proposition 3.6 of \cite{Komatsu1} for $N_{k'_p}(\lambda)$. Similarly, in the $(M_p)$ case one obtains $\ds \sup_{T\in A',\, F\in B'}|\langle T_x\otimes F_y,\chi(x,y)\rangle|\leq c_0^2C^2\|\chi\|_{hH}$. Hence, the $\epsilon$ topology on $\SSS^*\left(\RR^{d_1}\right)\otimes\SSS^*\left(\RR^{d_2}\right)$ is weaker than the induced one from $\SSS^*\left(\RR^{d_1+d_2}\right)$. This gives the isomorphism $\SSS^*\left(\RR^{d_1}\right)\hat{\otimes}\SSS^*\left(\RR^{d_2}\right)\cong\SSS^*\left(\RR^{d_1+d_2}\right)$. Proposition 50.5 of \cite{Treves} yields the isomorphisms $\SSS^*\left(\RR^{d_1}\right)\hat{\otimes}\SSS^*\left(\RR^{d_2}\right)\cong \mathcal{L}_b\left(\SSS'^*\left(\RR^{d_1}\right),\SSS^*\left(\RR^{d_2}\right)\right)$ and $\SSS'^*\left(\RR^{d_1}\right)\hat{\otimes}\SSS'^*\left(\RR^{d_2}\right)\cong \mathcal{L}_b\left(\SSS^*\left(\RR^{d_1}\right),\SSS'^*\left(\RR^{d_2}\right)\right)$ ($\SSS^*$ is a Montel space). Now, because $\SSS^{(M_p)}$ is $(F)$ - space, theorem 9.9 of \cite{Schaefer} gives the isomorphism $\SSS'^{(M_p)}\left(\RR^{d_1}\right)\hat{\otimes}\SSS'^{(M_p)}\left(\RR^{d_2}\right)\cong \SSS'^{(M_p)}\left(\RR^{d_1+d_2}\right)$. In the $\{M_p\}$ case, $\SSS^{\{M_p\}}$ is $(DFS)$ - space, i.e. the strong dual of the $(FS)$ - space $\SSS'^{\{M_p\}}$, hence this theorem implies the same isomorphism in the $\{M_p\}$ case.
\end{proof}

\section{Definition and basic properties of the symbol classes}

Let $a\in\SSS'^{*}\left(\RR^{2d}\right)$. For $\tau\in\RR$, consider the ultradistribution
\beq\label{3}
K_{\tau}(x,y)=\mathcal{F}^{-1}_{\xi\rightarrow x-y}(a)((1-\tau)x+\tau y,\xi)\in\SSS'^{*}\left(\RR^{2d}\right).
\eeq
Let $\Op_{\tau}(a)$ be the operator from $\SSS^*$ to $\SSS'^{*}$ corresponding to the kernel $K_{\tau}(x,y)$, i.e.
\beq
\langle \Op_{\tau}(a)u,v\rangle=\langle K_{\tau},v\otimes u\rangle,\, u,v\in\SSS^{*}\left(\RR^d\right).
\eeq
$a$ will be called the $\tau$-symbol of the pseudo-differential operator $\Op_{\tau}(a)$. When $\tau=0$, we will denote $\Op_{0}(a)$ by $a(x,D)$. When $a\in\SSS^{*}\left(\RR^{2d}\right)$,
\beq\label{5}
\Op_{\tau}(a)u(x)=\frac{1}{(2\pi)^d}\int_{\RR^{2d}}e^{i(x-y)\xi}a((1-\tau)x+\tau y,\xi)u(y)dyd\xi,
\eeq
where the integral is absolutely convergent.

\begin{proposition}\label{6}
The correspondence $a\mapsto K_{\tau}$ is an isomorphism of $\SSS^{*}\left(\RR^{2d}\right)$, of $\SSS'^{*}\left(\RR^{2d}\right)$ and of $L^2\left(\RR^{2d}\right)$. The inverse map is given by
\beqs
a(x,\xi)=\mathcal{F}_{y\rightarrow\xi}K_{\tau}(x+\tau y,x-(1-\tau)y).
\eeqs
\end{proposition}
\begin{proof} The partial Fourier transform and the composition with the change of variable $\Xi(x,y)=((1-\tau)x+\tau y,x-y)$ are isomorphisms of $\SSS^{*}\left(\RR^{2d}\right)$, of $\SSS'^{*}\left(\RR^{2d}\right)$ and of $L^2\left(\RR^{2d}\right)$. The last part is just an easy computation.
\end{proof}
Operators with symbols in $\SSS^{*}$ correspond to kernels in $\SSS^{*}$ and by proposition \ref{ktr}, those extend to continuous operators from $\SSS'^{*}$ to $\SSS^{*}$. We will call these *-regularizing operators.\\
\indent Now we will define the announced global symbol classes. Let $A_p$ and $B_p$ be sequences that satisfy $(M.1)$, $(M.3)'$ and $A_0=1$ and $B_0=1$. Moreover, let $A_p\subset M_p$ and $B_p\subset M_p$ i.e. there exist $c_0>0$ and $L>0$ such that $A_p\leq c_0 L^pM_p$ and $B_p\leq c_0 L^pM_p$, for all $p\in\NN$ (it is obvious that without losing generality we can assume that this $c_0$ is the same with $c_0$ from the conditions $(M.2)$ and $(M.3)$ for $M_p$). For $0<\rho\leq 1$, define $\Gamma_{A_p,B_p,\rho}^{M_p,\infty}\left(\RR^{2d};h,m\right)$ as the space of all $a\in \mathcal{C}^{\infty}\left(\RR^{2d}\right)$ for which the following norm is finite
\beqs
\sup_{\alpha,\beta}\sup_{(x,\xi)\in\RR^{2d}}\frac{\left|D^{\alpha}_{\xi}D^{\beta}_x a(x,\xi)\right|
\langle (x,\xi)\rangle^{\rho|\alpha|+\rho|\beta|}e^{-M(m|\xi|)}e^{-M(m|x|)}}{h^{|\alpha|+|\beta|}A_{\alpha}B_{\beta}}.
\eeqs
It is easily verified that it is a Banach space. Define
\beqs
\Gamma_{A_p,B_p,\rho}^{(M_p),\infty}\left(\RR^{2d};m\right)=\lim_{\substack{\longleftarrow\\h\rightarrow 0}}
\Gamma_{A_p,B_p,\rho}^{M_p,\infty}\left(\RR^{2d};h,m\right),\,
\Gamma_{A_p,B_p,\rho}^{(M_p),\infty}\left(\RR^{2d}\right)=\lim_{\substack{\longrightarrow\\m\rightarrow\infty}}
\Gamma_{A_p,B_p,\rho}^{(M_p),\infty}\left(\RR^{2d};m\right),\\
\Gamma_{A_p,B_p,\rho}^{\{M_p\},\infty}\left(\RR^{2d};h\right)=\lim_{\substack{\longleftarrow\\m\rightarrow 0}}
\Gamma_{A_p,B_p,\rho}^{M_p,\infty}\left(\RR^{2d};h,m\right),\,
\Gamma_{A_p,B_p,\rho}^{\{M_p\},\infty}\left(\RR^{2d}\right)=\lim_{\substack{\longrightarrow\\h\rightarrow\infty}}
\Gamma_{A_p,B_p,\rho}^{\{M_p\},\infty}\left(\RR^{2d};h\right).
\eeqs

\begin{remark}
$\Gamma_{A_p,B_p,\rho}^{(M_p),\infty}\left(\RR^{2d};m\right)$ and $\Gamma_{A_p,B_p,\rho}^{\{M_p\},\infty}\left(\RR^{2d};h\right)$ are $(F)$ - spaces. Obviously, the inclusion mappings $\Gamma_{A_p,B_p,\rho}^{(M_p),\infty}\left(\RR^{2d};m\right)\longrightarrow\SSS'^{(M_p)}\left(\RR^{2d}\right)$ and $\Gamma_{A_p,B_p,\rho}^{\{M_p\},\infty}\left(\RR^{2d};h\right)\longrightarrow\SSS'^{\{M_p\}}\left(\RR^{2d}\right)$ are continuous, hence $\Gamma_{A_p,B_p,\rho}^{(M_p),\infty}\left(\RR^{2d}\right)$ and $\Gamma_{A_p,B_p,\rho}^{\{M_p\},\infty}\left(\RR^{2d}\right)$ are Hausdorff l.c.s. Moreover, as inductive limits of barreled and bornological l.c.s., they are barreled and bornological.
\end{remark}

\begin{remark} By proposition 7 of \cite{PBD} it follows that every element of $\Gamma_{A_p,B_p,\rho}^{*,\infty}\left(\RR^{2d}\right)$ is a multiplier for $\SSS'^*\left(\RR^{2d}\right)$.
\end{remark}

\begin{remark} Examples of nontrivial elements of $\Gamma_{A_p,B_p,\rho}^{*,\infty}\left(\RR^{2d}\right)$ are given by every ultrapolynomial of class *.
\end{remark}

\begin{proposition}\label{10}
For every $a\in\Gamma_{A_p,B_p,\rho}^{*,\infty}\left(\RR^{2d}\right)$ there exists a sequence $\chi_j$, $j\in\ZZ_+$, in $\DD^{*}\left(\RR^{2d}\right)$ such that $\chi_j\longrightarrow a$ in $\Gamma_{A_p,B_p,\rho}^{*,\infty}\left(\RR^{2d}\right)$.
\end{proposition}
\begin{proof} Let $\varphi(x)\in\DD^{(B_p)}\left(\RR^{d}\right)$ and $\psi(\xi)\in\DD^{(A_p)}\left(\RR^{d}\right)$, in the $(M_p)$ case, resp. $\varphi(x)\in\DD^{\{B_p\}}\left(\RR^{d}\right)$ and $\psi(\xi)\in\DD^{\{A_p\}}\left(\RR^{d}\right)$ in the $\{M_p\}$ case, are such that $0\leq\varphi,\psi\leq 1$, $\varphi(x)=1$ when $|x|\leq 1/4$, $\psi(\xi)=1$ when $|\xi|\leq 1/4$ and $\varphi(x)=0$ when $|x|\geq 1/2$, $\psi(\xi)=0$ when $|\xi|\geq 1/2$ (such functions exist because $A_p$ and $B_p$ satisfy $(M.3)'$). Put $\chi(x,\xi)=\varphi(x)\psi(\xi)$, $\chi_n(x,\xi)=\chi(x/n,\xi/n)$ for $n\in\ZZ_+$. Then $\chi,\chi_n\in\DD^{(M_p)}\left(\RR^{2d}\right)$, resp. $\chi,\chi_n\in\DD^{\{M_p\}}\left(\RR^{2d}\right)$. For $a\in \Gamma_{A_p,B_p,\rho}^{*,\infty}\left(\RR^{2d}\right)$, it is readily seen that $a_n(x,\xi)=\chi_n(x,\xi)a(x,\xi)$ is an element of $\DD^{*}\left(\RR^{2d}\right)$. It is easy to prove that there exists $m>0$ such that for every $h>0$, $a_n\longrightarrow a$ in $\Gamma_{A_p,B_p,\rho}^{M_p,\infty}\left(\RR^{2d};h,m\right)$ in the $(M_p)$ case, resp. there exists $h>0$ such that for every $m>0$, $a_n\longrightarrow a$ in $\Gamma_{A_p,B_p,\rho}^{M_p,\infty}\left(\RR^{2d};h,m\right)$ in the $\{M_p\}$ case.
\end{proof}

\begin{theorem}\label{17}
Let $a\in\Gamma_{A_p,B_p,\rho}^{*,\infty}\left(\RR^{2d}\right)$. Then the integral (\ref{5}) is well defined as an iterated integral. The ultradistribution $\Op_{\tau}(a)u$, $u\in\SSS^*$, coincides with the function defined by that iterated integral.
\end{theorem}
\begin{proof} The $(M_p)$ case. Because $a\in\Gamma_{A_p,B_p,\rho}^{(M_p),\infty}\left(\RR^{2d}\right)$, there exists $m>0$ such that, for every $h>0$ there exists $C_1>0$ such that
\beqs
\left|D^{\alpha}_{\xi}D^{\beta}_x a(x,\xi)\right|\leq C_1\frac{h^{|\alpha|+|\beta|}A_{\alpha}B_{\beta}e^{M(m|\xi|)}e^{M(m|x|)}}
{\langle (x,\xi)\rangle^{\rho|\alpha|+\rho|\beta|}},\, \forall \alpha,\beta\in\NN^d,\, \forall (x,\xi)\in\RR^{2d}.
\eeqs
Hence, $\ds b(x,\xi)=\int_{\RR^d}e^{i(x-y)\xi}a((1-\tau)x+\tau y,\xi)u(y)dy$ is well defined and $b\in \mathcal{C}^{\infty}\left(\RR^{2d}\right)$. Choose $m_0>0$ large enough such that, for all $m'\geq m_0$, $\ds\int_{\RR^d}e^{M(2m|\tau y|)}e^{-M(m'|y|)}dy<\infty$. Because $u\in\SSS^{(M_p)}$, for such $m'$ we get $\ds \sup_{\alpha\in\NN^d}\frac{m'^{|\alpha|}\left\|D^{\alpha}u(y)e^{M(m'|y|)}\right\|_{L^{\infty}}}{M_{\alpha}}<\infty$. One obtains\\
$\ds \left|\xi^{\alpha}b(x,\xi)\right|$
\beqs
&=&\left|\int_{\RR^d}e^{i(x-y)\xi}D^{\alpha}_y\left(a((1-\tau)x+\tau y,\xi)u(y)\right)dy\right|\\
&\leq&\sum_{\gamma\leq\alpha}{\alpha\choose\gamma}\int_{\RR^d}|\tau|^{|\gamma|}
\left|D^{\gamma}_xa((1-\tau)x+\tau y,\xi)\right|\left|D^{\alpha-\gamma}u(y)\right|dy\\
&\leq&C\sum_{\gamma\leq\alpha}{\alpha\choose\gamma}\int_{\RR^d}|\tau|^{|\gamma|}
h^{|\gamma|}B_{\gamma}e^{M(m|\xi|)}e^{M(m|(1-\tau)x+\tau y|)}
\frac{M_{\alpha-\gamma}e^{-M(m'|y|)}}{m'^{|\alpha|-|\gamma|}}dy\\
&\leq&C'\sum_{\gamma\leq\alpha}{\alpha\choose\gamma}\int_{\RR^d}(|\tau|hL)^{|\gamma|} M_{\gamma}e^{M(m|\xi|)}e^{M(2m|(1-\tau)x|)}e^{M(2m|\tau y|)}
\frac{M_{\alpha-\gamma}e^{-M(m'|y|)}}{m'^{|\alpha|-|\gamma|}}dy\\
&\leq&C''e^{M(m|\xi|)}e^{M(2m|(1-\tau)x|)}M_{\alpha}\left(|\tau|hL+\frac{1}{m'}\right)^{|\alpha|},
\eeqs
where we used $B_p\subset M_p$. For $l>0$ consider $P_l(\xi)$. By proposition \ref{orn}, we can choose $l$ such that $|P_l(\xi)|\geq c''e^{M(r|\xi|)}$ where $r>0$ is such that $\ds \int_{\RR^d}e^{M(m|\xi|)}e^{-M(r|\xi|)}d\xi<\infty$ and $P_l(\xi)$ is never zero. Also, if we represent $\ds P_l(\xi)=\sum_{\alpha}c_{\alpha}\xi^{\alpha}$, there exists $L'>0$ and $C'>0$ such that $|c_{\alpha}|\leq C'L'^{|\alpha|}/M_{\alpha}$. Choose $h>0$ so small and $m'\geq m_0$ so large such that $\ds \left(|\tau|hL+\frac{1}{m'}\right)L'<\frac{1}{4}$. Then, we have
\beqs
\left|P_l(\xi)b(x,\xi)\right|&\leq&\sum_{\alpha}|c_{\alpha}|\left|\xi^{\alpha}b(x,\xi)\right|\leq C''e^{M(m|\xi|)}e^{M(2m|(1-\tau)x|)}\sum_{\alpha}|c_{\alpha}|M_{\alpha}
\left(|\tau|hL+\frac{1}{m'}\right)^{|\alpha|}\\
&\leq& C_0e^{M(m|\xi|)}e^{M(2m|(1-\tau)x|)}.
\eeqs
Hence $\ds \int_{\RR^d}|b(x,\xi)|d\xi$ is finite for every $x$, i.e. (\ref{5}) is well defined as iterated integral. From this estimate also follows that $b(x,\xi)v(x)\in L^1\left(\RR^{2d}\right)$, for any $v\in\SSS^{(M_p)}$.\\
\indent Let us consider the $\{M_p\}$ case. Because $a\in\Gamma_{A_p,B_p,\rho}^{\{M_p\},\infty}\left(\RR^{2d}\right)$, there exists $h>0$ such that, for every $m>0$ there exists $C_1>0$ such that
\beqs
\left|D^{\alpha}_{\xi}D^{\beta}_x a(x,\xi)\right|\leq C_1\frac{h^{|\alpha|+|\beta|}A_{\alpha}B_{\beta}e^{M(m|\xi|)}e^{M(m|x|)}}
{\langle (x,\xi)\rangle^{\rho|\alpha|+\rho|\beta|}},\, \forall \alpha,\beta\in\NN^d,\, \forall (x,\xi)\in\RR^{2d}.
\eeqs
Hence, $\ds b(x,\xi)=\int_{\RR^d}e^{i(x-y)\xi}a((1-\tau)x+\tau y,\xi)u(y)dy$ is well defined and $b\in \mathcal{C}^{\infty}\left(\RR^{2d}\right)$. Put
\beqs
g(\lambda)=\sup_{|(x,\xi)|\leq\lambda}\sup_{\alpha,\beta}\ln_+\frac{\left|D^{\alpha}_{\xi}D^{\beta}_x a(x,\xi)\right|
\langle(x,\xi)\rangle^{\rho|\alpha|+\rho|\beta|}}{h^{|\alpha|+|\beta|}A_{\alpha}B_{\beta}}.
\eeqs
$g$ is an increasing function and by proposition 3.6 of \cite{Komatsu1}, it satisfies the condition of lemma \ref{15}. Hence, there exists subordinate function $\epsilon(\lambda)$ and a constant $C'>1$ such that $g(\lambda)\leq M(\epsilon(\lambda))+\ln C'$. We get that
\beqs
\left|D^{\alpha}_{\xi}D^{\beta}_x a(x,\xi)\right|\leq C'\frac{h^{|\alpha|+|\beta|}A_{\alpha}B_{\beta}e^{M(\epsilon(|(x,\xi)|))}}
{\langle(x,\xi)\rangle^{\rho|\alpha|+\rho|\beta|}},\, \forall \alpha,\beta\in\NN^d,\, \forall (x,\xi)\in\RR^{2d}.
\eeqs
By lemma 3.12 of \cite{Komatsu1}, there exist another sequence $\tilde{N}_p$, which satisfies $(M.1)$, such that $\tilde{N}(\lambda)\geq M(\epsilon(\lambda))$ and $k'_p=\tilde{n}_p/m_p\longrightarrow\infty$ when $p\longrightarrow\infty$. There exist $(k''_p)\in\mathfrak{R}$ such that $k''_p\leq k'_p$, for $p\in\ZZ_+$. Then
\beqs
e^{N_{k''_p}(\lambda)}=\sup_p\frac{\lambda^p}{M_p\prod_{j=1}^p k''_j}\geq \sup_p\frac{\lambda^p}{M_p\prod_{j=1}^p k'_j}=e^{\tilde{N}(\lambda)}\geq e^{M(\epsilon(\lambda))}.
\eeqs
From this, we obtain the estimate
\beq
\left|D^{\alpha}_{\xi}D^{\beta}_x a(x,\xi)\right|\leq C\frac{h^{|\alpha|+|\beta|}A_{\alpha}B_{\beta}e^{N_{k_p}(|\xi|)}e^{N_{k_p}(|x|)}}{\langle (x,\xi)\rangle^{\rho|\alpha|+\rho|\beta|}},\, \forall \alpha,\beta\in\NN^d,\, \forall (x,\xi)\in\RR^{2d},
\eeq
where we choose $(k_p)\in\mathfrak{R}$ such that $e^{N_{k''_p}(|(x,\xi)|)}\leq c' e^{N_{k_p}(|\xi|)}e^{N_{k_p}(|x|)}$, for some $c'>0$. Because $u\in\SSS^{\{M_p\}}$, there exists $h_1>0$ such that for every $(s_p)\in\mathfrak{R}$, $\ds \sup_{\alpha}\frac{h_1^{|\alpha|}\left\|e^{N_{s_p}(|x|)}D^{\alpha}u(x)\right\|_{L^{\infty}}}{M_{\alpha}}<\infty$. Choose $(s_p)\in\mathfrak{R}$, such that $\ds\int_{\RR^d}e^{N_{k_p/2}(|\tau y|)}e^{-N_{s_p}(|y|)}dy<\infty$. Then, we have\\
$\ds \left|\xi^{\alpha}b(x,\xi)\right|$
\beqs
&=&\left|\int_{\RR^d}e^{i(x-y)\xi}D^{\alpha}_y\left(a((1-\tau)x+\tau y,\xi)u(y)\right)dy\right|\\
&\leq&\sum_{\gamma\leq\alpha}{\alpha\choose\gamma}\int_{\RR^d}|\tau|^{|\gamma|}
\left|D^{\gamma}_xa((1-\tau)x+\tau y,\xi)\right|\left|D^{\alpha-\gamma}u(y)\right|dy\\
&\leq&C'\sum_{\gamma\leq\alpha}{\alpha\choose\gamma}\int_{\RR^d}|\tau|^{|\gamma|}
h^{|\gamma|}B_{\gamma}e^{N_{k_p}(|\xi|)}e^{N_{k_p}(|(1-\tau)x+\tau y|)}
\frac{e^{-N_{s_p}(|y|)}M_{\alpha-\gamma}}{h_1^{|\alpha|-|\gamma|}}dy\\
&\leq&C''\sum_{\gamma\leq\alpha}{\alpha\choose\gamma}\int_{\RR^d}
\frac{(|\tau|hL)^{|\gamma|}M_{\gamma}e^{N_{k_p}(|\xi|)}e^{N_{k_p/2}(|(1-\tau)x|)}e^{N_{k_p/2}(|\tau y|)}e^{-N_{s_p}(|y|)}M_{\alpha-\gamma}}{h_1^{|\alpha|-|\gamma|}}dy\\
&\leq&C\sum_{\gamma\leq\alpha}{\alpha\choose\gamma}
\frac{(|\tau|hL)^{|\gamma|}e^{N_{k_p}(|\xi|)}e^{N_{k_p/2}(|(1-\tau)x|)}M_{\alpha}}
{h_1^{|\alpha|-|\gamma|}}\\
&=&Ce^{N_{k_p}(|\xi|)}e^{N_{k_p/2}(|(1-\tau)x|)}M_{\alpha}
\left(|\tau|hL+\frac{1}{h_1}\right)^{|\alpha|},
\eeqs
where we used $B_p\subset M_p$. For $(l_p)\in\mathfrak{R}$ consider $P_{l_p}(\xi)$. By proposition \ref{orn} we can choose $(l_p)\in\mathfrak{R}$ such that $|P_{l_p}(\xi)|\geq c''e^{N_{r_p}(|\xi|)}$ where $(r_p)\in \mathfrak{R}$ is such that $\ds \int_{\RR^d}e^{N_{k_p}(|\xi|)}e^{-N_{r_p}(|\xi|)}d\xi<\infty$ and $P_{l_p}(\xi)$ is never zero. Also, if we represent $\ds P_{l_p}(\xi)=\sum_{\alpha}c_{\alpha}\xi^{\alpha}$, then for any $L'>0$ there exists $C'>0$ such that $|c_{\alpha}|\leq C'L'^{|\alpha|}/M_{\alpha}$. Choose $L'>0$ such that, $\ds \left(|\tau|hL+\frac{1}{h_1}\right)L'<\frac{1}{4}$. By the above estimate, we have
\beqs
\left|P_{l_p}(\xi)b(x,\xi)\right|&\leq&\sum_{\alpha}|c_{\alpha}|\left|\xi^{\alpha}b(x,\xi)\right|\leq Ce^{N_{k_p}(|\xi|)}e^{N_{k_p/2}(|(1-\tau)x|)}\sum_{\alpha}|c_{\alpha}|M_{\alpha}
\left(|\tau|hL+\frac{1}{h_1}\right)^{|\alpha|}\\
&\leq& C_0e^{N_{k_p}(|\xi|)}e^{N_{k_p/2}(|(1-\tau)x|)}.
\eeqs
Hence $\ds \int_{\RR^d}|b(x,\xi)|d\xi$ is finite for every $x$, i.e. (\ref{5}) is well defined as iterated integral. From this estimate also follows that $b(x,\xi)v(x)\in L^1\left(\RR^{2d}\right)$, for any $v\in\SSS^{\{M_p\}}$.\\
\indent Hence, in both cases we get that $\ds \int_{\RR^d}|b(x,\xi)|d\xi$ is finite for every $x$, i.e. (\ref{5}) is well defined as iterated integral, and $b(x,\xi)v(x)\in L^1\left(\RR^{2d}\right)$, for any $v\in\SSS^*$. We will temporary denote $\ds F(x)=\frac{1}{(2\pi)^d}\int_{\RR^d}b(x,\xi)d\xi$. From the above estimates it is obvious that $F\in\SSS'^*$. By Fubini's theorem, we have
\beqs
\langle F,v\rangle=\frac{1}{(2\pi)^d}\int_{\RR^d}\int_{\RR^d}\int_{\RR^d}e^{i(x-y)\xi}a((1-\tau)x+\tau y,\xi)u(y)v(x)dydxd\xi.
\eeqs
By the growth condition of $a$, it is obvious that the integral
\beqs
\int_{\RR^{2d}}e^{i(x-y)\xi}a((1-\tau)x+\tau y,\xi)u(y)v(x)dydx
\eeqs
converges, so if we put the change of variable $\Xi(x,y)=((1-\tau)x+\tau y,x-y)$ in the last term of the above equality we obtain $\langle F,v\rangle=\left\langle a,\mathcal{F}^{-1}_2\left((v\otimes u)\circ\Xi^{-1}\right)\right\rangle=\langle \Op_{\tau}(a)u,v\rangle$, which completes the proof of the theorem.
\end{proof}

We will define more general classes of operators and symbols.

\begin{definition}
Denote by $\Pi_{A_p,B_p,\rho}^{M_p,\infty}\left(\RR^{3d};h,m\right)$ the Banach space of all $a\in C^{\infty}\left(\RR^{3d}\right)$ with the norm
\beqs
\sup_{\alpha,\beta,\gamma\in\NN^d}\sup_{(x,y,\xi)\in\RR^{3d}}
\frac{\left|D^{\alpha}_{\xi}D^{\beta}_xD^{\gamma}_y a(x,y,\xi)\right|\langle (x,y,\xi)\rangle^{\rho|\alpha|+\rho|\beta|+\rho|\gamma|}
e^{-M(m|\xi|)}e^{-M(m|x|)}e^{-M(m|y|)}}{h^{|\alpha|+|\beta|+|\gamma|}\langle x-y\rangle^{\rho|\alpha|+\rho|\beta|+\rho|\gamma|}A_{\alpha}B_{\beta+\gamma}}
\eeqs
\end{definition}

We will denote this norm by $\|\cdot\|_{h,m,\Pi}$. Define
\beqs
\Pi_{A_p,B_p,\rho}^{(M_p),\infty}\left(\RR^{3d};m\right)=\lim_{\substack{\longleftarrow\\h\rightarrow 0}}
\Pi_{A_p,B_p,\rho}^{M_p,\infty}\left(\RR^{3d};h,m\right),\,
\Pi_{A_p,B_p,\rho}^{(M_p),\infty}\left(\RR^{3d}\right)=\lim_{\substack{\longrightarrow\\m\rightarrow\infty}}
\Pi_{A_p,B_p,\rho}^{(M_p),\infty}\left(\RR^{3d};m\right),\\
\Pi_{A_p,B_p,\rho}^{\{M_p\},\infty}\left(\RR^{3d};h\right)=\lim_{\substack{\longleftarrow\\m\rightarrow 0}}
\Pi_{A_p,B_p,\rho}^{M_p,\infty}\left(\RR^{3d};h,m\right),\,
\Pi_{A_p,B_p,\rho}^{\{M_p\},\infty}\left(\RR^{3d}\right)=\lim_{\substack{\longrightarrow\\h\rightarrow\infty}}
\Pi_{A_p,B_p,\rho}^{\{M_p\},\infty}\left(\RR^{3d};h\right).
\eeqs
$\Pi_{A_p,B_p,\rho}^{(M_p),\infty}\left(\RR^{3d};m\right)$ and $\Pi_{A_p,B_p,\rho}^{\{M_p\},\infty}\left(\RR^{3d};h\right)$ are $(F)$ - spaces. Similarly as for $\Gamma_{A_p,B_p,\rho}^{*,\infty}\left(\RR^{2d}\right)$, one proves that $\Pi_{A_p,B_p,\rho}^{*,\infty}\left(\RR^{3d}\right)$ are barreled and bornological l.c.s.\\
\indent One easily sees that, for $a\in \Pi_{A_p,B_p,\rho}^{*,\infty}\left(\RR^{3d}\right)$, the function $b(x,\xi)=a(x,x,\xi)$ belongs to $\Gamma_{A_p,B_p,\rho}^{*,\infty}\left(\RR^{2d}\right)$. Moreover, if $p\in \Gamma_{A_p,B_p,\rho}^{*,\infty}\left(\RR^{2d}\right)$ and $\tau\in\RR$, then $a(x,y,\xi)=p((1-\tau)x+\tau y,\xi)$ belongs to $\Pi_{A_p,B_p,\rho}^{*,\infty}\left(\RR^{3d}\right)$.

\begin{remark} The $\Gamma$ and $\Pi$ classes defined here are appropriate generalization (for symbols of infinite order) in ultradistributional setting of the corresponding classes in the setting of Schwartz distributions (see \cite{SchwartzK} for the theory od distributions and \cite{Shubin} for the corresponding $\Gamma$ and $\Pi$ symbol classes and calculus in the setting of Schwartz distributions).
\end{remark}

\begin{lemma}\label{nn1}
Let $h>0$ be fixed. For every bounded set $B$ in $\Pi_{A_p,B_p,\rho}^{\{M_p\},\infty}\left(\RR^{3d};h\right)$, there exist $C>0$ and $(k_p)\in\mathfrak{R}$ such that, for all $a\in B$,
\beqs
\sup_{\alpha,\beta,\gamma\in\NN^d}\sup_{(x,y,\xi)\in\RR^{3d}}
\frac{\left|D^{\alpha}_{\xi}D^{\beta}_xD^{\gamma}_y a(x,y,\xi)\right|\langle (x,y,\xi)\rangle^{\rho|\alpha|+\rho|\beta|+\rho|\gamma|}
e^{-N_{k_p}(|\xi|)}e^{-N_{k_p}(|x|)}e^{-N_{k_p}(|y|)}}{h^{|\alpha|+|\beta|+|\gamma|}\langle x-y\rangle^{\rho|\alpha|+\rho|\beta|+\rho|\gamma|}A_{\alpha}B_{\beta+\gamma}}\leq C.
\eeqs
\end{lemma}
\begin{proof} Because $B$ is bounded, for every $m>0$ there exists a constant $C_m>0$ (which depends on $m$) such that, for every $a\in B$, $\|a\|_{h,m,\Pi}\leq C_m$, i.e.
\beqs
\frac{\left|D^{\alpha}_{\xi}D^{\beta}_xD^{\gamma}_y a(x,y,\xi)\right|\langle (x,y,\xi)\rangle^{\rho|\alpha|+\rho|\beta|+\rho|\gamma|}}
{h^{|\alpha|+|\beta|+|\gamma|}\langle x-y\rangle^{\rho|\alpha|+\rho|\beta|+\rho|\gamma|}A_{\alpha}B_{\beta+\gamma}}\leq C_m e^{M(m|\xi|)}e^{M(m|x|)}e^{M(m|y|)},
\eeqs
for all $(x,y,\xi)\in\RR^{3d}$ and all $\alpha,\beta,\gamma\in\NN^d$. Without losing generality, we can take $C_m\geq 1$. Put
\beqs
g_a(x,y,\xi)=\sup_{\alpha,\beta,\gamma}\ln_+\left(\frac{\left|D^{\alpha}_{\xi}D^{\beta}_xD^{\gamma}_y a(x,y,\xi)\right|\langle (x,y,\xi)\rangle^{\rho|\alpha|+\rho|\beta|+\rho|\gamma|}}
{h^{|\alpha|+|\beta|+|\gamma|}\langle x-y\rangle^{\rho|\alpha|+\rho|\beta|+\rho|\gamma|}A_{\alpha}B_{\beta+\gamma}}\right).
\eeqs
Then, by proposition 3.6 of \cite{Komatsu1}, we have
\beqs
g_a(x,y,\xi)&\leq& M(m|\xi|)+M(m|x|)+M(m|y|)+\ln C_m\leq 3M(m|(x,y,\xi)|)+\ln C_m\\
&\leq& M(mH^2|(x,y,\xi)|)+\ln(c_0^2C_m).
\eeqs
Now, define $\ds\tilde{g}_a(\lambda)=\sup_{|(x,y,\xi)|\leq \lambda}g_a(x,y,\xi)$. Then $\ds\tilde{g}_a(\lambda)\leq M(mH^2\lambda)+\ln(c_0^2C_m)$, for $\lambda\geq 0$ and $a\in B$. Then, if we put $\ds \tilde{g}(\lambda)=
\sup_{a\in B}\tilde{g}_a(\lambda)$, we have $\ds \tilde{g}(\lambda)\leq M(mH^2\lambda)+\ln(c_0^2C_m)$, for $\lambda\geq 0$. $\ds \tilde{g}_a(\lambda)$ is an increasing function of $\lambda$ for every $a\in B$, hence $\ds \tilde{g}(\lambda)$ is an increasing function of $\lambda$. So $\tilde{g}$ satisfies the conditions in lemma \ref{15}. Hence, there exist subordinate function $\epsilon(\lambda)$ and a constant $C'>1$ such that $\tilde{g}(\lambda)\leq M(\epsilon(\lambda))+\ln C'$. We get that
\beqs
\ln_+\left(\frac{\left|D^{\alpha}_{\xi}D^{\beta}_xD^{\gamma}_y a(x,y,\xi)\right|\langle (x,y,\xi)\rangle^{\rho|\alpha|+\rho|\beta|+\rho|\gamma|}}
{h^{|\alpha|+|\beta|+|\gamma|}\langle x-y\rangle^{\rho|\alpha|+\rho|\beta|+\rho|\gamma|}A_{\alpha}B_{\beta+\gamma}}\right)\leq \tilde{g}\left(|(x,y,\xi)|\right)\leq M\left(\epsilon\left(|(x,y,\xi)|\right)\right)+\ln C',
\eeqs
for all $(x,y,\xi)\in\RR^{3d}$, $\alpha,\beta,\gamma\in\NN^d$ and $a\in B$, i.e.
\beqs
\frac{\left|D^{\alpha}_{\xi}D^{\beta}_xD^{\gamma}_y a(x,y,\xi)\right|\langle (x,y,\xi)\rangle^{\rho|\alpha|+\rho|\beta|+\rho|\gamma|}}
{h^{|\alpha|+|\beta|+|\gamma|}\langle x-y\rangle^{\rho|\alpha|+\rho|\beta|+\rho|\gamma|}A_{\alpha}B_{\beta+\gamma}}\leq C' e^{M\left(\epsilon\left(|(x,y,\xi)|\right)\right)},
\eeqs
for all $(x,y,\xi)\in\RR^{3d}$, $\alpha,\beta,\gamma\in\NN^d$ and $a\in B$. By lemma 3.12 of \cite{Komatsu1}, there exists another sequence $\tilde{N}_p$, which satisfies $(M.1)$, such that $\tilde{N}(\lambda)\geq M(\epsilon(\lambda))$ and $k'_p=\tilde{n}_p/m_p\longrightarrow\infty$ when $p\longrightarrow\infty$. There exists $(k''_p)\in\mathfrak{R}$ such that $k''_p\leq k'_p$, for $p\in\ZZ_+$. Then
\beqs
e^{N_{k''_p}(\lambda)}=\sup_p\frac{\lambda^p}{M_p\prod_{j=1}^p k''_j}\geq \sup_p\frac{\lambda^p}{M_p\prod_{j=1}^p k'_j}=\sup_p\frac{\lambda^p \tilde{N}_0}{\tilde{N}_p}=e^{\tilde{N}(\lambda)}\geq e^{M(\epsilon(\lambda))}.
\eeqs
From this, we obtain the estimate
\beqs
\frac{\left|D^{\alpha}_{\xi}D^{\beta}_xD^{\gamma}_y a(x,y,\xi)\right|\langle (x,y,\xi)\rangle^{\rho|\alpha|+\rho|\beta|+\rho|\gamma|}}
{h^{|\alpha|+|\beta|+|\gamma|}\langle x-y\rangle^{\rho|\alpha|+\rho|\beta|+\rho|\gamma|}A_{\alpha}B_{\beta+\gamma}}\leq C e^{N_{k_p}\left(|\xi|\right)}e^{N_{k_p}\left(|x|\right)}e^{N_{k_p}\left(|y|\right)},
\eeqs
for all $(x,y,\xi)\in\RR^{3d}$, $\alpha,\beta,\gamma\in\NN^d$ and $a\in B$, where we choose $(k_p)\in\mathfrak{R}$ such that $e^{N_{k''_p}(|(x,y,\xi)|)}\leq c' e^{N_{k_p}(|\xi|)}e^{N_{k_p}(|x|)}e^{N_{k_p}(|y|)}$, for some constant $c'>0$.
\end{proof}

\begin{lemma}\label{45}
Let $a\in \Pi_{A_p,B_p,\rho}^{*,\infty}\left(\RR^{3d}\right)$. For $\delta>0$ and $u,\chi\in\SSS^*\left(\RR^d\right)$, such that $\chi(0)=1$, define
\beqs
I_{\chi,\delta}(x)=\frac{1}{(2\pi)^d}\int_{\RR^{2d}}e^{i(x-y)\xi}a(x,y,\xi)\chi(\delta\xi)u(y)dyd\xi.
\eeqs
Then $I_{\chi,\delta}(x)$ has a limit when $\delta\longrightarrow 0^+$ and the limit doesn't depend on $\chi$. Moreover, the limit function is continuous and has ultrapolynomial growth of class *.
\end{lemma}
\begin{proof} The $(M_p)$ case. Let $a\in \Pi_{A_p,B_p,\rho}^{(M_p),\infty}\left(\RR^{3d};m\right)$. For $l>0$ consider $P_l(\xi)$. By proposition \ref{orn}, $P_l(\xi)$ is never zero and we can choose $P_l(\xi)$ such that, $\ds \left|P_l(\xi)\right|\geq c_1 e^{M(r|\xi|)}$, where $r>0$ is such that $\ds\int_{\RR^d}e^{M(m|\xi|)}e^{-M(r|\xi|)}d\xi<\infty$. Also, we have the estimate $\ds \left|D^{\alpha}_{\xi}\frac{1}{P_l(\xi)}\right|\leq c'_1\frac{\alpha!}{d_1^{|\alpha|}}$, for some $c'_1>0$ and $d_1>0$. On the other hand if we represent $P_l(\xi)=\sum_{\alpha}c_{\alpha}\xi^{\alpha}$ then there exist $L_0>0$ and $C_0>0$ such that $\left|c_{\alpha}\right|\leq C_0L_0^{|\alpha|}/M_{\alpha}$. Observe that
\beqs
e^{i(x-y)\xi}=\frac{1}{P_l(y-x)}P_l(-D_\xi)\left(\frac{1}{P_l(\xi)}P_l(-D_y)e^{i(x-y)\xi}\right).
\eeqs
Then we have
\beq\label{50}
I_{\chi,\delta}(x)=\frac{1}{(2\pi)^d}\int_{\RR^{2d}}e^{i(x-y)\xi}
\frac{1}{P_l(\xi)}P_l(D_y)\left(\frac{1}{P_l(y-x)}P_l(D_\xi)\left(a(x,y,\xi)\chi(\delta\xi)u(y)\right)\right)dyd\xi.
\eeq
Because $u,\chi\in\SSS^{(M_p)}\left(\RR^d\right)$, for every $s>0$
\beqs
\sup_{\alpha\in\NN^d}\frac{s^{|\alpha|}\left\|e^{M(s|\xi|)}D^{\alpha}\chi(\xi)\right\|_{L^{\infty}}}{M_{\alpha}}<\infty,
\quad \sup_{\alpha\in\NN^d}\frac{s^{|\alpha|}\left\|e^{M(s|y|)}D^{\alpha}u(y)\right\|_{L^{\infty}}}{M_{\alpha}}<\infty.
\eeqs
Now, we estimate as follows\\
$\ds \left|P_l(D_y)\left(\frac{1}{P_l(y-x)}P_l(D_\xi)\left(a(x,y,\xi)\chi(\delta\xi)u(y)\right)\right)\right|$
\beqs
&\leq&\sum_{\alpha,\gamma}\left|c_{\alpha}\right|\left|c_{\gamma}\right|\sum_{\substack{\alpha'\leq\alpha\\ \gamma'\leq\gamma}}\sum_{\gamma''\leq\gamma'}{\alpha\choose\alpha'}{\gamma\choose\gamma'}{\gamma'\choose\gamma''}\\
&{}&\hspace{20 pt}\cdot\left|D^{\alpha'}_{\xi}D^{\gamma''}_ya(x,y,\xi)\right|\left| D^{\gamma'-\gamma''}_y\frac{1}{P_l(y-x)}\right|\delta^{|\alpha|-|\alpha'|}
\left|D^{\alpha-\alpha'}_{\xi}\chi(\delta\xi)\right|\left|D^{\gamma-\gamma'}_yu(y)\right|\\
&\leq&C'\sum_{\alpha,\gamma}\left|c_{\alpha}\right|\left|c_{\gamma}\right|\sum_{\substack{\alpha'\leq\alpha\\ \gamma'\leq\gamma}}\sum_{\gamma''\leq\gamma'}{\alpha\choose\alpha'}{\gamma\choose\gamma'}{\gamma'\choose\gamma''}
\frac{(\gamma'-\gamma'')!}{d_1^{|\gamma'|-|\gamma''|}}\\
&{}&\cdot\frac{h^{|\alpha'|+|\gamma''|}\langle x-y\rangle^{\rho|\alpha'|+\rho|\gamma''|}A_{\alpha'}B_{\gamma''}
e^{M(m|\xi|)}e^{M(m|x|)}e^{M(m|y|)}}{\langle(x,y,\xi)\rangle^{\rho|\alpha'|+\rho|\gamma''|}}\delta^{|\alpha|-|\alpha'|}
\frac{M_{\alpha-\alpha'}M_{\gamma-\gamma'}e^{-M(s\delta|\xi|)}}
{s^{|\alpha|-|\alpha'|+|\gamma|-|\gamma'|}e^{M(s|y|)}}\\
&\leq&C''\sum_{\alpha,\gamma}\frac{L_0^{|\alpha|+|\gamma|}}{M_{\alpha}M_{\gamma}}\sum_{\substack{\alpha'\leq\alpha\\ \gamma'\leq\gamma}}\sum_{\gamma''\leq\gamma'}{\alpha\choose\alpha'}{\gamma\choose\gamma'}{\gamma'\choose\gamma''}
\frac{(\gamma'-\gamma'')!(4L_0)^{|\gamma'|-|\gamma''|}}{d_1^{|\gamma'|-|\gamma''|}(4L_0)^{|\gamma'|-|\gamma''|}}\\
&{}&\cdot\frac{(2Lh)^{|\alpha'|+|\gamma'|}
e^{M(m|\xi|)}e^{M(m|x|)}e^{M(m|y|)}}{(2Lh)^{|\gamma'|-|\gamma''|}M_{\gamma'-\gamma''}}\delta^{|\alpha|-|\alpha'|}
\frac{M_{\alpha}M_{\gamma}}
{s^{|\alpha|-|\alpha'|+|\gamma|-|\gamma'|}e^{M(s|y|)}}\\
&\leq&C_1\frac{e^{M(m|\xi|)}e^{M(m|x|)}e^{M(m|y|)}}{e^{M(s|y|)}}\\
&{}&\cdot\sum_{\alpha,\gamma}
\left(\frac{\delta L_0}{s}\right)^{|\alpha|}\left(\frac{L_0}{s}\right)^{|\gamma|}\sum_{\substack{\alpha'\leq\alpha\\ \gamma'\leq\gamma}}{\alpha\choose\alpha'}{\gamma\choose\gamma'}
\left(\frac{2sLh}{\delta}\right)^{|\alpha'|}(2sLh)^{|\gamma'|}\sum_{\gamma''\leq\gamma'}{\gamma'\choose\gamma''}
\frac{1}{(8L_0Lh)^{|\gamma'|-|\gamma''|}}\\
&=&C_1\frac{e^{M(m|\xi|)}e^{M(m|x|)}e^{M(m|y|)}}{e^{M(s|y|)}}
\sum_{\alpha,\gamma}\left(\frac{\delta L_0}{s}+2LL_0h\right)^{|\alpha|}\left(\frac{L_0}{s}+ 2L_0Lh+\frac{1}{4}\right)^{|\gamma|}.
\eeqs
Choose $h$ such that $LL_0h<1/8$ and then choose $s$ such that the above sum converge for $\delta=1$ and denote its value by $C_2$ (then, obviously, for $0<\delta<1$ the sum is not greater than $C_2$). Moreover, choose $s$ large enough, such that $\ds \int_{\RR^d}e^{M(m|y|)}e^{-M(s|y|)}dy<\infty$. Hence
\beqs
\left|I_{\chi,\delta}(x)\right|&\leq&\frac{C_1C_2}{(2\pi)^d}\int_{\RR^{2d}}
\frac{1}{\left|P_l(\xi)\right|}\frac{e^{M(m|\xi|)}e^{M(m|x|)}e^{M(m|y|)}}{e^{M(s|y|)}}dyd\xi\\
&\leq&Ce^{M(m|x|)}\int_{\RR^d}
\frac{e^{M(m|\xi|)}}{e^{M(r|\xi|)}}d\xi\cdot \int_{\RR^d}\frac{e^{M(m|y|)}}{e^{M(s|y|)}}dy,
\eeqs
which is finite for every $x$. Note that $a(x,y,\xi)\chi(\delta\xi)u(y)\longrightarrow a(x,y,\xi)u(y)$ in $\EE^{(M_p)}\left(\RR^{2d}_{y,\xi}\right)$ for each fixed $x$ when $\delta\longrightarrow 0^+$, so $\ds\frac{1}{P_l(\xi)}P_l(D_y)\left(\frac{1}{P_l(y-x)}P_l(D_\xi)\left(a(x,y,\xi)\chi(\delta\xi)u(y)\right)\right)$ tends to $\ds\frac{1}{P_l(\xi)}P_l(D_y)\left(\frac{1}{P_l(y-x)}P_l(D_\xi)a(x,y,\xi)u(y)\right)$ in $\EE^{(M_p)}\left(\RR^{2d}_{y,\xi}\right)$ for each fixed $x$, when $\delta\longrightarrow 0^+$. If we take the limit in (\ref{50}) as $\delta\longrightarrow 0^+$, from dominated convergence, it follows that
\beqs
\lim_{\delta\rightarrow 0^+}I_{\chi,\delta}(x)=\frac{1}{(2\pi)^d}\int_{\RR^{2d}}e^{i(x-y)\xi}
\frac{1}{P_l(\xi)}P_l(D_y)\left(\frac{1}{P_l(y-x)}P_l(D_\xi)\left(a(x,y,\xi)u(y)\right)\right)dyd\xi.
\eeqs
Moreover, by similar estimates as above, one proves that the function in the last integral can be dominated by $\ds Ce^{M(m|x|)}\frac{e^{M(m|\xi|)}}{e^{M(r|\xi|)}}\cdot \frac{e^{M(m|y|)}}{e^{M(s|y|)}}$. Thus, $\ds\lim_{\delta\rightarrow 0^+}I_{\chi,\delta}(x)$ is a continuous function with $(M_p)$ - ultrapolynomial growth. Note that the choice of $P_l$ does not depend on $\chi$ and $u$, only on $m$ such that $a\in\Pi_{A_p,B_p,\rho}^{(M_p),\infty}\left(\RR^{3d};m\right)$. Hence, one can choose the same $P_l$ for all $a\in\Pi_{A_p,B_p,\rho}^{(M_p),\infty}\left(\RR^{3d};m\right)$. From this, the claim in the lemma follows.\\
\indent The $\{M_p\}$ case. Let $a\in \Pi_{A_p,B_p}^{\{M_p\},\infty}\left(\RR^{3d};h\right)$. By lemma \ref{nn1} there exists $(k_p)\in\mathfrak{R}$ such that
\beq\label{zon}
\left|D^{\alpha}_{\xi}D^{\beta}_xD^{\gamma}_y a(x,y,\xi)\right|\leq C_0\frac{h^{|\alpha|+|\beta|+|\gamma|}
\langle x-y\rangle^{\rho|\alpha|+\rho|\beta|+\rho|\gamma|}A_{\alpha}B_{\beta+\gamma}
e^{N_{k_p}(|\xi|)}e^{N_{k_p}(|x|)}e^{N_{k_p}(|y|)}}{\langle (x,y,\xi)\rangle^{\rho|\alpha|+\rho|\beta|+\rho|\gamma|}},
\eeq
for all $\alpha,\beta,\gamma\in\NN^d$ and $(x,y,\xi)\in\RR^{3d}$. For $(l_p)\in\mathfrak{R}$ consider $P_{l_p}(\xi)$. By proposition \ref{orn}, we can choose $P_{l_p}(\xi)$ such that, $\ds \left|P_{l_p}(\xi)\right|\geq c'' e^{N_{r_p}(|\xi|)}$, where $(r_p)\in\mathfrak{R}$ is such that $\ds \int_{\RR^d}e^{N_{k_p}(|\xi|)}e^{-N_{r_p}(|\xi|)}d\xi<\infty$. On the other hand, if we represent $\ds P_{l_p}(\xi)=\sum_{\alpha}c_{\alpha}\xi^{\alpha}$, then for every $L'>0$ there exists $C'>0$ such that $\left|c_{\alpha}\right|\leq C'L'^{|\alpha|}/M_{\alpha}$. Also, we have the same estimates, as in the $(M_p)$ case, for the derivatives of $1/P_{l_p}(\xi)$, i.e $\ds \left|D^{\alpha}_{\xi}\frac{1}{P_{l_p}(\xi)}\right|\leq c'_1\frac{\alpha!}{d_1^{|\alpha|}}$, for some $c'_1>0$ and $d_1>0$. Because $u,\chi\in\SSS^{\{M_p\}}\left(\RR^d\right)$, there exists $s>0$, such that
\beqs
\sup_{\alpha\in\NN^d}\frac{s^{|\alpha|}\left\|e^{M(s|\xi|)}D^{\alpha}\chi(\xi)\right\|_{L^{\infty}}}{M_{\alpha}}
<\infty,\quad \sup_{\alpha\in\NN^d}\frac{s^{|\alpha|}\left\|e^{M(s|y|)}D^{\alpha}u(y)\right\|_{L^{\infty}}}{M_{\alpha}}<\infty.
\eeqs
(We can choose $s$ to be the same for $u$ and $\chi$). Similarly as for the $(M_p)$ case, one obtains (\ref{50}), but with $P_{l_p}$ in place of $P_l$ and obtains the estimate\\
$\ds \left|P_{l_p}(D_y)\left(\frac{1}{P_{l_p}(y-x)}P_{l_p}(D_\xi)\left(a(x,y,\xi)\chi(\delta\xi)u(y)\right)\right)\right|$
\beqs
\leq C_1\frac{e^{N_{k_p}(|\xi|)}e^{N_{k_p}(|x|)}e^{N_{k_p}(|y|)}}{e^{M(s|y|)}}
\sum_{\alpha,\gamma}\left(\frac{\delta L'}{s}+2LL'h\right)^{|\alpha|}\left(\frac{L'}{s}+ 2L'Lh+\frac{1}{4}\right)^{|\gamma|}.
\eeqs
Choose $L'$, small enough, such that the above sum converges for $\delta=1$ and denote its value by $C_2$. Similarly as above, we obtain the estimate
\beqs
\left|I_{\chi,\delta}(x)\right|&\leq&Ce^{N_{k_p}(|x|)}\int_{\RR^d}
e^{-N_{r_p}(|\xi|)}e^{N_{k_p}(|\xi|)}d\xi\cdot \int_{\RR^d}e^{N_{k_p}(|y|)}e^{-M(s|y|)}dy.
\eeqs
The first integral converges by the choice of $(r_p)$ and the convergence of the second can be easily proven. By similar arguments as in the $(M_p)$ case and dominated convergence, the claim of the lemma follows. Note that the choice of $P_{l_p}$ does not depend on $u$ and $\chi$, only on $a$.
\end{proof}

By the lemma, $\ds\lim_{\delta\rightarrow 0^+}I_{\chi,\delta}(x)$ is an element of $\SSS'^*\left(\RR^d\right)$. For every $a\in\Pi_{A_p,B_p,\rho}^{*,\infty}\left(\RR^{3d}\right)$ define the operator $A:\SSS^*\left(\RR^d\right)\longrightarrow\SSS'^*\left(\RR^d\right)$, $Au(x)=\ds\lim_{\delta\rightarrow 0^+}I_{\chi,\delta}(x)$. By the proof of the above lemma we obtain that
\beqs
Au(x)=\frac{1}{(2\pi)^d}\int_{\RR^{2d}}e^{i(x-y)\xi}
\frac{1}{P_l(\xi)}P_l(D_y)\left(\frac{1}{P_l(y-x)}P_l(D_\xi)a(x,y,\xi)u(y)\right)dyd\xi,
\eeqs
for the $(M_p)$ case, respectively
\beqs
Au(x)=\frac{1}{(2\pi)^d}\int_{\RR^{2d}}e^{i(x-y)\xi}
\frac{1}{P_{l_p}(\xi)}P_{l_p}(D_y)\left(\frac{1}{P_{l_p}(y-x)}P_{l_p}(D_\xi)a(x,y,\xi)u(y)\right)dyd\xi,
\eeqs
for the $\{M_p\}$ case and moreover, the choice of $P_l$ in the $(M_p)$ case, respectively $P_{l_p}$ in the $\{M_p\}$ case does not depend on $u\in\SSS^*$. If $P_{l'}$, resp. $P_{l'_p}$, is another operator such that $\left|P_{l'}(\xi)\right|\geq c_1 e^{M(r|\xi|)}$, resp. $\left|P_{l'_p}(\xi)\right|\geq c'' e^{N_{r_p}(|\xi|)}$, where $\ds\int_{\RR^d}e^{M(m|\xi|)}e^{-M(r|\xi|)}d\xi<\infty$, resp. $\ds \int_{\RR^d}e^{N_{k_p}(|\xi|)}e^{-N_{r_p}(|\xi|)}d\xi<\infty$, then $Au(x)$ can be given in the above form with $P_{l'}$ in place of $P_l$, resp. $P_{l'_p}$ in place of $P_{l_p}$. To prove the continuity of $A$, put
\beqs
K(x,y,\xi)=e^{i(x-y)\xi}
\frac{1}{P_l(\xi)}P_l(D_y)\left(\frac{1}{P_l(y-x)}P_l(D_\xi)a(x,y,\xi)u(y)\right)
\eeqs
in the $(M_p)$ case, resp., the same but with $P_{l_p}$ in place of $P_l$, in the $\{M_p\}$ case. For $v\in\SSS^*$,
\beqs
\langle Au(x),v(x)\rangle=\frac{1}{(2\pi)^d}\int_{\RR^{3d}}K(x,y,\xi)v(x)dyd\xi dx.
\eeqs
Let $v\in \SSS^*$ be fixed. If $u\in B$, where $B$ is a bounded set in $\SSS^*$, similarly as in the proof of the above lemma, one can prove that $\langle Au(x),v(x)\rangle$ is bounded when $u\in B$. Hence the set $A(B)$ is simply bounded in $\SSS'^*$, consequently it is strongly bounded. Because $\SSS^*$ is bornological and $A$ maps bounded sets into bounded sets it must be continuous.\\

\begin{theorem}\label{npr}
The mapping $(a,u)\mapsto Au$, $\Pi_{A_p,B_p,\rho}^{*,\infty}\left(\RR^{3d}\right)\times
\SSS^*\left(\RR^d\right)\longrightarrow\SSS^*\left(\RR^d\right)$, is hypocontinuous.
\end{theorem}
\begin{proof} Because $\Pi_{A_p,B_p,\rho}^{*,\infty}\left(\RR^{3d}\right)$ and $\SSS^*\left(\RR^d\right)$ are barreled it is enough to prove that the mapping is separately continuous. We will consider first the $(M_p)$ case. It is enough to prove that, for every $m>0$, the mapping $\Pi_{A_p,B_p,\rho}^{(M_p),\infty}\left(\RR^{3d};m\right)\times
\SSS^{(M_p)}\left(\RR^d\right)\longrightarrow\SSS^{(M_p)}\left(\RR^d\right)$ is separately continuous. We will prove that it is continuous i.e. that for every $s>0$, there exists a constant $C>0$ and $h>0$, $t>0$ such that $\|Au\|_s\leq C \|a\|_{h,m,\Pi}\|u\|_t$, where $\ds\|\phi\|_s=\sup_{\alpha}\frac{s^{|\alpha|}\left\|D^{\alpha}\phi(\cdot)e^{M(s|\cdot|)}\right\|_{L^{\infty}}}
{M_{\alpha}}$ are the seminorms in $\SSS^{(M_p)}\left(\RR^d\right)$. Let $s>0$. Obviously, without losing generality, we can assume that $s\geq 1$. Choose $P_l(\xi)$ as in the proof of the above lemma and represent $Au$ in the form
\beqs
Au(x)=\frac{1}{(2\pi)^d}\int_{\RR^{2d}}e^{i(x-y)\xi}
\frac{1}{P_l(\xi)}P_l(D_y)\left(\frac{1}{P_l(y-x)}P_l(D_\xi)a(x,y,\xi)u(y)\right)dyd\xi.
\eeqs
In the proof of the above lemma we proved that $P_l$ can be chosen the same for all $a\in\Pi_{A_p,B_p,\rho}^{(M_p),\infty}\left(\RR^{3d};m\right)$ (it depends only on $m$). By proposition \ref{orn} $P_l(\xi)$ is never zero and we can choose $l$ small enough such that $\ds \left|D^{\alpha}_{\xi}\frac{1}{P_l(\xi)}\right|\leq c'_1\frac{\alpha!}{d_1^{|\alpha|}}e^{-M(r|\xi|)}$, for some $c'_1>0$ and $d_1>0$, where $r>0$ is such that $\ds \int_{\RR^d}\frac{e^{M(m|\xi|)}e^{M(2s|\xi|)}}{e^{M(r|\xi|)}}d\xi$ converges and $e^{M(\frac{r}{2}|x|)}\geq \tilde{C} e^{M(s|x|)}e^{M(m|x|)}$. On the other hand, if we represent $\ds P_l(\xi)=\sum_{\alpha}c_{\alpha}\xi^{\alpha}$, there exist $L_0\geq 1$ and $C_0>0$ such that $\left|c_{\alpha}\right|\leq C_0L_0^{|\alpha|}/M_{\alpha}$. Then, for $a\in\Pi_{A_p,B_p,\rho}^{(M_p),\infty}\left(\RR^{3d};m\right)$ and $u\in\SSS^{(M_p)}$, we have\\
$\ds\left|D^{\beta}_x\left(e^{i(x-y)\xi}\frac{1}{P_l(\xi)}P_l(D_y)
\left(\frac{1}{P_l(y-x)}P_l(D_\xi)a(x,y,\xi)u(y)\right)\right)\right|$
\beqs
&\leq&\sum_{\beta'\leq\beta}\sum_{\alpha,\gamma}\sum_{\gamma'\leq\gamma}\sum_{\gamma''\leq\gamma'}
\sum_{\beta''\leq\beta'}{\beta\choose\beta'}{\beta'\choose\beta''}{\gamma\choose\gamma'}{\gamma'\choose\gamma''}
\left|c_{\alpha}\right|\left|c_{\gamma}\right|\frac{|\xi|^{|\beta|-|\beta'|}}{\left|P_l(\xi)\right|}\\
&{}&\hspace{30 pt}\cdot\left|D^{\gamma'-\gamma''}_yD^{\beta'-\beta''}_x\left(\frac{1}{P_l(y-x)}\right)\right|
\left|D^{\alpha}_{\xi}D^{\beta''}_xD^{\gamma''}_y a(x,y,\xi)\right|\left|D^{\gamma-\gamma'}_y u(y)\right|\\
&\leq&C_1\|a\|_{h,m,\Pi}\|u\|_t
\sum_{\beta'\leq\beta}\sum_{\alpha,\gamma}\sum_{\gamma'\leq\gamma}\sum_{\gamma''\leq\gamma'}
\sum_{\beta''\leq\beta'}{\beta\choose\beta'}{\beta'\choose\beta''}{\gamma\choose\gamma'}{\gamma'\choose\gamma''}
\frac{L_0^{|\alpha|+|\gamma|}}{M_{\alpha}M_{\gamma}}\frac{|\xi|^{|\beta|-|\beta'|}}{e^{M(r|\xi|)}}\\
&{}&\hspace{30 pt}\cdot\frac{(\beta'-\beta''+\gamma'-\gamma'')!e^{-M(r|x-y|)}}
{d_1^{|\beta'|-|\beta''|+|\gamma'|-|\gamma''|}}\cdot\frac{M_{\gamma-\gamma'}e^{-M(t|y|)}}
{t^{|\gamma|-|\gamma'|}}\\
&{}&\hspace{30 pt}\cdot\frac{h^{|\alpha|+|\beta''|+|\gamma''|}\langle x-y\rangle^{\rho|\alpha|+\rho|\beta''|+\rho|\gamma''|}A_{\alpha}B_{\beta''+\gamma''}
e^{M(m|\xi|)}e^{M(m|x|)}e^{M(m|y|)}}{\langle (x,y,\xi)\rangle^{\rho|\alpha|+\rho|\beta''|+\rho|\gamma''|}}\\
&\leq&C_2\|a\|_{h,m,\Pi}\|u\|_t
\sum_{\beta'\leq\beta}\sum_{\alpha,\gamma}\sum_{\gamma'\leq\gamma}\sum_{\gamma''\leq\gamma'}
\sum_{\beta''\leq\beta'}{\beta\choose\beta'}{\beta'\choose\beta''}{\gamma\choose\gamma'}{\gamma'\choose\gamma''}
L_0^{|\alpha|+|\gamma|}\frac{|\xi|^{|\beta|-|\beta'|}}{e^{M(r|\xi|)}}\\
&{}&\hspace{30 pt}\cdot\frac{(\beta'-\beta''+\gamma'-\gamma'')!e^{-M(r|x-y|)}}
{d_1^{|\beta'|-|\beta''|+|\gamma'|-|\gamma''|}}\cdot\frac{e^{-M(t|y|)}}
{t^{|\gamma|-|\gamma'|}M_{\gamma'-\gamma''}}\\
&{}&\hspace{30 pt}\cdot(2hL)^{|\alpha|+|\beta''|+|\gamma''|}H^{|\beta''|+|\gamma''|}M_{\beta''}
e^{M(m|\xi|)}e^{M(m|x|)}e^{M(m|y|)}\\
&\leq&C_2\|a\|_{h,m,\Pi}\|u\|_t\frac{M_{\beta}e^{M(m|\xi|)}e^{M(m|x|)}e^{M(m|y|)}}
{e^{M(r|\xi|)}e^{M(t|y|)}e^{M(r|x-y|)}}\\
&{}&\cdot\sum_{\beta'\leq\beta}\sum_{\alpha,\gamma}\sum_{\gamma'\leq\gamma}\sum_{\gamma''\leq\gamma'}
\sum_{\beta''\leq\beta'}{\beta\choose\beta'}{\beta'\choose\beta''}{\gamma\choose\gamma'}{\gamma'\choose\gamma''}
L_0^{|\alpha|+|\gamma|}\frac{|\xi|^{|\beta|-|\beta'|}}{M_{\beta-\beta'}}\cdot
\frac{(2s)^{|\beta|-|\beta'|}}{(2s)^{|\beta|-|\beta'|}}\\
&{}&\hspace{30 pt}\cdot\frac{(\beta'-\beta''+\gamma'-\gamma'')!}
{d_1^{|\beta'|-|\beta''|+|\gamma'|-|\gamma''|}}\cdot\frac{(2hL)^{|\alpha|+|\beta''|+|\gamma''|}H^{|\beta''|+|\gamma''|}}
{t^{|\gamma|-|\gamma'|}M_{\beta'-\beta''}M_{\gamma'-\gamma''}}\\
&\leq&C_3\|a\|_{h,m,\Pi}\|u\|_t\frac{M_{\beta}e^{M(m|\xi|)}e^{M(2s|\xi|)}e^{M(m|x|)}e^{M(m|y|)}}
{e^{M(r|\xi|)}e^{M(t|y|)}e^{M(r|x-y|)}}\\
&{}&\cdot\sum_{\beta'\leq\beta}\sum_{\alpha,\gamma}\sum_{\gamma'\leq\gamma}\sum_{\gamma''\leq\gamma'}
\sum_{\beta''\leq\beta'}{\beta\choose\beta'}{\beta'\choose\beta''}{\gamma\choose\gamma'}{\gamma'\choose\gamma''}
(2hLL_0)^{|\alpha|}L_0^{|\gamma|}\\
&{}&\hspace{30 pt}\cdot\frac{(\beta'-\beta''+\gamma'-\gamma'')!(4sHL_0)^{|\beta'|-|\beta''|+|\gamma'|-|\gamma''|}}
{d_1^{|\beta'|-|\beta''|+|\gamma'|-|\gamma''|}M_{\beta'-\beta''+\gamma'-\gamma''}
(4sHL_0)^{|\beta'|-|\beta''|+|\gamma'|-|\gamma''|}}
\cdot\frac{(2hL)^{|\beta''|+|\gamma''|}H^{|\beta'|+|\gamma'|}}{t^{|\gamma|-|\gamma'|}(2s)^{|\beta|-|\beta'|}}\\
&\leq&C_4\|a\|_{h,m,\Pi}\|u\|_t\frac{M_{\beta}e^{M(m|\xi|)}e^{M(2s|\xi|)}e^{M(m|x|)}e^{M(m|y|)}}
{e^{M(r|\xi|)}e^{M(t|y|)}e^{M(r|x-y|)}}\\
&{}&\cdot\sum_{\beta'\leq\beta}\sum_{\alpha,\gamma}\sum_{\gamma'\leq\gamma}\sum_{\gamma''\leq\gamma'}
\sum_{\beta''\leq\beta'}{\beta\choose\beta'}{\beta'\choose\beta''}{\gamma\choose\gamma'}{\gamma'\choose\gamma''}
(2hLL_0)^{|\alpha|}L_0^{|\gamma|}\\
&{}&\hspace{30 pt}\cdot\frac{(2hL)^{|\beta''|+|\gamma''|}H^{|\beta'|+|\gamma'|}}
{(4sHL_0)^{|\beta'|-|\beta''|+|\gamma'|-|\gamma''|}t^{|\gamma|-|\gamma'|}(2s)^{|\beta|-|\beta'|}}\\
&=&C_4\|a\|_{h,m,\Pi}\|u\|_t\frac{M_{\beta}e^{M(m|\xi|)}e^{M(2s|\xi|)}e^{M(m|x|)}e^{M(m|y|)}}
{e^{M(r|\xi|)}e^{M(t|y|)}e^{M(r|x-y|)}}\\
&{}&\cdot\left(\frac{1}{2s}+2hLH+\frac{1}{4sL_0}\right)^{|\beta|}
\sum_{\alpha,\gamma}(2hLL_0)^{|\alpha|}\left(\frac{L_0}{t}+2hLL_0H+\frac{1}{4s}\right)^{|\gamma|}.
\eeqs
Note that $e^{M(\frac{r}{2}|x|)}\leq C_6 e^{M(r|x-y|)}e^{M(r|y|)}$. For the chosen $r$ we choose $t$ such that the integral $\ds\int_{\RR^d}\frac{e^{M(m|y|)}e^{M(r|y|)}}{e^{M(t|y|)}}dy$ converges and moreover, we take $h$ small enough and $t$ large enough such that the above sum converges. Moreover, choose $h$ small enough such that $\ds\frac{1}{2s}+2hLH+\frac{1}{4sL_0}\leq \frac{1}{s}$. Then for the derivatives of $Au$ we obtain
\beqs
\left|D^{\beta}_x Au(x)\right|\leq C\|a\|_{h,m,\Pi}\|u\|_t e^{-M(s|x|)}\frac{M_{\beta}}{s^{|\beta|}},
\eeqs
which is the desired estimate.\\
\indent Now we will consider the $\{M_p\}$ case. Note that it is enough to prove that, for every $h>0$, $\Pi_{A_p,B_p,\rho}^{\{M_p\},\infty}\left(\RR^{3d};h\right)\times
\SSS^{\{M_p\}}\left(\RR^d\right)\longrightarrow\SSS^{\{M_p\}}\left(\RR^d\right)$ is separately continuous. Because $\Pi_{A_p,B_p,\rho}^{\{M_p\},\infty}\left(\RR^{3d};h\right)$ and $\SSS^{\{M_p\}}\left(\RR^d\right)$ are bornological it is enough to prove that this mapping maps products of bounded sets into bounded sets in $\SSS^{\{M_p\}}\left(\RR^d\right)$. Let $B_1$ and $B_2$ be bounded sets in $\Pi_{A_p,B_p,\rho}^{\{M_p\},\infty}\left(\RR^{3d};h\right)$, respectively in $\SSS^{\{M_p\}}\left(\RR^d\right)$. Then, by lemma \ref{nn1}, there exist $\tilde{C}_1>0$ and $(k_p)\in\mathfrak{R}$ such that
\beq\label{tt1}
\frac{\left|D^{\alpha}_{\xi}D^{\beta}_xD^{\gamma}_y a(x,y,\xi)\right|\langle (x,y,\xi)\rangle^{\rho|\alpha|+\rho|\beta|+\rho|\gamma|}
e^{-N_{k_p}(|\xi|)}e^{-N_{k_p}(|x|)}e^{-N_{k_p}(|y|)}}{h^{|\alpha|+|\beta|+|\gamma|}\langle x-y\rangle^{\rho|\alpha|+\rho|\beta|+\rho|\gamma|}A_{\alpha}B_{\beta+\gamma}}\leq \tilde{C}_1,
\eeq
for all $a\in B_1$, $(x,y,\xi)\in\RR^{3d}$ and $\alpha,\beta,\gamma\in\NN^d$. We know that $\ds \SSS^{\{M_p\}}\left(\RR^d\right)=\lim_{\substack{\longrightarrow\\s\rightarrow 0}}
\SSS^{M_p,s}_{\infty}\left(\RR^d\right)$, where $\SSS^{M_p,s}_{\infty}\left(\RR^d\right)$ is the $(B)$ - space with the norm $\ds\|\phi\|_s=\sup_{\alpha}\frac{s^{|\alpha|}\left\|D^{\alpha}\phi(\cdot)e^{M(s|\cdot|)}\right\|_{L^{\infty}}}
{M_{\alpha}}$ and $\SSS^{\{M_p\}}\left(\RR^d\right)$ is a $(DFS)$ - space generated by this inductive limit (the linking mappings are compact inclusions). So, there exists $t>0$ such that $B_2\subseteq \SSS^{M_p,t}_{\infty}\left(\RR^d\right)$ and it is bounded there. Hence, there exists $\tilde{C}_2>0$ such that $\ds\|u\|_t\leq \tilde{C}_2$, for all $u\in B_2$. On the other hand, we know that $\ds \SSS^{\{M_p\}}\left(\RR^d\right)=\lim_{\substack{\longleftarrow\\(s_p),(s'_p)\in\mathfrak{R}}}
\SSS^{M_p}_{(s_p),(s'_p)}\left(\RR^d\right)$, where $\SSS^{M_p}_{(s_p),(s'_p)}\left(\RR^d\right)$ is the $(B)$ - space with the norm $\ds \|\phi\|_{(s_p),(s'_p)}=\sup_{\alpha}\frac{\left\|D^{\alpha}\phi(\cdot)e^{N_{s'_p}(|\cdot|)}\right\|_{L^{\infty}}}
{M_{\alpha}\prod_{j=1}^{|\alpha|}s_j}$. So, it is enough to prove that, for arbitrary $(s_p),(s'_p)\in\mathfrak{R}$, $\|Au\|_{(s_p),(s'_p)}$ is bounded for all $a\in B_1$ and $u\in B_2$. So, let $(s_p),(s'_p)\in\mathfrak{R}$ be fixed. Represent $Au$ in the form
\beqs
Au(x)=\frac{1}{(2\pi)^d}\int_{\RR^{2d}}e^{i(x-y)\xi}
\frac{1}{P_{l_p}(\xi)}P_{l_p}(D_y)\left(\frac{1}{P_{l_p}(y-x)}P_{l_p}(D_\xi)a(x,y,\xi)u(y)\right)dyd\xi.
\eeqs
In the proof of lemma \ref{45} we proved that the choice of $P_{l_p}$ depends only on $(k_p)$ such that (\ref{tt1}) holds. But $(k_p)$ is the same for all $a\in B_1$ hence we can choose $P_{l_p}$ the same for all $a\in B_1$. By proposition \ref{orn}, $P_{l_p}(\xi)$ is never zero and we can choose $(l_p)\in\mathfrak{R}$ such that, $\ds \left|D^{\alpha}_{\xi}\frac{1}{P_{l_p}(\xi)}\right|\leq c'_1\frac{\alpha!}{d_1^{|\alpha|}}e^{-N_{r_p}(|\xi|)}$, for some $c'_1>0$ and $d_1>0$, where $(r_p)\in\mathfrak{R}$ is chosen such that $\ds \int_{\RR^d}\frac{e^{N_{k_p}(|\xi|)}e^{N_{s_p}(|\xi|)}}
{e^{N_{r_p}(|\xi|)}}d\xi$ converges and $e^{N_{2r_p}(|x|)}\geq \tilde{C} e^{N_{s'_p}(|x|)}e^{N_{k_p}(|x|)}$ (see also the remarks after the proof of lemma \ref{45}). On the other hand, if we represent $\ds P_{l_p}(\xi)=\sum_{\alpha}c_{\alpha}\xi^{\alpha}$, then for every $L'>0$ there exists $C'>0$ such that $\left|c_{\alpha}\right|\leq C'L'^{|\alpha|}/M_{\alpha}$. For $a\in B_1$ and $u\in B_2$, similarly as in the $(M_p)$ case, one obtains the estimate\\
$\ds\left|D^{\beta}_x\left(e^{i(x-y)\xi}\frac{1}{P_{l_p}(\xi)}P_{l_p}(D_y)
\left(\frac{1}{P_{l_p}(y-x)}P_{l_p}(D_\xi)a(x,y,\xi)u(y)\right)\right)\right|$
\beqs
&\leq&C_2\tilde{C}_1\tilde{C}_2\frac{M_{\beta}e^{N_{k_p}(|\xi|)}e^{N_{k_p}(|x|)}e^{N_{k_p}(|y|)}}
{e^{N_{r_p}(|\xi|)}e^{M(t|y|)}e^{N_{r_p}(|x-y|)}}\\
&{}&\cdot\sum_{\beta'\leq\beta}\sum_{\alpha,\gamma}\sum_{\gamma'\leq\gamma}\sum_{\gamma''\leq\gamma'}
\sum_{\beta''\leq\beta'}{\beta\choose\beta'}{\beta'\choose\beta''}{\gamma\choose\gamma'}{\gamma'\choose\gamma''}
L'^{|\alpha|+|\gamma|}\frac{|\xi|^{|\beta|-|\beta'|}}{M_{\beta-\beta'}}\cdot
\prod_{j=1}^{|\beta|-|\beta'|}\frac{1}{s_j}\cdot\prod_{j=1}^{|\beta|-|\beta'|}s_j\\
&{}&\cdot\frac{(\beta'-\beta''+\gamma'-\gamma'')!}
{d_1^{|\beta'|-|\beta''|+|\gamma'|-|\gamma''|}}\cdot\frac{(2hL)^{|\alpha|+|\beta''|+|\gamma''|}H^{|\beta''|+|\gamma''|}}
{t^{|\gamma|-|\gamma'|}M_{\beta'-\beta''}M_{\gamma'-\gamma''}}\\
&\leq&C_3\tilde{C}_1\tilde{C}_2\frac{M_{\beta}e^{N_{k_p}(|\xi|)}e^{N_{s_p}(|\xi|)}e^{N_{k_p}(|x|)}e^{N_{k_p}(|y|)}}
{e^{N_{r_p}(|\xi|)}e^{M(t|y|)}e^{N_{r_p}(|x-y|)}}\\
&{}&\cdot\sum_{\beta'\leq\beta}\sum_{\alpha,\gamma}\sum_{\gamma'\leq\gamma}\sum_{\gamma''\leq\gamma'}
\sum_{\beta''\leq\beta'}{\beta\choose\beta'}{\beta'\choose\beta''}{\gamma\choose\gamma'}{\gamma'\choose\gamma''}
(2hLL')^{|\alpha|}L'^{|\gamma|}\\
&{}&\cdot\frac{(\beta'-\beta''+\gamma'-\gamma'')!}
{d_1^{|\beta'|-|\beta''|+|\gamma'|-|\gamma''|}M_{\beta'-\beta''+\gamma'-\gamma''}}
\cdot\frac{(2hL)^{|\beta''|+|\gamma''|}H^{|\beta'|+|\gamma'|}}{t^{|\gamma|-|\gamma'|}}\cdot
\prod_{j=1}^{|\beta|-|\beta'|}s_j\\
&\leq&C_4\tilde{C}_1\tilde{C}_2\frac{M_{\beta}e^{N_{k_p}(|\xi|)}e^{N_{s_p}(|\xi|)}e^{N_{k_p}(|x|)}e^{N_{k_p}(|y|)}}
{e^{N_{r_p}(|\xi|)}e^{M(t|y|)}e^{N_{r_p}(|x-y|)}}\\
&{}&\cdot\sum_{\beta'\leq\beta}\sum_{\alpha,\gamma}\sum_{\gamma'\leq\gamma}\sum_{\gamma''\leq\gamma'}
\sum_{\beta''\leq\beta'}{\beta\choose\beta'}{\beta'\choose\beta''}{\gamma\choose\gamma'}{\gamma'\choose\gamma''}
(2hLL')^{|\alpha|}L'^{|\gamma|}\\
&{}&\cdot\frac{(2hL)^{|\beta''|+|\gamma''|}H^{|\beta'|+|\gamma'|}}
{t^{|\gamma|-|\gamma'|}}\cdot\prod_{j=1}^{|\beta|-|\beta'|}s_j\\
&\leq&C_5\tilde{C}_1\tilde{C}_2\frac{M_{\beta}e^{N_{k_p}(|\xi|)}e^{N_{s_p}(|\xi|)}e^{N_{k_p}(|x|)}e^{N_{k_p}(|y|)}}
{e^{N_{r_p}(|\xi|)}e^{M(t|y|)}e^{N_{r_p}(|x-y|)}}\\
&{}&\cdot2^{|\beta|}\prod_{j=1}^{|\beta|}s_j\sum_{\alpha,\gamma}
(2hLL')^{|\alpha|}\left(\frac{L'}{t}+2hLL'H+L'H\right)^{|\gamma|},
\eeqs
where, in the last inequality, we used that $\ds\frac{\lambda^p}{\prod_{j=1}^{p}s_j}\longrightarrow 0$, when $p\longrightarrow \infty$, for any fixed $\lambda>0$ (i.e. it is bounded for all $p\in\ZZ_+$). (This follows from the fact that $(s_p)\in\mathfrak{R}$.) Note that $e^{N_{2r_p}(|x|)}\leq C_6 e^{N_{r_p}(|x-y|)}e^{N_{r_p}(|y|)}$. Also, the integral $\ds\int_{\RR^d}\frac{e^{N_{k_p}(|y|)}e^{N_{r_p}(|y|)}}
{e^{M(t|y|)}}dy$ converges (this easily follows from the fact that $e^{N_{k_p}(|y|)}\leq c''e^{M(t'|y|)}$ for every $t'>0$, where the constant $c''$ depends on $t'$; similarly for $e^{N_{r_p}(|y|)}$). Take $L'$ such that the sum converges. Then, for the derivatives of $Au$, we obtain $\ds\left|D^{\beta}_x Au(x)\right|\leq C\tilde{C}_1\tilde{C}_2e^{-N_{s'_p}(|x|)}M_{\beta}\prod_{j=1}^{|\beta|}(2s_j)$, i.e. $\ds \|Au\|_{(2s_p),(s'_p)}\leq C\tilde{C}_1\tilde{C}_2$, for all $a\in B_1$ and $u\in B_2$.
\end{proof}
Let $\tau\in\RR$ be fixed. The inclusion $\Gamma_{A_p,B_p,\rho}^{*,\infty}\left(\RR^{2d}\right)\longrightarrow \Pi_{A_p,B_p,\rho}^{*,\infty}\left(\RR^{3d}\right)$, $b\in\Gamma_{A_p,B_p,\rho}^{*,\infty}\left(\RR^{2d}\right)$, $b\mapsto a$, where $a(x,y,\xi)=b((1-\tau)x+\tau y,\xi)$, is continuous. Moreover, if $u,\phi\in\SSS^*\left(\RR^d\right)$ such that $\phi(0)=1$, by theorem \ref{17}, we have
\beqs
\Op_{\tau}(b)u(x)&=&\frac{1}{(2\pi)^d}\int_{\RR^d}\int_{\RR^d}e^{i(x-y)\xi}b((1-\tau)x+\tau y,\xi)u(y)dyd\xi\\
&=&\lim_{\delta\rightarrow 0^+}\frac{1}{(2\pi)^d}\int_{\RR^{2d}}e^{i(x-y)\xi}b((1-\tau)x+\tau y,\xi)\phi(\delta\xi)u(y)dyd\xi.
\eeqs
Hence, the operator $\Op_{\tau}(b)$ coincides with the operator $B$ corresponding to $b$ when we observe $b((1-\tau)x+\tau y,\xi)$ as an element of $\Pi_{A_p,B_p,\rho}^{*,\infty}\left(\RR^{3d}\right)$. We get that the mapping $(b,u)\mapsto \Op_{\tau}(b)u$, $\Gamma_{A_p,B_p,\rho}^{*,\infty}\left(\RR^{2d}\right)\times
\SSS^*\left(\RR^d\right)\longrightarrow\SSS^*\left(\RR^d\right)$, is hypocontinuous. For $b\in \Gamma_{A_p,B_p,\rho}^{*,\infty}\left(\RR^{2d}\right)$, denote its kernel by $K(x,y)$. If we consider the transposed of the operator $\Op_{\tau}(b)$ then its kernel is $K(y,x)$. On the other hand, by (\ref{3}),
\beqs
K(y,x)=\mathcal{F}^{-1}_{\xi\rightarrow x-y}(b)(\tau x+(1-\tau) y,-\xi).
\eeqs
Hence ${}^t\Op_{\tau}(b(x,\xi))=\Op_{1-\tau}(b(x,-\xi))$ i.e. ${}^{t}\Op_{\tau}(b)$ is pseudo-differential operator and by the above it is a continuous mapping from $\SSS^*\left(\RR^d\right)$ to $\SSS^*\left(\RR^d\right)$. Using this we can extend $\Op_{\tau}(b)$ to a continuous operator from  $\SSS'^*\left(\RR^d\right)$ to $\SSS'^*\left(\RR^d\right)$ in the following way
\beqs
\langle \Op_{\tau}(b)u,v\rangle=\langle u,{}^{t}\Op_{\tau}(b)v\rangle,\, u\in\SSS'^*\left(\RR^d\right), v\in\SSS^*\left(\RR^d\right).
\eeqs

We need the following technical lemmas.

\begin{lemma}\label{69}
Let $M_p$ be a sequence which satisfies $(M.1)$, $(M.2)$ and $(M.3)$ and $m$ a positive real. Then, for all $n\in\ZZ_+$, $M(mm_n)\leq 2(c_0m+2)n\ln H+\ln c_0$, where $c_0$ is the constant form the conditions $(M.2)$ and $(M.3)$. If $(t_p)\in\mathfrak{R}$ then, $N_{t_p}(mm_n)\leq n\ln H+\ln c$ for all $n\in\ZZ_+$, where the constant $c$ depends only on $M_p$, $(t_p)$ and $m$, but not on $n$.
\end{lemma}
\begin{proof} By $(M.3)$, for all $p\geq n+1$, $p\in\NN$, we have
\beqs
\frac{1}{m_{n+1}}+\frac{1}{m_{n+2}}+...+\frac{1}{m_p}\leq c_0\frac{n}{m_{n+1}}\leq c_0\frac{n}{m_n}.
\eeqs
If we multiply the above inequality with $m_p$ and use the fact that the sequence $m_n$ is monotonically increasing, we obtain $\ds p-n\leq c_0\frac{nm_p}{m_n}$, i.e. $\ds\frac{mm_n}{m_p}\leq c_0\frac{mn}{p-n}$. Hence, for $p\geq [c_0m]n+2n\geq n+1$, we obtain that $mm_n\leq m_p$. Denote by $k$ the term $[c_0m]+2$. $M(\rho)$ is monotonically increasing, so $M(mm_n)\leq M(m_{kn})$. For $p\geq kn$, we have
\beqs
\frac{m_{kn}^{p+1}}{M_{p+1}}=\frac{m_{kn}^p}{M_p}\cdot\frac{m_{kn}}{m_{p+1}}\leq\frac{m_{kn}^p}{M_p}.
\eeqs
Hence $\ds M(m_{kn})=\sup_p \ln_+\frac{m_{kn}^p}{M_p}=\sup_{p\leq kn} \ln_+\frac{m_{kn}^p}{M_p}$. But, for $p\leq kn$, $p\in\NN$, we have
\beqs
\frac{m_{kn}^p}{M_p}\leq \frac{m_{kn+1}\cdot m_{kn+2}\cdot...\cdot m_{kn+p}}{M_p}=\frac{M_{kn+p}}{M_pM_{kn}}\leq c_0H^{kn+p}\leq c_0H^{2kn},
\eeqs
where, in the second inequality, we used $(M.2)$. We obtained
\beqs
M(mm_n)\leq M(m_{kn})=\sup_{p\leq kn} \ln_+\frac{m_{kn}^p}{M_p}\leq 2kn\ln H+\ln c_0\leq 2(c_0m+2)n\ln H+\ln c_0,
\eeqs
which completes the proof for the first part. For the second part, denote by $T_p$ the product $\prod_{j=1}^p t_j$. Observe that, for $p\in\ZZ_+$, we have
\beqs
\frac{m^pm_n^p}{T_p M_p}\leq \frac{m^pm_{n+1}\cdot m_{n+2}\cdot...\cdot m_{n+p}}{T_pM_p}=\frac{m^p M_{n+p}}{T_pM_pM_n}\leq c_0H^n\frac{(mH)^p}{T_p}\leq c H^n,
\eeqs
where, in the last inequality, we used the fact that $(t_p)$ monotonically increases to infinity. Obviously $c$ does not depend on $p$ or $n$, only on $m$, $(t_p)$ and $M_p$. From this we obtain $N_{t_p}(mm_n)\leq n\ln H+\ln c$, which completes the second part of the lemma.
\end{proof}

\begin{lemma}\label{70}
Let $M_p$ be a sequence which satisfies $(M.1)$ and $(M.3)'$ and $\ds R>1+\frac{1}{M_1}$ be arbitrary. There exist a sequence $\psi_n(\xi)\in\DD^*\left(\RR^d\right)$, $n\in\NN$, such that $\ds\sum_{n=0}^{\infty}\psi_n=1$, $\ds\supp\psi_0\subseteq\left\{\xi\in\RR^d|\langle \xi\rangle< 3RM_1\right\}$, $\supp\psi_n\subseteq \left\{\xi\in\RR^d|2Rm_n<\langle \xi\rangle<3Rm_{n+1}\right\}$, for $n\in\ZZ_+$ and for every $h>0$ there exists $C>0$, resp. there exist $h>0$ and $C>0$ such that
\beqs
\left|D^{\alpha}\psi_0(\xi)\right|\leq C\left(\frac{h}{RM_1}\right)^{|\alpha|}M_{\alpha}, \mbox{ and } \left|D^{\alpha}\psi_n(\xi)\right|\leq C\left(\frac{h}{Rm_{n}}\right)^{|\alpha|}M_{\alpha},\, \forall n\in\ZZ_+,
\eeqs
for all $\xi\in\RR^d$ and $\alpha\in\NN^d$.
\end{lemma}
\begin{proof} Let $\phi\in\DD^*$ such that $0\leq\phi\leq 1$, $\phi(\xi)=1$, for $\langle \xi\rangle<\sqrt{6}$, $\phi(\xi)=0$, for $\langle \xi\rangle>3$. Put
\beqs
\psi_0(\xi)=\phi\left(\frac{\xi}{RM_1}\right),\, \psi_n(\xi)=\phi\left(\frac{\xi}{Rm_{n+1}}\right)-\phi\left(\frac{\xi}{Rm_n}\right).
\eeqs
It is easy to check that $\psi_n$, $n\in\NN$, satisfy the claim in the lemma.
\end{proof}

Let $\ds\rho_0=\inf\{\rho\in\RR_+|A_p\subset M_p^{\rho}\}$. Obviously $0<\rho_0\leq1$. In general, the infimum can not be reached.\\
\indent Counterexample. Let $r_1=1$ and $\ds r_p=p^{1-1/(2\sqrt{\ln p})}$ for $p\in\NN$, $p\geq 2$. The sequences $r_p$ and $p^{1/(2\sqrt{\ln p})}$ are monotonically increasing. Put $\ds R_p=\prod_{j=1}^pr_p$. Take $M_0=1$, $M_p=p!^2 R_p$ and $A_p=p!^2$. Then, obviously, $A_p$ satisfies $(M.1)$, $(M.2)$ and $(M.3)$. One easily checks that $M_p$ satisfies $(M.1)$, $(M.2)$ and $(M.3)$. It is clear that $A_p\subset M_p$. Note that $A_p\not\subset M_p^{2/3}$. In the contrary, there will exist $C>0$ and $L>0$ such that $p!^2\leq CL^p p!^{4/3} R_p^{2/3}$, i.e. $\ds\frac{p!}{L^{3p/2}R_p}\leq C_1$, for all $p\in\ZZ_+$, where we put $C_1=C^{3/2}$. This is impossible, because this means that $\ds\sum_{j=1}^p\ln \frac{j}{L^{3/2}r_j}$ is bounded from above for all $p\in\ZZ_+$, but
\beqs
\lim_{j\rightarrow\infty}\ln \frac{j}{L^{3/2}r_j}=\lim_{j\rightarrow\infty}\ln \frac{j^{1/(2\sqrt{\ln j})}}{L^{3/2}}=
\lim_{j\rightarrow\infty}\left(\frac{\sqrt{\ln j}}{2}-\frac{3}{2}\ln L\right)=\infty.
\eeqs
On the other hand, note that for $\lambda>2/3$, $A_p\subset M_p^{\lambda}$. This is true because
\beqs
\frac{p!^2}{p!^{2\lambda}R_p^{\lambda}}=\frac{p!^{2(1-\lambda)}}{R_p^{\lambda}}=\prod_{j=2}^p\frac{j^{2(1-\lambda)}}
{j^{\lambda-\lambda/(2\sqrt{\ln j})}}=\prod_{j=2}^p\frac{j^{\lambda/(2\sqrt{\ln j})}}{j^{3\lambda-2}}
\eeqs
and the last term converges to zero when $p\longrightarrow \infty$ (note that $3\lambda-2>0$ when $\lambda>2/3$). From now on we will assume that $\rho$ is such that $\rho_0\leq \rho\leq 1$ if the infimum can be reached, otherwise $\rho_0<\rho\leq 1$.\\
\indent For $0<r<1$, define the set $\Omega_r=\left\{(x,y)\in\RR^{2d}||x-y|>r\langle x\rangle\right\}$.

\begin{lemma}\label{pll}
Let $0<r<1$. There exists $\theta\in\EE^*\left(\RR^{2d}\right)$ such that $0\leq \theta\leq 1$, $\theta=0$ on $\RR^{2d}\backslash \Omega_{r/4}$, $\theta=1$ on $\Omega_{3r/4}$ and for every $h>0$ there exists $C>0$, resp. there exist $h>0$ and $C>0$, such that $\left|D^{\beta}_x D^{\gamma}_y \theta(x,y)\right|\leq C h^{|\beta|+|\gamma|}M_{\beta+\gamma}$, for all $(x,y)\in\RR^{2d}$, $\alpha,\beta\in\NN^d$.
\end{lemma}

\begin{proof} Let $f(x,y)=1$ on $\Omega_{r/2}$ and $f(x,y)=0$ on $\RR^{2d}\backslash \Omega_{r/2}$. Let $\mu\in\DD^*\left(\RR^{2d}\right)$ is such that $\mu\geq 0$ with support in the closed ball with center at the origin and radius $r/16$ and $\ds\int_{\RR^{2d}}\mu(x,y)dxdy=1$. Put $\theta=f*\mu$. Then, one easily checks that $\theta$ satisfies the conditions in the lemma.
\end{proof}

\begin{proposition}\label{72}
Let $a\in\Pi_{A_p,B_p,\rho}^{*,\infty}\left(\RR^{3d}\right)$ and $A$ be the operator corresponding to $a$ as defined above. The kernel $K$ of this operator is an element of $\mathcal{C}^{\infty}\left(\Omega_r\right)$ for every $0<r<1$ and for every such $\Omega_r$ and every $h>0$, resp. there exists $h>0$, such that
\beq\label{75}
\sup_{\beta,\gamma\in\NN^d}\sup_{(x,y)\in\overline{\Omega_r}}\frac{h^{\beta+\gamma}\left|D^{\beta}_xD^{\gamma}_y K(x,y)\right|e^{M(h|(x,y)|)}}{M_{\beta+\gamma}}<\infty.
\eeq
Moreover, if there exists $r$, $0<r<1$, such that $a(x,y,\xi)=0$ for $(x,y,\xi)\in\left(\RR^{2d}\backslash\Omega_r\right)\times\RR^d$ then $K\in\SSS^*\left(\RR^{2d}\right)$, i.e. $A$ is *-regularizing.
\end{proposition}
\begin{proof} Let $\psi_n\in \DD^*\left(\RR^d\right)$ be as in lemma \ref{70}, where $R$ will be chosen later. Then, note that the sum $\ds\sum_{n=0}^{\infty}1_y\otimes\psi_n(\xi)$ converges to $1_{y,\xi}$ in $\EE^*\left(\RR^{2d}_{y,\xi}\right)$ (with $1_{y}$ we denote the function of variable $y$ that is identically equal to $1$, similarly $1_{y,\xi}$ is the function of variables $(y,\xi)$ that is identically equal to $1$). Because $a(x,y,\xi)$ is an element of $\EE^*\left(\RR^{2d}_{y,\xi}\right)$, for every fixed $x$, we get
\beqs
a(x,y,\xi)=a(x,y,\xi)\sum_{n=0}^{\infty}\psi_n(\xi)=\sum_{n=0}^{\infty}\left(\psi_n(\xi)a(x,y,\xi)\right),
\eeqs
in $\EE^*\left(\RR^{2d}_{y,\xi}\right)$. Let $u\in\SSS^*\left(\RR^d\right)$. Because $1/P_l(y-x)$ and $1/P_l(\xi)$, resp. $1/P_{l_p}(y-x)$ and $1/P_{l_p}(\xi)$ are elements of $\EE^*\left(\RR^{2d}_{y,\xi}\right)$, for $*=(M_p)$, resp. $*=\{M_p\}$, for fixed $x$, we get\\
$\ds\frac{1}{P_l(\xi)}P_l(D_y)\left(\frac{1}{P_l(y-x)}P_l(D_\xi)\left(a(x,y,\xi)u(y)\right)\right)$
\beqs
=\sum_{n=0}^{\infty}\frac{1}{P_l(\xi)}P_l(D_y)\left(\frac{1}{P_l(y-x)}P_l(D_\xi)
\left(a(x,y,\xi)\psi_n(\xi)u(y)\right)\right),
\eeqs
in the $(M_p)$ case and with $P_{l_p}$ in place of $P_l$ in the $\{M_p\}$ case, in $\EE^*\left(\RR^{2d}_{y,\xi}\right)$. If we choose $l$ small enough such that $\ds \left|P_l(\xi)\right|\geq c_1 e^{M(r|\xi|)}\geq c'_1 e^{2M(r'|\xi|)}$, where $r'>0$ is such that $\ds\int_{\RR^d}e^{M(m|\xi|)}e^{-M(r'|\xi|)}d\xi<\infty$, by the properties of $\psi_n$ similarly as in the proof of lemma \ref{45}, we obtain
\beqs
\left|\frac{1}{P_l(\xi)}P_l(D_y)\left(\frac{1}{P_l(y-x)}P_l(D_\xi)
\left(a(x,y,\xi)\psi_n(\xi)u(y)\right)\right)\right|\leq C
\frac{e^{M(m|x|)}e^{M(m|\xi|)}}{e^{M(r'|\xi|)}e^{M(r'Rm_n)}}\cdot \frac{e^{M(m|y|)}}{e^{M(s|y|)}}
\eeqs
in the $(M_p)$ case, where $m$ is such that $a\in\Pi_{A_p,B_p,\rho}^{(M_p),\infty}\left(\RR^{3d};m\right)$ and $s$ is from the $\SSS^{(M_p)}$ - seminorms of $u$. Respectively, if we choose $(l_p)\in\mathfrak{R}$ small enough such that $\ds \left|P_{l_p}(\xi)\right|\geq c'' e^{N_{r_p}(|\xi|)}\geq c''_1e^{2N_{r'_p}(|\xi|)}$, where $(r'_p)\in\mathfrak{R}$ is such that $\ds \int_{\RR^d}e^{N_{k_p}(|\xi|)}e^{-N_{r'_p}(|\xi|)}d\xi<\infty$, we get
\beqs
\left|\frac{1}{P_{l_p}(\xi)}P_{l_p}(D_y)\left(\frac{1}{P_{l_p}(y-x)}P_{l_p}(D_\xi)
\left(a(x,y,\xi)\psi_n(\xi)u(y)\right)\right)\right|
\leq C\frac{e^{N_{k_p}(|x|)}e^{N_{k_p}(|\xi|)}}{e^{N_{r'_p}(|\xi|)}e^{N_{r'_p}(Rm_n)}}\cdot \frac{e^{N_{k_p}(|y|)}}{e^{M(s|y|)}}
\eeqs
in the $\{M_p\}$ case, where $(k_p)$ is such that (\ref{zon}) holds for $a$ and $s$ depends on $u$. Hence, by dominated convergence,
\beqs
Au(x)&=&\frac{1}{(2\pi)^d}\sum_{n=0}^{\infty}
\int_{\RR^{2d}}e^{i(x-y)\xi}\frac{1}{P_l(\xi)}P_l(D_y)\left(\frac{1}{P_l(y-x)}P_l(D_\xi)
\left(a(x,y,\xi)\psi_n(\xi)u(y)\right)\right)dyd\xi\\
&=&\frac{1}{(2\pi)^d}\sum_{n=0}^{\infty}\int_{\RR^{2d}}e^{i(x-y)\xi}a(x,y,\xi)\psi_n(\xi)u(y)dyd\xi,
\eeqs
in the $(M_p)$ case, resp. the same but with $P_{l_p}$ in place of $P_l$ in the $\{M_p\}$ case and the convergence is uniform for $x$ in compact subsets of $\RR^d$ and in $\SSS'^*\left(\RR^d\right)$. For simpler notation, put $a_n(x,y,\xi)=a(x,y,\xi)\psi_n(\xi)$ and $A_n$ for the associated operator to $a_n$. Then, we get $\ds Au(x)=\sum_{n=0}^{\infty}A_nu(x)$, where the convergence is uniform for $x$ in compact subsets of $\RR^d$ and in $\SSS'^*\left(\RR^d\right)$. So $\ds\sum_{k=0}^nA_k\longrightarrow A$, when $n\longrightarrow\infty$, in $\mathcal{L}_{\sigma}\left(\SSS^*\left(\RR^{d}\right), \SSS'^*\left(\RR^d\right)\right)$. $\SSS^*$ is barreled, so, by the Banach - Steinhaus theorem (see \cite{Schaefer}, theorem 4.6), it follows that $\ds\sum_{k=0}^{n}A_k\longrightarrow A$, when $n\longrightarrow\infty$, in the topology of precompact convergence. But $\SSS^*$ is Montel space, so the convergence holds in $\mathcal{L}_b\left(\SSS^*\left(\RR^d\right),\SSS'^*\left(\RR^d\right)\right)$ (the topology of bounded convergence). Hence, if we denote by $K(x,y)$ the kernel of $A$ and by $K_n$ the kernel of $A_n$, by proposition \ref{ktr}, we get $\ds K=\sum_{n=0}^{\infty}K_n$, where the convergence holds in $\SSS'^*\left(\RR^{2d}\right)$. Now, observe that
\beqs
K_n(x,y)=\frac{1}{(2\pi)^d}\int_{\RR^d}e^{i(x-y)\xi}a(x,y,\xi)\psi_n(\xi)d\xi
\eeqs
and $K_n$ is a $\mathcal{C}^{\infty}$ function. Take $R$ such that $Rm_1\geq 1$. Later on we will impose more conditions on $R$. Let $r\in(0,1)$ be fixed. First, we will observe the $(M_p)$ case. There exists $m>0$ such that $a\in\Pi_{A_p,B_p,\rho}^{M_p,\infty}\left(\RR^{3d};h,m\right)$, for all $h>0$. Let $m'$ be arbitrary but fixed positive real number. We want to prove (\ref{75}) for this $m'$. Obviously, without losing generality, we can assume that $m'\geq 1$. Let $(x,y)\in \Omega_r$ be arbitrary but fixed. Let $q\in\{1,...,d\}$ be such that $|x_q-y_q|\geq |x_j-y_j|$, for all $j\in\{1,...,d\}$. Then $\ds |x_q-y_q|>\frac{r}{d}\langle x\rangle$. We calculate\\
$\ds D^{\beta}_xD^{\gamma}_y K_n(x,y)$
\beqs
&=&\frac{1}{(2\pi)^d}\sum_{\substack{\beta'+\beta''=\beta\\ \gamma'+\gamma''=\gamma}}\sum_{k=0}^n\sum_{\substack{k'+k''=k\\k''\leq \beta''_q+\gamma''_q}}
{\beta\choose\beta'}{\gamma\choose\gamma'}{n\choose k}{k\choose k'}\frac{(\beta''+\gamma'')!}{(\beta''+\gamma''-e_q k'')!}
\frac{(-1)^{|\gamma''|+n}}{(x_q-y_q)^ni^{k''}}\\
&{}&\cdot
\int_{\RR^d}e^{i(x-y)\xi}\frac{1}{P_l(y-x)}P_l(D_{\xi})\left(\xi^{\beta''+\gamma''-e_q k''}D^{k'}_{\xi_q}D^{\beta'}_xD^{\gamma'}_y a(x,y,\xi)D^{n-k}_{\xi_q}\psi_n(\xi)\right)d\xi.
\eeqs
On $\Omega_r$ we have the following inequality
\beq\label{85}
|(x,y)|\leq |x|+|y|\leq\langle x\rangle+|x-y|+|x|\leq2\langle x\rangle+|x-y|\leq\left(\frac{2}{r}+1\right)|x-y|.
\eeq
Hence, by using proposition 3.6 of \cite{Komatsu1}, we can find $m''>0$ such that $e^{M(m''|x-y|)}\geq c''e^{M(m|x|)}e^{M(m|y|)}e^{M(m'|(x,y)|)}$ on $\Omega_r$. Take $l'\geq m''$. Then we have
\beq\label{90}
e^{M(l'|\xi|)}\geq c'''e^{M(m''|\xi|)}.
\eeq
By proposition \ref{orn}, we can find small enough $l>0$ such that  $\ds |P_l(\xi)|\geq c''e^{M(l'|\xi|)}$. On the other, hand if we represent $P_l(D)$ as $\sum_{\alpha}c_{\alpha}D^{\alpha}$, then there exist $C'_1>0$ and $L_0>0$ such that $|c_{\alpha}|\leq C'_1L_0^{|\alpha|}/M_{\alpha}$. We will estimate the part in the integral for $n\in\ZZ_+$ as follows\\
$\ds \left|\frac{1}{P_l(y-x)}P_l(D_{\xi})\left(\xi^{\beta''+\gamma''-e_q k''}D^{k'}_{\xi_q}D^{\beta'}_xD^{\gamma'}_y a(x,y,\xi)D^{n-k}_{\xi_q}\psi_n(\xi)\right)\right|$
\beqs
&\leq&\frac{1}{\left|P_l(y-x)\right|}\sum_{\alpha}\left|c_{\alpha}\right|
\sum_{\alpha'\leq\alpha}\sum_{\substack{\alpha''+\alpha'''=\alpha'\\ \alpha'''\leq\beta''+\gamma''-e_q k''}}
{\alpha\choose\alpha'}{\alpha'\choose\alpha''}
\frac{(\beta''+\gamma''-e_q k'')!}{(\beta''+\gamma''-e_q k''-\alpha''')!}\\
&{}&\cdot|\xi|^{|\beta''+\gamma''-e_q k''|-|\alpha'''|}
\left|D^{\alpha''+e_q k'}_{\xi}D^{\beta'}_xD^{\gamma'}_y a(x,y,\xi)\right|
\left|D^{\alpha-\alpha'+e_q(n-k)}_{\xi}\psi_n (\xi)\right|\\
&\leq&C_1e^{-M(l'|x-y|)}\sum_{\alpha}\left|c_{\alpha}\right|
\sum_{\alpha'\leq\alpha}\sum_{\substack{\alpha''+\alpha'''=\alpha'\\ \alpha'''\leq\beta''+\gamma''-e_q k''}}{\alpha\choose\alpha'}{\alpha'\choose\alpha''}
\frac{(\beta''+\gamma''-e_q k'')!}{(\beta''+\gamma''-e_q k''-\alpha''')!}\\
&{}&\cdot|\xi|^{|\beta''+\gamma''-e_q k''|-|\alpha'''|}\cdot
\frac{h_1^{|\alpha|-|\alpha'|+n-k}M_{\alpha-\alpha'+n-k}}{\left(Rm_n\right)^{|\alpha|-|\alpha'|+n-k}}\\
&{}&\cdot\frac{h^{|\alpha''|+|\beta'|+|\gamma'|+k'}
\langle x-y\rangle^{\rho|\alpha''|+\rho k'+\rho|\beta'|+\rho|\gamma'|}A_{\alpha''+k'}B_{\beta'+\gamma'}
e^{M(m|\xi|)}e^{M(m|x|)}e^{M(m|y|)}}{\langle (x,y,\xi)\rangle^{\rho|\alpha''|+\rho k'+\rho|\beta'|+\rho|\gamma'|}},
\eeqs
on the support of $\psi_n$. Note that $\langle x-y\rangle\leq 2(1+|x|^2+|y|^2)^{1/2}\leq 2\langle(x,y,\xi)\rangle$. Hence
\beqs
\langle x-y\rangle^{\rho|\alpha''|+\rho|\beta'|+\rho|\gamma'|}\leq 2^{\rho|\alpha''|+\rho|\beta'|+\rho|\gamma'|}\langle (x,y,\xi)\rangle^{\rho|\alpha''|+\rho|\beta'|+\rho|\gamma'|}.
\eeqs
Also, $(\beta''+\gamma''-e_q k'')!\leq 2^{|\beta''+\gamma''-e_q k''|}(\beta''+\gamma''-e_q k''-\alpha''')! \alpha'''!$. Moreover $B_{\beta'+\gamma'}\leq c'_0 L^{|\beta'|+|\gamma'|}M_{\beta'+\gamma'}$ and $A_{\alpha''+k'}\leq c'_0 L^{|\alpha''|+k'}M_{\alpha''+k'}^{\rho}$. Let $T_n=\left\{\xi\in\RR^d|2Rm_n\leq\langle \xi\rangle\leq3Rm_{n+1}\right\}$. By construction, $\supp \psi_n\subseteq T_n$. Note that, on $T_n$
\beqs
\frac{|\xi|^{|\beta''+\gamma''-e_q k''|-|\alpha'''|}}{\langle(x,y,\xi)\rangle^{\rho k'}}\leq \frac{\langle \xi\rangle^{|\beta''+\gamma''-e_q k''|-|\alpha'''|}}{\langle \xi\rangle^{\rho k'}}
\leq\frac{\left(3Rm_{n+1}\right)^{|\beta''|+|\gamma''|}}{\left(3Rm_{n+1}\right)^{|\alpha'''|+k''}
\left(2Rm_n\right)^{\rho k'}}.
\eeqs
Because $m_n$ is monotonically increasing, $m_n^{n-k}\geq m_n\cdot m_{n-1}\cdot...\cdot m_{k+1}=M_n/M_{k}\geq M_{n-k}$ and similarly, $m_n^{k'}\geq M_{k'}$ and $m_n^{k''}\geq M_{k''}$. Moreover, there exists $\tilde{c}>0$ such that $M_p^{\rho}\leq \tilde{c}M_p$. We use this to estimate the above integral. By Fatou's lemma we have $\int_{\RR^d}\left|\sum ...\right|d\xi\leq \sum\int_{\RR^d}\left| ...\right|d\xi$. Considering the parts that are depended on $\alpha$, $\alpha'$, $\alpha''$ and $\alpha'''$, after using the above inequalities, one obtains\\
$\ds e^{M(3mRm_{n+1})}\sum_{\alpha}
\sum_{\alpha'\leq\alpha}\sum_{\substack{\alpha''+\alpha'''=\alpha'\\ \alpha'''\leq\beta''+\gamma''-e_q k''}}{\alpha\choose\alpha'}{\alpha'\choose\alpha''}\frac{L_0^{|\alpha|-|\alpha'''|}L_0^{|\alpha'''|}}{M_{\alpha}}$\\
${}$\hspace{140 pt}$\ds\cdot\frac{\alpha'''!(2h)^{|\alpha''|}
L^{|\alpha''|+k'}M_{\alpha''+k'}^{\rho}
h_1^{|\alpha|-|\alpha'|}M_{\alpha-\alpha'+n-k}}{\left(3Rm_{n+1}\right)^{|\alpha'''|+k''}
\left(2Rm_n\right)^{\rho k'}\left(Rm_n\right)^{|\alpha|-|\alpha'|+n-k}}\cdot|T_n|$
\beqs
&\leq& C_2e^{M(3mRm_{n+1})}\sum_{\alpha}
\sum_{\alpha'\leq\alpha}\sum_{\substack{\alpha''+\alpha'''=\alpha'\\ \alpha'''\leq\beta''+\gamma''-e_q k''}}{\alpha\choose\alpha'}{\alpha'\choose\alpha''}\frac{L_0^{|\alpha|-|\alpha'''|}}{M_{\alpha}}\\
&{}&\hspace{110 pt}\cdot
\frac{(2h)^{|\alpha''|}L^{|\alpha|+n}h_1^{|\alpha|-|\alpha'|}H^{|\alpha|+n}
M_{\alpha'''}M_{\alpha''}M_{k'}^{\rho}M_{\alpha-\alpha'}M_{n-k}}
{\left(Rm_n\right)^{|\alpha'''|+ k''}
\left(Rm_n\right)^{\rho k'}\left(Rm_n\right)^{|\alpha|-|\alpha'|+n-k}}\cdot|T_n|\\
&\leq&\frac{C_2|T_n|e^{M(3mRm_{n+1})}}{R^{\rho n}}\sum_{\alpha}
\sum_{\alpha'\leq\alpha}\sum_{\substack{\alpha''+\alpha'''=\alpha'\\ \alpha'''\leq\beta''+\gamma''-e_q k''}}
{\alpha\choose\alpha'}{\alpha'\choose\alpha''}\frac{(HL)^{|\alpha|+n}(2L_0h)^{|\alpha''|}
(L_0h_1)^{|\alpha|-|\alpha'|}}{(RM_1)^{|\alpha|-|\alpha''|}m_n^{k''}}\\
&\leq&\frac{C_2|T_n|(HL)^n e^{M(3mRm_{n+1})}}{R^{\rho n}M_{k''}}\sum_{\alpha}\frac{(HL)^{|\alpha|}}
{(RM_1)^{|\alpha|}}\sum_{\alpha'\leq\alpha}{\alpha\choose\alpha'}
\left(1+2L_0hRM_1\right)^{|\alpha'|}(L_0h_1)^{|\alpha|-|\alpha'|}\\
&=&\frac{C_2|T_n|(HL)^n e^{M(3mRm_{n+1})}}{R^{\rho n}M_{k''}}\sum_{\alpha}
\left(\frac{HL}{RM_1}+2hHLL_0+\frac{h_1HLL_0}{RM_1}\right)^{|\alpha|}.
\eeqs
Take $R$ such that $\ds \frac{HL}{RM_1}+\frac{1}{8}+\frac{1}{RM_1}\leq \frac{1}{2}$ and take $h$ and $h_1$ small enough such that $2hHLL_0\leq 1/8$ and $h_1HLL_0\leq 1$. Then, the sum will be uniformly convergent for all $h$ and $h_1$ for which the previous inequalities hold. The choice of $R$ depends only on $A_p$, $B_p$ and $M_p$ (and not on $L_0$, hence not on the operator $P_l$). Also, the choice of $h$ and $h_1$ depend on $A_p$, $B_p$, $M_p$ and the operator $P_l$, but not on $R$. Before we continue, note that, from the way we choose $q$, we have the following inequality
\beqs
1+|x-y|^2\leq \langle x\rangle^2+d|x_q-y_q|^2\leq \frac{d^2}{r^2}|x_q-y_q|^2+d|x_q-y_q|^2\leq \left(\frac{d}{r}+d\right)^2|x_q-y_q|^2.
\eeqs
For shorter notation, put $\ds r_1=\frac{d}{r}+d$. So, we obtain $\langle x-y\rangle \leq r_1|x_q-y_q|$. Now, for the estimate of $\left|D^{\beta}_x D^{\gamma}_y K_n(x,y)\right|$, by using (\ref{90}), we obtain\\
$\ds \left|D^{\beta}_xD^{\gamma}_y K_n(x,y)\right|$
\beqs
&\leq& C_3\sum_{\substack{\beta'+\beta''=\beta\\ \gamma'+\gamma''=\gamma}}\sum_{k=0}^n\sum_{\substack{k'+k''=k\\k''\leq \beta''_q+\gamma''_q}}
{\beta\choose\beta'}{\gamma\choose\gamma'}{n\choose k}{k\choose k'}
\frac{(\beta''+\gamma'')!}{(\beta''+\gamma''-e_q k'')!}\frac{\langle x-y\rangle^{\rho k'}}{|x_q-y_q|^n}\\
&{}&\cdot 2^{|\beta''+\gamma''-e_q k''|}\left(3Rm_{n+1}\right)^{|\beta''|+|\gamma''|}h^{|\beta'|+|\gamma'|+k'}2^{|\beta'|+|\gamma'|}L^{|\beta'|+|\gamma'|}
M_{\beta'+\gamma'}h_1^{n-k}\\
&{}&\hspace{110 pt}\cdot e^{-M(l'|y-x|)}e^{M(m|x|)}e^{M(m|y|)}
\frac{|T_n|(HL)^ne^{M(3mRm_{n+1})}}{R^{\rho n}M_{k''}}\\
&\leq&C_3r_1^n\sum_{\substack{\beta'+\beta''=\beta\\ \gamma'+\gamma''=\gamma}}\sum_{k=0}^n\sum_{k'+k''=k}
{\beta\choose\beta'}{\gamma\choose\gamma'}{n\choose k}{k\choose k'}
\frac{4^{|\beta''|+|\gamma''|}k''!}{M_{k''}}\cdot\frac{h_1^{k''}}{h_1^{k''}}\\
&{}&\cdot \frac{1}{2^{k''}(m'R)^{|\beta''|+|\gamma''|}}
\frac{\left(3m'R^2m_{n+1}\right)^{|\beta''|+|\gamma''|}}
{M_{\beta''+\gamma''}}
h^{|\beta'|+|\gamma'|+k'}(2L)^{|\beta'|+|\gamma'|}h_1^{n-k}M_{\beta+\gamma}\\
&{}&\hspace{110 pt}\cdot e^{-M(m'|(x,y)|)}\frac{|T_n|(HL)^ne^{M(3mRm_{n+1})}}{R^{\rho n}}.
\eeqs
Note that $\ds\frac{\left(3m'R^2m_{n+1}\right)^{|\beta''|+|\gamma''|}}
{M_{\beta''+\gamma''}}\leq e^{M(3m'R^2m_{n+1})}$. Also, by using $(M.2)$, we obtain
\beqs
|T_n|=\omega_d\left(\left(9R^2m_{n+1}^2-1\right)^{d/2}-\left(4R^2m_n^2-1\right)^{d/2}\right)\leq \omega_d(3Rm_{n+1})^d\leq \omega_d (3c_0RM_1)^dH^{(n+1)d},
\eeqs
where $\omega_d$ is the volume of the $d$-dimensional unit ball. By proposition 3.6 of \cite{Komatsu1}
\beqs
e^{M(3mRm_{n+1})}e^{M(3m'R^2m_{n+1})} \leq c_0e^{M(3Hm'R^2m_{n+1})},
\eeqs
where we take $R\geq m$ (which depends only on $a$). We obtain\\
$\ds \left|D^{\beta}_x D^{\gamma}_y K_n(x,y)\right|$
\beqs
\leq C_4(3c_0RM_1H)^d\frac{M_{\beta+\gamma}(HL)^n H^{nd}e^{M(3Hm'R^2m_{n+1})}r_1^n}{e^{M(m'|(x,y)|)}R^{\rho n}}
\left(\frac{4}{m'R}+2hL\right)^{|\beta|+|\gamma|}\left(h_1+h+\frac{h_1}{2}\right)^n.
\eeqs
By lemma \ref{69} we have
\beqs
e^{M(3Hm'R^2m_{n+1})}\leq c_0 H^{2(3c_0Hm'R^2+2)(n+1)}=c_0 H^{2(3c_0Hm'R^2+2)}\left(H^{2(3c_0Hm'R^2+2)}\right)^n.
\eeqs
Take $\ds R^{\rho}> H^{d+1}Lr_1$ and $R\geq 8$. For the fixed $m'$ in the beginning of the proof, choose $h$ small enough such that $2hL\leq 1/(2m')$. Then $\ds \frac{4}{m'R}+2hL\leq\frac{1}{m'}$. For the chosen $R$, choose $h$ and $h_1$ smaller then the chosen before such that $\ds H^{2(3c_0Hm'R^2+2)}\left(h_1+h+\frac{h_1}{2}\right)\leq 1$. (Note that the choice of $R$ and hence the choice of $\psi_n$, $n\in\NN$, depends only on $A_p$, $B_p$, $M_p$ and $a$, but not on the operator $P_l$ or $m'$.) Then $\ds\sum_{n=1}^{\infty} \left|D^{\beta}_x D^{\gamma}_y K_n(x,y)\right|$ will converge and we have the following estimate
\beqs
\sum_{n=1}^{\infty}\left|D^{\beta}_x D^{\gamma}_y K_n(x,y)\right|\leq C\frac{M_{\beta+\gamma}}{e^{M(m'|(x,y)|)}m'^{|\beta|+|\gamma|}}.
\eeqs
For $\left|D^{\beta}_x D^{\gamma}_y K_0(x,y)\right|$, by similar procedure, we obtain the same estimate. Hence (\ref{75}) holds and the proof for the $(M_p)$ case is complete.\\
\indent The $\{M_p\}$ case. We will prove that for every $(t_p),(t'_p)\in\mathfrak{R}$,
\beq\label{pon}
\sup_{\beta,\gamma\in\NN^d}\sup_{(x,y)\in\overline{\Omega_r}}\frac{\left|D^{\beta}_xD^{\gamma}_y K(x,y)\right|e^{N_{t_p}(|(x,y)|)}}{T'_{\beta+\gamma}M_{\beta+\gamma}}<\infty,
\eeq
for every fixed $0<r<1$, where $\ds T'_{\beta+\gamma}=\prod_{j=1}^{|\beta|+|\gamma|}t'_j$ and $T'_0=1$. From this, the claim in the lemma follows. To prove this, fix $0<r<1$ and take $\theta\in\EE^{\{M_p\}}\left(\RR^{2d}\right)$ as in lemma \ref{pll}. Define $\tilde{K}=K\theta$. Then $\tilde{K}$ is $\mathcal{C}^{\infty}$ function and for every $(t_p),(t'_p)\in\mathfrak{R}$, $\ds \sup_{\beta,\gamma\in\NN^d}\sup_{(x,y)\in\RR^{2d}}\frac{\left|D^{\beta}_xD^{\gamma}_y \tilde{K}(x,y)\right|e^{N_{t_p}(|(x,y)|)}}{T'_{\beta+\gamma}M_{\beta+\gamma}}<\infty$. Hence $\tilde{K}\in \SSS^{\{M_p\}}\left(\RR^{2d}\right)$. So, there exists $h>0$ such that $\ds \sup_{\beta,\gamma\in\NN^d}\sup_{(x,y)\in\RR^{2d}}\frac{h^{|\beta|+|\gamma|}\left|D^{\beta}_xD^{\gamma}_y \tilde{K}(x,y)\right|e^{M(h|(x,y)|)}}{M_{\beta+\gamma}}<\infty$. But, $\tilde{K}(x,y)=K(x,y)$ on $\Omega_{3r/4}$ and the desired estimate follows. Now, to prove (\ref{pon}). Let $a\in\Pi_{A_p,B_p}^{\{M_p\},\infty}\left(\RR^{3d}\right)$. Then there exists $h>0$ such that $a\in\Pi_{A_p,B_p,\rho}^{M_p,\infty}\left(\RR^{3d};h,m\right)$, for all $m>0$. By lemma \ref{nn1}, there exist $(k_p)\in\mathfrak{R}$ and $c'_0>0$ such that
\beqs
\left|D^{\alpha}_{\xi}D^{\beta}_xD^{\gamma}_y a(x,y,\xi)\right|\leq c'_0\frac{h^{|\alpha|+|\beta|+|\gamma|}\langle x-y\rangle^{\rho|\alpha|+\rho|\beta|+\rho|\gamma|}A_{\alpha}B_{\beta+\gamma}
e^{N_{k_p}(|\xi|)}e^{N_{k_p}(|x|)}e^{N_{k_p}(|y|)}}{\langle (x,y,\xi)\rangle^{\rho|\alpha|+\rho|\beta|+\rho|\gamma|}},
\eeqs
for all $\alpha,\beta,\gamma\in\NN^d$ and $(x,y,\xi)\in\RR^{3d}$. Let $(t_p),(t'_p)\in\mathfrak{R}$ be fixed. For $(l_p)\in\mathfrak{R}$ consider $P_{l_p}(\xi)$. By proposition \ref{orn}, we can choose $P_{l_p}(\xi)$ such that, $\ds \left|P_{l_p}(\xi)\right|\geq c'' e^{N_{l'_p}(|\xi|)}$ where $(l'_p)\in\mathfrak{R}$ is such that $e^{N_{l'_p}(|x-y|)}\geq c_1e^{N_{k_p}(|x|)}e^{N_{k_p}(|y|)}e^{N_{t_p}(|(x,y)|)}$ on $\Omega_r$. This is possible because of (\ref{85}). On the other hand, if we represent $\ds P_{l_p}(\xi)=\sum_{\alpha}c_{\alpha}\xi^{\alpha}$ then for every $L'>0$ there exists $C'>0$ such that $\left|c_{\alpha}\right|\leq C'L'^{|\alpha|}/M_{\alpha}$. By the same calculations, one obtains the same form for $D^{\beta}_xD^{\gamma}_y K_n(x,y)$ as in the $(M_p)$ case, but with $P_{l_p}$ in place of $P_l$. The prove continues in the same way as above. We will point out only the notable differences. The first difference is in the estimate of the part that is depended on $\alpha$, $\alpha'$, $\alpha''$ and $\alpha'''$ (for $n\in\ZZ_+$) and the integral over $\RR^d_{\xi}$, where in the $\{M_p\}$ case one obtains the estimate
\beqs
\frac{C_2|T_n|(HL)^n e^{N_{k_p}(3Rm_{n+1})}}{R^{\rho n}M_{k''}}\sum_{\alpha}
\left(\frac{HL}{RM_1}+2hHLL'+\frac{h_1HLL'}{RM_1}\right)^{|\alpha|}.
\eeqs
The convergence of this sum follows from the fact that we can take $R$ arbitrary large and $L'$ arbitrary small. Moving on to the estimate of $\left|D^{\beta}_xD^{\gamma}_y K_n(x,y)\right|$, in similar fashion, one obtains the following\\
$\ds\left|D^{\beta}_xD^{\gamma}_y K_n(x,y)\right|$
\beqs
&\leq&C_3r_1^n\sum_{\substack{\beta'+\beta''=\beta\\ \gamma'+\gamma''=\gamma}}\sum_{k=0}^n\sum_{k'+k''=k}
{\beta\choose\beta'}{\gamma\choose\gamma'}{n\choose k}{k\choose k'}
\frac{12^{|\beta''|+|\gamma''|}k''!}{M_{k''}}\frac{R^{|\beta''|+|\gamma''|}}{2^{k''}}\\
&{}&\cdot \frac{m_{n+1}^{|\beta''|+|\gamma''|}}
{M_{\beta''+\gamma''}}
h^{|\beta'|+|\gamma'|+k'}(2L)^{|\beta'|+|\gamma'|}h_1^{n-k}M_{\beta+\gamma}
e^{-N_{t_p}(|(x,y)|)}\frac{|T_n|(HL)^ne^{N_{k_p}(3Rm_{n+1})}}{R^{\rho n}}.
\eeqs
By using the increasingness of $m_p$ and $(M.2)$, we obtain
\beqs
\frac{m_{n+1}^{|\beta''|+|\gamma''|}}{M_{\beta''+\gamma''}}\leq
\frac{m_{n+2}\cdot m_{n+3}\cdot...\cdot m_{n+1+|\beta''|+|\gamma''|}}{M_{\beta''+\gamma''}}=\frac{M_{n+1+\beta''+\gamma''}}{M_{\beta''+\gamma''}M_{n+1}}\leq c_0H^{n+1+|\beta''|+|\gamma''|}.
\eeqs
We get the estimate:
\beqs
\left|D^{\beta}_xD^{\gamma}_y K_n(x,y)\right|\leq C_4
\frac{M_{\beta+\gamma}|T_n|(H^2L)^ne^{N_{k_p}(3Rm_{n+1})}r_1^n}{e^{N_{t_p}(|(x,y)|)}R^{\rho n}}
(12RH+2hL)^{|\beta|+|\gamma|}\left(h_1+h+\frac{1}{2}\right)^n.
\eeqs
By lemma \ref{69} $\ds e^{N_{k_p}(3Rm_{n+1})}\leq cH^{n+1}$, where $c$ depends only on $(k_p)$, $R$ and $M_p$ (does not depend on $n$). Now, if we use the same estimate for $|T_n|$ as in the $(M_p)$ case, if we take large enough $R$, the sum $\ds\sum_{n=1}^{\infty} \left|D^{\beta}_xD^{\gamma}_y K_n(x,y)\right|$ will converge and we obtain
\beqs
\sum_{n=1}^{\infty}\left|D^{\beta}_xD^{\gamma}_y K(x,y)\right|\leq C\frac{M_{\beta+\gamma}}{e^{N_{t_p}(|(x,y)|)}} (12RH+2hL)^{|\beta|+|\gamma|}.
\eeqs
One obtains similar estimates for $\left|D^{\beta}_x D^{\gamma}_y K_0(x,y)\right|$. Hence we obtain (\ref{pon}) and the proof for the $\{M_p\}$ case is complete. It remains to prove the fact that if there exists $r$, $0<r<1$, such that $a(x,y,\xi)=0$ for $(x,y,\xi)\in\left(\RR^{2d}\backslash\Omega_r\right)\times\RR^d$ then $K\in\SSS^*\left(\RR^{2d}\right)$. But this trivially follows from the proved growth condition of $D^{\beta}_x D^{\gamma}_y K(x,y)$ and the fact that for $(x,y)\in\RR^{2d}\backslash\Omega_r$, $K_n(x,y)=0$ for all $n\in\NN$, hence, $K=0$ on $\RR^{2d}\backslash \Omega_r$.
\end{proof}

\section{Symbolic calculus}

Let $\ds\rho_1=\inf\{\rho\in\RR_+|A_p\subset M_p^{\rho}\}$ and $\ds\rho_2=\inf\{\rho\in\RR_+|B_p\subset M_p^{\rho}\}$ and put $\rho_0=\max\{\rho_1,\rho_2\}$. Then $0<\rho_0\leq 1$ and for every $\rho$ such that $\rho_0\leq \rho\leq1$, if the larger infimum can be reached, or, otherwise $\rho_0< \rho\leq1$, $A_p\subset M_p^{\rho}$ and $B_p\subset M_p^{\rho}$. So, for every such $\rho$, there exists $c'_0>0$ and $L>0$ (which depend on $\rho$) such that, $A_p\leq c'_0 L^p M_p^{\rho}$, $B_p\leq c'_0 L^p M_p^{\rho}$. Moreover, because $M_p$ tends to infinity, there exists $\tilde{c}>0$ such that $M_p^{\rho}\leq \tilde{c}M_p$, for all such $\rho$. From now on we suppose that $\rho_0\leq\rho\leq 1$, if the larger infimum can be reached, or otherwise $\rho_0<\rho\leq 1$.\\
\indent For $t>0$, put $Q_t=\left\{(x,\xi)\in\RR^{2d}|\langle x\rangle<t, \langle \xi\rangle<t\right\}$ and $Q_t^c=\RR^{2d}\backslash Q_t$. Denote by $FS_{A_p,B_p,\rho}^{M_p,\infty}\left(\RR^{2d};B,h,m\right)$ the vector space of all formal series $\ds \sum_{j=0}^{\infty}a_j(x,\xi)$ such that $a_j\in \mathcal{C}^{\infty}\left(\mathrm{int\,}Q^c_{Bm_j}\right)$, $D^{\alpha}_{\xi} D^{\beta}_x a_j(x,\xi)$ can be extended to continuous function on $Q^c_{Bm_j}$ for all $\alpha,\beta\in\NN^d$ and
\beqs
\sup_{j\in\NN}\sup_{\alpha,\beta}\sup_{(x,\xi)\in Q_{Bm_j}^c}\frac{\left|D^{\alpha}_{\xi}D^{\beta}_x a_j(x,\xi)\right|
\langle (x,\xi)\rangle^{\rho|\alpha|+\rho|\beta|+2j\rho}e^{-M(m|\xi|)}e^{-M(m|x|)}}
{h^{|\alpha|+|\beta|+2j}A_{\alpha}B_{\beta}A_jB_j}<\infty.
\eeqs
In the above, we use the convention $m_0=0$ and hence $Q^c_{Bm_0}=\RR^{2d}$. It is easy to check that $FS_{A_p,B_p,\rho}^{M_p,\infty}\left(\RR^{2d};B,h,m\right)$ is a Banach space. Define
\beqs
FS_{A_p,B_p,\rho}^{(M_p),\infty}\left(\RR^{2d};B,m\right)=\lim_{\substack{\longleftarrow\\h\rightarrow 0}}
FS_{A_p,B_p,\rho}^{M_p,\infty}\left(\RR^{2d};B,h,m\right),\\
FS_{A_p,B_p,\rho}^{(M_p),\infty}\left(\RR^{2d}\right)=\lim_{\substack{\longrightarrow\\B,m\rightarrow\infty}}
FS_{A_p,B_p,\rho}^{(M_p),\infty}\left(\RR^{2d};B,m\right),\\
FS_{A_p,B_p,\rho}^{\{M_p\},\infty}\left(\RR^{2d};B,h\right)=\lim_{\substack{\longleftarrow\\m\rightarrow 0}}
FS_{A_p,B_p,\rho}^{M_p,\infty}\left(\RR^{2d};B,h\right),\\
FS_{A_p,B_p,\rho}^{\{M_p\},\infty}\left(\RR^{2d}\right)=\lim_{\substack{\longrightarrow\\B,h\rightarrow \infty}}
FS_{A_p,B_p,\rho}^{\{M_p\},\infty}\left(\RR^{2d};B,h\right).
\eeqs
Then, $FS_{A_p,B_p,\rho}^{(M_p),\infty}\left(\RR^{2d};B,m\right)$ and $FS_{A_p,B_p,\rho}^{\{M_p\},\infty}\left(\RR^{2d};B,h\right)$ are $(F)$ - spaces. The inclusions $\ds FS_{A_p,B_p,\rho}^{(M_p),\infty}\left(\RR^{2d};B,m\right)\longrightarrow \prod_{j=0}^{\infty}\EE^{(M_p)}\left(\mathrm{int\,}Q^c_{Bm_j}\right)$ and $\ds FS_{A_p,B_p,\rho}^{\{M_p\},\infty}\left(\RR^{2d};B,h\right)\longrightarrow \prod_{j=0}^{\infty}\EE^{\{M_p\}}\left(\mathrm{int\,}Q^c_{Bm_j}\right)$, $\ds \sum_{j=0}^{\infty}a_j\mapsto (a_0,a_1,a_2,...)$, are continuous, so $FS_{A_p,B_p,\rho}^{(M_p),\infty}\left(\RR^{2d}\right)$ and $FS_{A_p,B_p,\rho}^{\{M_p\},\infty}\left(\RR^{2d}\right)$ are Hausdorff l.c.s. Moreover, as inductive limits of barreled and bornological spaces they are barreled and bornological. Note, also, that the inclusions $\Gamma_{A_p,B_p,\rho}^{*,\infty}\left(\RR^{2d}\right)\longrightarrow FS_{A_p,B_p,\rho}^{*,\infty}\left(\RR^{2d}\right)$, defined as $\ds a\mapsto\sum_{j\in\NN}a_j$, where $a_0=a$ and $a_j=0$, $j\geq 1$, is continuous.

\begin{definition}
Two sums, $\ds\sum_{j\in\NN}a_j,\,\sum_{j\in\NN}b_j\in FS_{A_p,B_p,\rho}^{*,\infty}\left(\RR^{2d}\right)$, are said to be equivalent, in notation $\ds\sum_{j\in\NN}a_j\sim\sum_{j\in\NN}b_j$, if there exist $m>0$ and $B>0$, resp. there exist $h>0$ and $B>0$, such that for every $h>0$, resp. for every $m>0$,
\beqs
\sup_{N\in\ZZ_+}\sup_{\alpha,\beta}\sup_{(x,\xi)\in Q_{Bm_N}^c}\frac{\left|D^{\alpha}_{\xi}D^{\beta}_x \sum_{j<N}\left(a_j(x,\xi)-b_j(x,\xi)\right)\right|
\langle (x,\xi)\rangle^{\rho|\alpha|+\rho|\beta|+2N\rho}}
{h^{|\alpha|+|\beta|+2N}A_{\alpha}B_{\beta}A_NB_N}\cdot\\
\cdot e^{-M(m|\xi|)}e^{-M(m|x|)}<\infty.
\eeqs
\end{definition}

From now on, we assume that $A_p$ and $B_p$ satisfy $(M.2)$. Without losing generality we can assume that the constants $c_0$ and $H$ from the condition $(M.2)$ for $A_p$ and $B_p$ are the same as the corresponding constants for $M_p$.

\begin{theorem}
Let $a\in\Gamma_{A_p,B_p,\rho}^{*,\infty}\left(\RR^{2d}\right)$ be such that $a\sim 0$. Then, for every $\tau\in\RR$, $\Op_{\tau}(a)$ is *-regularizing.
\end{theorem}
\begin{proof} First we will prove the following lemma.
\begin{lemma}
Let $0<l\leq1$ and $B>1$. There exists $C>0$ depending on $B$, $l$ and $M_p$ and $\tilde{m}>0$ depending only on $B$ and $M_p$ and not on $l$ such that
\beqs
\inf\left\{\frac{M_n}{l^n \rho^n}\Big| n\in\ZZ_+,\,\rho\geq Bm_n\right\}\leq Ce^{-M(l\tilde{m}\rho)},\mbox{ for all } \rho\geq BM_1.
\eeqs
\end{lemma}
\begin{proof} For shorter notation put
\beqs
f(\rho)=\inf\left\{\frac{M_n}{l^n \rho^n}\Big| n\in\ZZ_+,\,\rho\geq Bm_n\right\}
\eeqs
and $T_{\rho,0}=\{n\in\ZZ_+|\rho\geq Bm_n\}$, $T_{\rho,1}=\{n\in\ZZ_+|\rho< Bm_n\}$. Obviously $T_{\rho,0}\cup T_{\rho,1}=\ZZ_+$ and they are not empty. For $n\in\ZZ_+$, denote by $\ZZ_{+,n}$ the set $\{1,...,n\}$. By the properties of $m_n$, there exists $k\in\ZZ_+$ (which depends on $\rho$) such that $T_{\rho,0}=\{1,2,...,k\}$. In the proof of lemma \ref{69}, we proved that, for $s\in\ZZ_+$, $\ds\frac{m_{k+s+1}}{m_{k+1}}\geq \frac{s}{c_0(k+1)}$. Take $s=2k([c_0]+1)$, and for shorter notation, put $t=2[c_0]+2$. Then $m_{k+kt+1}> m_{k+1}$. For $q\in\ZZ_+$, we get $Bm_{k+kt+q}\geq Bm_{k+kt+1}> Bm_{k+1}\geq l\rho$. Then, for $q\in\ZZ_+$, we have
\beqs
\frac{B^{k+kt+q}M_{k+kt+q}}{l^{k+kt+q}\rho^{k+kt+q}}=
\frac{B^{k+kt+q-1}M_{k+kt+q-1}}{l^{k+kt+q-1}\rho^{k+kt+q-1}}\cdot
\frac{Bm_{k+kt+q}}{l\rho}>\frac{B^{k+kt+q-1}M_{k+kt+q-1}}{l^{k+kt+q-1}\rho^{k+kt+q-1}}.
\eeqs
So, we obtain
\beq\label{130}
e^{-M(l\rho/B)}=\inf_{n\in\NN}\frac{B^nM_n}{l^{n}\rho^n}=\inf_{n\in\ZZ_{+,k+kt}}
\frac{B^nM_n}{l^{n}\rho^n},
\eeq
for $\rho>BM_1/l$ (the infimum can not be obtained for $n=0$). Now, let $0\leq q\leq t$, $q\in\NN$ and $n\in T_{\rho,0}$. One has
\beqs
\frac{B^{n+qk}M_{n+qk}}{l^{n+qk}\rho^{n+qk}}\geq \frac{B^{n}M_{n}}{l^{n}\rho^n}
\left(\frac{B^kM_k}{l^{k}\rho^k}\right)^q\geq f(\rho)^{q+1}\geq f(\rho)^{t+1},
\eeqs
where the last inequality holds because $f(\rho)\leq 1$ when $\rho>BM_1/l$. Hence, by (\ref{130}), $e^{-M(l\rho/B)}\geq f(\rho)^{t+1}$, for $\rho>BM_1/l$. Repeated use of proposition 3.6 of \cite{Komatsu1} yields
\beqs
(t+1)M\left(\frac{l\rho}{BH^{t+1}}\right)\leq 2^{t+1}M\left(\frac{l\rho}{BH^{t+1}}\right)\leq M\left(\frac{l\rho}{B}\right)+\ln c',
\eeqs
i.e. $\ds f(\rho)\leq e^{-\frac{1}{t+1}M(l\rho/B)}\leq C e^{-M(l\tilde{m}\rho)}$, $\forall\rho>BM_1/l$, where we put $\tilde{m}=1/(BH^{t+1})$, which depends only on $B$ and the sequence $M_p$ (recall that $t=2[c_0]+2$). For $BM_1\leq \rho\leq BM_1/l$, $f(\rho)$ is bounded so the same inequality holds, possibly with another $C$.
\end{proof}
\indent We continue the proof of the theorem. It is enough to prove that $a\in\SSS^*$, because then the claim will follow from proposition \ref{6}. Because $a\sim 0$, in the $(M_p)$ case, there exist $m>0$ and $B>0$, such that for every $h>0$ there exists $C>0$, resp. in the $\{M_p\}$ case, there exist $h>0$ and $B>0$, such that for every $m>0$ there exists $C>0$, such that
\beqs
\left|D^{\alpha}_{\xi}D^{\beta}_x a(x,\xi)\right|&\leq& C
\frac{h^{|\alpha|+|\beta|+2N}A_{\alpha}B_{\beta}A_NB_Ne^{M(m|\xi|)}e^{M(m|x|)}}
{\langle (x,\xi)\rangle^{\rho|\alpha|+\rho N}\langle (x,\xi)\rangle^{\rho|\beta|+\rho N}}\\
&\leq& C_1\frac{h^{|\alpha|+|\beta|}A_{\alpha}B_{\beta}e^{M(m|\xi|)}e^{M(m|x|)}}
{\langle (x,\xi)\rangle^{\rho|\alpha|+\rho|\beta|}}\cdot\frac{(hL)^{2N} M^{2\rho}_N}
{\langle (x,\xi)\rangle^{2N\rho}},
\eeqs
for all $N\in\ZZ_+$, $\alpha,\beta\in\NN^d$, $(x,\xi)\in Q_{B m_N}^c$. It is obvious that without losing generality we can assume that $B>1$. In the $(M_p)$ case let $m'>0$ be arbitrary but fixed. Let $(x,\xi)\in Q^c_{Bm_1}$. Then, there exists $N\in\ZZ_+$ such that $(x,\xi)\in Q_{Bm_{N+1}}\backslash Q_{Bm_N}$. We estimate as follows\\
$\ds \frac{m'^{|\alpha|+|\beta|}\left| D^{\alpha}_{\xi} D^{\beta}_x a(x,\xi)\right|e^{M(m'|(x,\xi)|)}}{M_{\alpha+\beta}}$
\beqs
&\leq& C_1\frac{(m'h)^{|\alpha|+|\beta|}A_{\alpha}B_{\beta}e^{M(m|\xi|)}e^{M(m|x|)}e^{M(m'|(x,\xi)|)}}
{\langle (x,\xi)\rangle^{\rho|\alpha|+\rho|\beta|}M_{\alpha+\beta}}\cdot\frac{(hL)^{2N} M^{2\rho}_N}
{\langle (x,\xi)\rangle^{2N\rho}}\\
&\leq& C_2\frac{(m'hL)^{|\alpha|+|\beta|}M^{\rho}_{\alpha+\beta}e^{2M(mBm_{N+1})}e^{M(2m'Bm_{N+1})}}
{\langle (x,\xi)\rangle^{\rho|\alpha|+\rho|\beta|}M_{\alpha+\beta}}\cdot\frac{(hL)^{2N} M^{2\rho}_N}
{(Bm_N)^{2N\rho}}\\
&\leq& C_3(m'hL)^{|\alpha|+|\beta|}(hL)^{2N}e^{2M(mBm_{N+1})}e^{M(2m'Bm_{N+1})},
\eeqs
where, in the last inequality, we used $m^N_N\geq M_N$. By lemma \ref{69}, we have
\beqs
e^{2M(mBm_{N+1})}e^{M(2m'Bm_{N+1})}&\leq& c^3_0 H^{4(c_0mB+2)(N+1)}H^{2(2c_0m'B+2)(N+1)}\\
&=&c^3_0 H^{4(c_0mB+2)}H^{2(2c_0m'B+2)}\left(H^{4(c_0mB+2)}H^{2(2c_0m'B+2)}\right)^N.
\eeqs
Take $h$ small enough such that $m'hL\leq 1$ and $h^2L^2H^{4(c_0mB+2)}H^{2(2c_0m'B+2)}\leq1$. We get
\beqs
\frac{m'^{|\alpha|+|\beta|}\left| D^{\alpha}_{\xi} D^{\beta}_x a(x,\xi)\right|e^{M(m'|(x,\xi)|)}}{M_{\alpha+\beta}}\leq C
\eeqs
for all $\alpha,\beta\in\NN^d$ and $(x,\xi)\in Q^c_{Bm_1}$. For $(x,\xi)\in Q_{Bm_1}$ the same estimate will hold, possibly for another $C>0$, because $a\in\Gamma_{A_p,B_p,\rho}^{(M_p),\infty}\left(\RR^{2d}\right)\subseteq \EE^{(M_p)}\left(\RR^{2d}\right)$ and $Q_{Bm_1}$ is bounded.\\
\indent In the $\{M_p\}$ case, by the above observations, we have
\beqs
\left|D^{\alpha}_{\xi}D^{\beta}_x a(x,\xi)\right|&\leq& C_1\frac{h^{|\alpha|+|\beta|}A_{\alpha}B_{\beta}e^{M(m|\xi|)}e^{M(m|x|)}}
{\langle(x,\xi)\rangle^{\rho|\alpha|+\rho|\beta|}}\\
&{}&\hspace{50 pt}\cdot\left(\inf\left\{\frac{\left(h^{1/\rho}L^{1/\rho}\right)^{N} M_N}
{\langle (x,\xi)\rangle^N}\Big| N\in\ZZ_+,\, (x,\xi)\in Q_{B m_N}^c\right\}\right)^{2\rho}.
\eeqs
and it is obvious that without losing generality we can assume that $h\geq1$ and $L\geq 1$ ($L$ is the constant from $A_p\subset M^{\rho}_p$ and $B_p\subset M^{\rho}_p$). Now, note that\\
$\ds\inf\left\{\frac{\left(h^{1/\rho}L^{1/\rho}\right)^{N} M_N}
{\langle (x,\xi)\rangle^N}\Big| N\in\ZZ_+,\, (x,\xi)\in Q_{B m_N}^c\right\}$
\beqs
&\leq&\inf\left\{\frac{\left(h^{1/\rho}L^{1/\rho}\right)^{N} M_N}
{\langle (x,\xi)\rangle^N}\Big| N\in\ZZ_+,\, \langle (x,\xi)\rangle\geq 2B m_N\right\}\leq C'e^{-M\left(\tilde{m}\langle(x,\xi)\rangle/(hL)^{1/\rho}\right)},
\eeqs
for all $\langle (x,\xi)\rangle\geq 2BM_1$, where in the last inequality we use the above lemma with $l=(hL)^{-1/\rho}\leq 1$. By proposition 3.6 of \cite{Komatsu1}, $e^{M(m|\xi|)}e^{M(m|x|)}\leq c_0e^{M(mH|(x,\xi)|)}$. Using the fact that $A_p\subset M^{\rho}_p$ and $B_p\subset M^{\rho}_p$ and the above inequalities, we get
\beqs
\left|D^{\alpha}_{\xi}D^{\beta}_x a(x,\xi)\right|\leq C_2(L^2h)^{|\alpha|+|\beta|}M_{\alpha+\beta}e^{M(mH|(x,\xi)|)}
e^{-M\left(|(x,\xi)|\tilde{m}/(hL)^{1/\rho}\right)},
\eeqs
for all $\alpha,\beta\in\NN^d$ and $\langle(x,\xi)\rangle\geq 2BM_1$. For $\langle(x,\xi)\rangle\leq 2BM_1$ the same estimate will hold, possibly for another $C>0$ and $\tilde{h}>0$ instead of $L^2h$, because $a\in\Gamma_{A_p,B_p,\rho}^{\{M_p\},\infty}\left(\RR^{2d}\right)\subseteq \EE^{\{M_p\}}\left(\RR^{2d}\right)$ and the set $\left\{(x,\xi)\in\RR^{2d}|\langle(x,\xi)\rangle\leq 2BM_1\right\}$ is bounded. $m$ can be arbitrary small, so if we take $m$ small enough we have $e^{M(mH|(x,\xi)|)}e^{-M\left(|(x,\xi)|\tilde{m}/(hL)^{\lambda/\rho}\right)}\leq C_3e^{-M(m'|(x,\xi)|)}$ for some, small enough, $m'>0$, which completes the proof in the $\{M_p\}$ case.
\end{proof}

\begin{theorem}\label{150}
Let $\ds\sum_{j\in\NN}a_j\in FS_{A_p,B_p,\rho}^{*,\infty}\left(\RR^{2d}\right)$ be given. Than, there exists $a\in\Gamma_{A_p,B_p,\rho}^{*,\infty}\left(\RR^{2d}\right)$, such that $\ds a\sim\sum_{j\in\NN}a_j$.
\end{theorem}
\begin{proof} Define $\varphi(x)\in\DD^{(B_p)}\left(\RR^{d}\right)$ and $\psi(\xi)\in\DD^{(A_p)}\left(\RR^{d}\right)$, in the $(M_p)$ case, resp. $\varphi(x)\in\DD^{\{B_p\}}\left(\RR^{d}\right)$ and $\psi(\xi)\in\DD^{\{A_p\}}\left(\RR^{d}\right)$ in the $\{M_p\}$ case, such that $0\leq\varphi,\psi\leq 1$, $\varphi(x)=1$ when $\langle x\rangle\leq 2$, $\psi(\xi)=1$ when $\langle\xi\rangle\leq 2$ and $\varphi(x)=0$ when $\langle x\rangle\geq 3$, $\psi(\xi)=0$ when $\langle\xi\rangle\geq 3$. Put $\chi(x,\xi)=\varphi(x)\psi(\xi)$, $\ds \chi_n(x,\xi)=\chi\left(\frac{x}{Rm_n},\frac{\xi}{Rm_n}\right)$ for $n\in\ZZ_+$ and $R>0$ and put $\chi_0(x,\xi)=0$. It is easily checked that $\chi,\chi_n\in\DD^{(M_p)}\left(\RR^{2d}\right)$, resp. $\chi,\chi_n\in\DD^{\{M_p\}}\left(\RR^{2d}\right)$.\\
\indent The $(M_p)$ case. Let $m,B>0$ are chosen such that $\sum_j a_j\in FS_{A_p,B_p,\rho}^{M_p,\infty}\left(\RR^{2d};B,h,m\right)$ for all $h>0$. For $R\geq 2B$, $a(x,\xi)=\sum_j\left(1-\chi_j(x,\xi)\right)a_j(x,\xi)$ is a well defined $\mathcal{C}^{\infty}\left(\RR^{2d}\right)$ function. We will prove that for sufficiently large $R$, $a\in\Gamma_{A_p,B_p,\rho}^{*,\infty}\left(\RR^{2d}\right)$ and $a\sim\sum_j a_j(x,\xi)$ which will complete the proof in the $(M_p)$ case. For $0<h<1$, using the fact that $1-\chi_j(x,\xi)=0$ for $(x,\xi)\in Q_{Rm_j}$, we have the estimates\\
$\ds\frac{\left|D^{\alpha}_{\xi}D^{\beta}_x a(x,\xi)\right|
\langle (x,\xi)\rangle^{\rho|\alpha|+\rho|\beta|}e^{-M(m|\xi|)}e^{-M(m|x|)}}
{(8h)^{|\alpha|+|\beta|}A_{\alpha}B_{\beta}}$
\beqs
&\leq&\sum_{j\in\NN}\sum_{\substack{\gamma\leq\alpha\\ \delta\leq\beta}}{\alpha\choose\gamma}{\beta\choose\delta}\left|D^{\alpha-\gamma}_{\xi}D^{\beta-\delta}_x a_j(x,\xi)\right|e^{-M(m|\xi|)}e^{-M(m|x|)}\\
&{}&\hspace{130 pt}\cdot\frac{\left|D^{\gamma}_{\xi}D^{\delta}_x\left(1-\chi_j(x,\xi)\right)\right|
\langle (x,\xi)\rangle^{\rho|\alpha|+\rho|\beta|}}{(8h)^{|\alpha|+|\beta|}A_{\alpha}B_{\beta}}\\
&\leq&C_0\sum_{j\in\NN}\sum_{\substack{\gamma\leq\alpha\\ \delta\leq\beta}}{\alpha\choose\gamma}{\beta\choose\delta}
\frac{h^{|\alpha|-|\gamma|+|\beta|-|\delta|+2j}A_{\alpha-\gamma}B_{\beta-\delta}A_jB_j}
{(8h)^{|\alpha|+|\beta|}A_{\alpha}B_{\beta}}\\
&{}&\hspace{130 pt} \cdot \langle (x,\xi)\rangle^{\rho|\gamma|+\rho|\delta|-2\rho j}
\left|D^{\gamma}_{\xi}D^{\delta}_x\left(1-\chi_j(x,\xi)\right)\right|\\
&\leq&C_0\sum_{j\in\NN}\frac{1}{8^{|\alpha|+|\beta|}}h^{2j}L^{2j}M_j^{2\rho}\left|1-\chi_j(x,\xi)\right|
\langle (x,\xi)\rangle^{-2\rho j}\\
&{}&+C_0\sum_{j\in\NN}\frac{1}{8^{|\alpha|+|\beta|}}\sum_{\substack{\gamma\leq\alpha, \delta\leq\beta\\(\delta,\gamma)\neq(0,0)}}{\alpha\choose\gamma}{\beta\choose\delta}
\frac{h^{2j}L^{2j}M_j^{2\rho}\left|D^{\gamma}_{\xi}D^{\delta}_x\left(1-\chi_j(x,\xi)\right)\right|
\langle (x,\xi)\rangle^{\rho|\gamma|+\rho|\delta|-2\rho j}}{h^{|\gamma|+|\delta|}A_{\gamma}B_{\delta}}\\
&=&S_1+S_2,
\eeqs
where $S_1$ and $S_2$ are the first and the second sum, correspondingly. To estimate $S_1$ note that, on the support of $1-\chi_j$ the inequality $\langle (x,\xi)\rangle\geq Rm_j$ holds. One obtains
\beqs
S_1\leq C_0\sum_{j\in\NN}\frac{(hL)^{2j}M_j^{2\rho}}{R^{2\rho j}m_j^{2\rho j}}\leq C_0\sum_{j\in\NN}\frac{(hL)^{2j}}{R^{2\rho j}}<\infty,
\eeqs
for large enough $R$ (in the second inequality we use the fact that $m_j^j\geq M_j$). For the estimate of $S_2$, note that $D^{\gamma}_{\xi}D^{\delta}_x\left(1-\chi_j(x,\xi)\right)=0$ when $(x,\xi)\in Q_{3Rm_j}^c$, because $(\delta,\gamma)\neq(0,0)$ and $\chi_j(x,\xi)=0$ on $Q_{3Rm_j}^c$. So, for $(x,\xi)\in Q_{3Rm_j}$, we have that $\langle(x,\xi)\rangle\leq \langle x\rangle+\langle \xi\rangle\leq 6Rm_j$. Moreover, from the construction of $\chi$, we have that for the chosen $h$, there exists $C_1>0$ such that $\ds \left|D^{\alpha}_{\xi}D^{\beta}_x \chi(x,\xi)\right|\leq C_1 h^{|\alpha|+|\beta|}A_{\alpha}B_{\beta}$. By using $m_j^j\geq M_j$, one obtains
\beqs
S_2\leq C_2\sum_{j\in\NN}\frac{1}{8^{|\alpha|+|\beta|}}\sum_{\substack{\gamma\leq\alpha, \delta\leq\beta\\(\delta,\gamma)\neq(0,0)}}{\alpha\choose\gamma}{\beta\choose\delta}
\frac{(hL)^{2j}6^{\rho|\gamma|+\rho|\delta|}M_j^{2\rho}(Rm_j)^{\rho|\gamma|+\rho|\delta|}}
{R^{2\rho j}m_j^{2\rho j}(Rm_j)^{|\gamma|+|\delta|}}
\leq C_3\sum_{j\in\NN}\frac{(hL)^{2j}}{R^{2\rho j}},
\eeqs
which is convergent for large enough $R$. Hence, we get that $a\in \Gamma_{A_p,B_p,\rho}^{M_p,\infty}\left(\RR^{2d};8h,m\right)$ for all $0<h<1$, from what we obtain $a\in \Gamma_{A_p,B_p,\rho}^{(M_p),\infty}\left(\RR^{2d}\right)$. Now, to prove that $\ds a\sim\sum_{j\in\NN}a_j(x,\xi)$. Note that, for $(x,\xi)\in Q_{3Rm_N}^c$, $\ds a-\sum_{j<N}a_j=\sum_{j\geq N}\left(1-\chi_j\right)a_j$. This easily follows from the definition of $\chi_j$ and the fact that $m_n$ is monotonically increasing.\\
$\ds \frac{\left|D^{\alpha}_{\xi}D^{\beta}_x\sum_{j\geq N}\left(1-\chi_j(x,\xi)\right)a_j(x,\xi)\right|
\langle (x, \xi)\rangle^{\rho|\alpha|+\rho|\beta|+2\rho N}e^{-M(m|\xi|)}e^{-M(m|x|)}}
{(8(1+H)h)^{|\alpha|+|\beta|+2N}A_{\alpha}B_{\beta}A_NB_N}$
\beqs
&\leq&\sum_{j\geq N}\frac{\left(1-\chi_j(x,\xi)\right)\left|D^{\alpha}_{\xi}D^{\beta}_xa_j(x,\xi)\right|\langle (x, \xi)\rangle^{\rho|\alpha|+\rho|\beta|+2\rho N}e^{-M(m|\xi|)}e^{-M(m|x|)}}
{(8(1+H)h)^{|\alpha|+|\beta|+2N}A_{\alpha}B_{\beta}A_NB_N}\\
&{}&+\sum_{j\geq N}\sum_{\substack{\gamma\leq\alpha,\delta\leq\beta\\(\delta,\gamma)\neq(0,0)}}
{\alpha\choose\gamma}{\beta\choose\delta}
\left|D^{\alpha-\gamma}_{\xi}D^{\beta-\delta}_xa_j(x,\xi)\right|
e^{-M(m|\xi|)}e^{-M(m|x|)}\\
&{}&\hspace{130 pt}\cdot\frac{\left|D^{\gamma}_{\xi}D^{\delta}_x\left(1-\chi_j(x,\xi)\right)\right|
\langle (x, \xi)\rangle^{\rho|\alpha|+\rho|\beta|+2\rho N}}{(8(1+H)h)^{|\alpha|+|\beta|+2N}A_{\alpha}B_{\beta}A_NB_N}\\
&\leq&C_0\sum_{j\geq N}\frac{\left(1-\chi_j(x,\xi)\right)h^{2j-2N}A_jB_j}
{(1+H)^{2N}\langle (x,\xi)\rangle^{2\rho j-2\rho N}A_NB_N}\\
&{}&+C_0\sum_{j\geq N}\frac{1}{8^{|\alpha|+|\beta|}}
\sum_{\substack{\gamma\leq\alpha,\delta\leq\beta\\(\delta,\gamma)\neq(0,0)}}
{\alpha\choose\gamma}{\beta\choose\delta}
\frac{h^{2j-2N}\left|D^{\gamma}_{\xi}D^{\delta}_x\left(1-\chi_j(x,\xi)\right)\right|
\langle (x,\xi)\rangle^{\rho|\gamma|+\rho|\delta|}A_jB_j}
{(1+H)^{2N}h^{|\gamma|+|\delta|}\langle (x,\xi)\rangle^{2\rho j-2\rho N}A_{\gamma}B_{\delta}A_NB_N}\\
&=&S_1+S_2,
\eeqs
where $S_1$ and $S_2$ are the first and the second sum, correspondingly. To estimate $S_1$, observe that on the support of $1-\chi_j$ the inequality $\langle(x,\xi)\rangle\geq Rm_j$ holds. Using the monotone increasingness of $m_n$ and $(M.2)$ for $A_p$ and $B_p$, one obtains
\beqs
S_1&\leq& C'_0\sum_{j\geq N}\frac{h^{2j-2N}H^{2j}A_{j-N}B_{j-N}}
{(1+H)^{2N}R^{2\rho j-2\rho N}m_j^{2\rho j-2\rho N}}
\leq C_4\sum_{j\geq N}\frac{h^{2j-2N}H^{2j}L^{2j-2N}M_{j-N}^{2\rho}}
{(1+H)^{2N}R^{2\rho j-2\rho N}m_{j-N}^{2\rho j-2\rho N}}\\
&=&C_4\frac{H^{2N}}{(1+H)^{2N}}\sum_{j=0}^{\infty}\left(\frac{hHL}{R^{\rho}}\right)^{2j}
\leq C_4\sum_{j=0}^{\infty}\left(\frac{hHL}{R^{\rho}}\right)^{2j}<\infty,
\eeqs
uniformly, for $N\in\ZZ_+$, for large enough $R$. For $S_2$, note that $D^{\gamma}_{\xi}D^{\delta}_x\left(1-\chi_j(x,\xi)\right)=0$ when $(x,\xi)\in Q_{3Rm_j}^c$, because $(\delta,\gamma)\neq(0,0)$ and $\chi_j(x,\xi)=0$ on $Q_{3Rm_j}^c$. Moreover, from the construction of $\chi$, we have that for the chosen $h$, there exists $C_1>0$ such that $\ds \left|D^{\alpha}_{\xi}D^{\beta}_x \chi(x,\xi)\right|\leq C_1 h^{|\alpha|+|\beta|}A_{\alpha}B_{\beta}$. Now
\beqs
S_2&\leq& C_5\sum_{j\geq N}\frac{1}{8^{|\alpha|+|\beta|}}\sum_{\substack{\gamma\leq\alpha,\delta\leq\beta\\(\delta,\gamma)\neq(0,0)}}
{\alpha\choose\gamma}{\beta\choose\delta}
\frac{h^{2j-2N}6^{|\gamma|+|\delta|}H^{2j}A_{j-N}B_{j-N}}{(1+H)^{2N}R^{2\rho j-2\rho N}m_j^{2\rho j-2\rho N}}\\
&\leq& C_6\sum_{j\geq N}\frac{h^{2j-2N}H^{2j}A_{j-N}B_{j-N}}{(1+H)^{2N}R^{2\rho j-2\rho N}m_j^{2\rho j-2\rho N}},
\eeqs
which we already proved that is bounded uniformly for $N\in\ZZ_+$. Hence, we obtained
\beqs
&{}&\sup_{N\in\ZZ_+}\sup_{\alpha,\beta}\sup_{(x,\xi)\in Q_{3Rm_N}^c}\left|D^{\alpha}_{\xi}D^{\beta}_x\sum_{j\geq N}\left(1-\chi_j(x,\xi)\right)a_j(x,\xi)\right|\\
&{}&\hspace{170 pt}\cdot\frac{\langle (x,\xi)\rangle^{\rho |\alpha|+\rho|\beta|+2\rho N}e^{-M(m|\xi|)}e^{-M(m|x|)}}
{(8(1+H)h)^{|\alpha|+|\beta|+2N}A_{\alpha}B_{\beta}A_NB_N}<\infty,
\eeqs
for arbitrary $h>0$, i.e. $\ds a\sim\sum_{j\in\NN}a_j(x,\xi)$. For the $\{M_p\}$ case, let $h,B>0$ are such that $a\in FS_{A_p,B_p,\rho}^{M_p,\infty}\left(\RR^{2d};B,h,m\right)$ for all $m>0$. Then, for $R\geq 2B$ we define $\ds a(x,\xi)=\sum_{j\in\NN}\left(1-\chi_j(x,\xi)\right)a_j(x,\xi)$ and similarly as above, one proves that, for sufficiently large $R$, $a$ satisfies the claim in the theorem.
\end{proof}

Now we will prove theorems for change of quantization and composition of operators. Note that, unlike in \cite{C3} and \cite{C4}, we do not impose additional conditions on $A_p$ and $B_p$ in the composition theorem.

\begin{theorem}\label{200}
Let $\tau,\tau_1\in\RR$ and $a\in\Gamma_{A_p,B_p,\rho}^{*,\infty}\left(\RR^{2d}\right)$. There exists $b\in\Gamma_{A_p,B_p,\rho}^{*,\infty}\left(\RR^{2d}\right)$ and *-regularizing operator $T$ such that $\Op_{\tau_1}(a)=\Op_{\tau}(b)+T$. Moreover,
\beqs
b(x,\xi)\sim\sum_{\beta}\frac{1}{\beta!}(\tau_1-\tau)^{|\beta|}\partial^{\beta}_{\xi}D^{\beta}_xa(x,\xi), \mbox{ in } FS_{A_p,B_p,\rho}^{*,\infty}\left(\RR^{2d}\right).
\eeqs
\end{theorem}

\begin{proof} Put $\ds p_j(x,\xi)=\sum_{|\beta|=j}\frac{1}{\beta!}(\tau_1-\tau)^{|\beta|}\partial^{\beta}_{\xi}D^{\beta}_xa(x,\xi)$. One easily verifies that $\sum_j p_j\in FS_{A_p,B_p,\rho}^{*,\infty}\left(\RR^{2d}\right)$. Take the sequence $\chi_j(x,\xi)$, $j\in\NN$, constructed in the proof of theorem \ref{150}, such that $b=\sum_j (1-\chi_j)p_j$ is an element of $\Gamma_{A_p,B_p,\rho}^{*,\infty}\left(\RR^{2d}\right)$ and $b\sim\sum_j p_j$. By the observations after theorem \ref{npr}, the operators $\Op_{\tau_1}(a)$ and $\Op_{\tau}(b)$ coincide with the operators $A$ and $B$ corresponding to $a$ and $b$ when we observe $a((1-\tau_1)x+\tau_1 y,\xi)$ and $b((1-\tau)x+\tau y,\xi)$ as elements of $\Pi_{A_p,B_p,\rho}^{*,\infty}\left(\RR^{3d}\right)$. It is clear that it is enough to prove that the kernel of $A-B$ is in $\SSS^*\left(\RR^{2d}\right)$. To prove that, write\\
$\ds a((1-\tau_1)x+\tau_1 y,\xi)-b((1-\tau)x+\tau y,\xi)$
\beqs
=\chi_0((1-\tau)x+\tau y,\xi)a((1-\tau_1)x+\tau_1 y,\xi)+\sum_{n=0}^{\infty}
\left((\chi_{n+1}-\chi_n)((1-\tau)x+\tau y,\xi)\right)\\
\cdot\left(a((1-\tau_1)x+\tau_1 y,\xi)-\sum_{j=0}^n p_j((1-\tau)x+\tau y,\xi)\right).
\eeqs
By construction $\chi_0=0$, so $\chi_0a=0$. Note that the above sum is locally finite and it converges in $\EE^*\left(\RR^{3d}\right)$. Denote by $A_n$ the operator corresponding to
\beqs
a_n(x,y,\xi)&=&\left(\chi_{n+1}-\chi_n\right)((1-\tau)x+\tau y,\xi)\\
&{}&\hspace{110 pt}\cdot\left(a((1-\tau_1)x+\tau_1 y,\xi)-\sum_{j=0}^n p_j((1-\tau)x+\tau y,\xi)\right)
\eeqs
considered as an element of $\Pi_{A_p,B_p,\rho}^{*,\infty}\left(\RR^{3d}\right)$. For $u\in\SSS^*\left(\RR^d\right)$, we obtain\\
$\ds Au(x)-Bu(x)$
\beqs
=\frac{1}{(2\pi)^d}\int_{\RR^{2d}}e^{i(x-y)\xi}
\frac{1}{P_l(\xi)}P_l(D_y)\left(\frac{1}{P_l(y-x)}P_l(D_\xi)
\left(\sum_{n=0}^{\infty}a_n(x,y,\xi)u(y)\right)\right)dyd\xi,
\eeqs
in the $(M_p)$ case and the same but with $P_{l_p}$ in place of $P_l$ in the $\{M_p\}$ case. Note that, because of the convergence of the sum in $\EE^*\left(\RR^{3d}\right)$, we can interchange the sum with the ultradifferential operators and with $1/P_l(y-x)$ and $1/P_l(\xi)$, resp. with $1/P_{l_p}(y-x)$ and $1/P_{l_p}(\xi)$. For $v\in\SSS^*\left(\RR^d\right)$, by the way we define $p_j$ and using the fact about the support of $\chi_n$, with similar technic as in the proof of lemma \ref{45}, one proves that
\beqs
\sum_{n=0}^{\infty}\int_{\RR^{3d}}\left|
\frac{1}{P_l(\xi)}P_l(D_y)\left(\frac{1}{P_l(y-x)}P_l(D_\xi)
\left(a_n(x,y,\xi)u(y)\right)\right)v(x)\right|dyd\xi dx<\infty,
\eeqs
for sufficiently small $l$ and sufficiently large $R$ (from the definition of $\chi_n$) in the $(M_p)$ case, resp. the same but with $P_{l_p}$ in place of $P_l$ for sufficiently small $(l_p)\in\mathfrak{R}$ and sufficiently large $R$ (from the definition of $\chi_n$) in the $\{M_p\}$ case. Hence, from monotone and dominated convergence it follows that\\
$\langle Au-Bu,v\rangle$
\beqs
&=&\frac{1}{(2\pi)^d}\sum_{n=0}^{\infty}\int_{\RR^{3d}}e^{i(x-y)\xi}
\frac{1}{P_l(\xi)}P_l(D_y)\left(\frac{1}{P_l(y-x)}P_l(D_\xi)
\left(a_n(x,y,\xi)u(y)\right)\right)v(x)dyd\xi dx\\
&=&\frac{1}{(2\pi)^d}\sum_{n=0}^{\infty}\int_{\RR^{3d}}e^{i(x-y)\xi}a_n(x,y,\xi)u(y)v(x)dyd\xi dx=\sum_{n=0}^{\infty}\langle A_nu,v\rangle
\eeqs
in the $(M_p)$ case, resp. the same but with $P_{l_p}$ in place of $P_l$ in the $\{M_p\}$ case. Hence, $\ds \sum_{k=0}^n A_ku \longrightarrow Au-Bu$, when $n\longrightarrow \infty$ in $\SSS'^*\left(\RR^d\right)$ for every fixed $u\in\SSS^*\left(\RR^d\right)$. But then, because $\SSS^*$ is barreled, by the Banach - Steinhaus theorem (see \cite{Schaefer}, theorem 4.6), $\ds \sum_{k=0}^n A_k\longrightarrow Au-Bu$, when $n\longrightarrow\infty$ in the topology of precompact convergence in $\mathcal{L}\left(\SSS^*\left(\RR^d\right), \SSS'^*\left(\RR^d\right)\right)$. $\SSS^*$ is Montel, hence the convergence holds in $\mathcal{L}_b\left(\SSS^*\left(\RR^d\right), \SSS'^*\left(\RR^d\right)\right)$. If we denote by $K$ and $K_n$, $n\in\NN$, the kernels of the operators $A-B$ and $A_n$, $n\in\NN$ correspondingly, then, by proposition \ref{ktr}, it follows that $\ds K=\sum_{n=0}^{\infty}K_n$, where the convergence is in $\SSS'^*\left(\RR^{2d}\right)$. Let $r=1/(8(1+|\tau|+|\tau_1|))$. Take $\theta\in\EE^*\left(\RR^{2d}\right)$ as in lemma \ref{pll} and put $\tilde{\theta}=1-\theta$. $\theta$ and $\tilde{\theta}$ are obviously multipliers for $\SSS'^*$. By proposition \ref{72} and the properties of $\theta$, $\theta K\in\SSS^*\left(\RR^{2d}\right)$. It is enough to prove that $\tilde{\theta} K\in \SSS^*\left(\RR^{2d}\right)$. Note that $\tilde{\theta} K=\sum_n\tilde{\theta} K_n$. Our goal is to prove that $\sum_n \tilde{\theta}K_n\in\SSS^*$. Observe that
\beqs
K_n(x,y)&=&\frac{1}{(2\pi)^d}\int_{\RR^d}e^{i(x-y)\xi}\left(\chi_{n+1}-\chi_n\right)((1-\tau)x+\tau y,\xi)\\
&{}&\hspace{90 pt}\cdot\left(a((1-\tau_1)x+\tau_1 y,\xi)-\sum_{j=0}^n p_j((1-\tau)x+\tau y,\xi)\right)d\xi,
\eeqs
for all $n\in\NN$. Put
$\ds
\left\{\begin{array}{ll}
x'=(1-\tau)x+\tau y,\\
y'=x-y,
\end{array}\right.
$
from what we obtain
$\ds
\left\{\begin{array}{ll}
x=x'+\tau y',\\
y=x'-(1-\tau)y'.
\end{array}\right.
$
Hence $a((1-\tau_1)x+\tau_1 y,\xi)=a(x'+(\tau-\tau_1) y',\xi)$. If we Taylor expand the right hand side in $y'=0$, we get
\beqs
a((1-\tau_1)x+\tau_1 y,\xi)=\sum_{|\beta|\leq n}
\frac{1}{\beta!}(\tau-\tau_1)^{|\beta|}\partial^{\beta}_x a(x',\xi)(x-y)^{\beta}+W_{n+1}(x,y,\xi),
\eeqs
where $W_{n+1}$ is the reminder of the expansion:
\beqs
W_{n+1}(x,y,\xi)=(n+1)\sum_{|\beta|=n+1}\frac{1}{\beta!}(x-y)^{\beta}
(\tau-\tau_1)^{|\beta|}\int_0^1(1-t)^n \partial^{\beta}_x a(x'+t(\tau-\tau_1)y',\xi)dt.
\eeqs
If we insert the above expression for $a$ in the expression for $K_n$ we obtain
\beqs
K_n(x,y)&=&\frac{1}{(2\pi)^d}\sum_{|\beta|\leq n} \frac{1}{\beta!}(\tau-\tau_1)^{|\beta|}
\int_{\RR^d}e^{i(x-y)\xi}(-D_{\xi})^{\beta}\left(\left(\chi_{n+1}-\chi_n\right)(x',\xi)\partial^{\beta}_x a(x',\xi)\right)d\xi\\
&{}&+\frac{1}{(2\pi)^d}\int_{\RR^d}e^{i(x-y)\xi}\left(\chi_{n+1}-\chi_n\right)(x',\xi)W_{n+1}(x,y,\xi)d\xi\\
&{}&-\frac{1}{(2\pi)^d}\sum_{j=0}^n\int_{\RR^d}e^{i(x-y)\xi}\left(\chi_{n+1}-\chi_n\right)(x',\xi)p_j((1-\tau)x+\tau y,\xi)d\xi\\
&=&S_{1,n}(x,y)+S_{2,n}(x,y)-S_{3,n}(x,y).
\eeqs
Our goal is to prove that each of the sums $\sum_n \tilde{\theta}(S_{1,n}-S_{3,n})$ and $\sum_n \tilde{\theta}S_{2,n}$, is $\SSS^*$ function. Because of the way we defined $p_j$, one obtains\\
$\ds S_{1,n}(x,y)-S_{3,n}(x,y)$
\beqs
=\frac{1}{(2\pi)^d}\sum_{0\neq|\beta|\leq n}\sum_{0\neq\delta\leq\beta}
{\beta\choose\delta}\frac{1}{\beta!}(\tau_1-\tau)^{|\beta|} \int_{\RR^d}e^{i(x-y)\xi}\left(D^{\delta}_{\xi}\left(\chi_{n+1}-\chi_n\right)\right)(x',\xi)
D^{\beta-\delta}_{\xi}\partial^{\beta}_x a(x',\xi)d\xi.
\eeqs
Put
\beqs
\tilde{S}_{\beta,n}(x,y)&=&\frac{1}{(2\pi)^d}\sum_{0\neq\delta\leq\beta}
{\beta\choose\delta}\frac{1}{\beta!}(\tau_1-\tau)^{|\beta|}\\
&{}&\hspace{50 pt}\cdot\int_{\RR^d}e^{i(x-y)\xi}\left(D^{\delta}_{\xi}\left(\chi_{n+1}-\chi_n\right)\right)(x',\xi)
D^{\beta-\delta}_{\xi}\partial^{\beta}_x a(x',\xi)d\xi.
\eeqs
Obviously $\tilde{S}_{\beta,n}\in\EE^*\left(\RR^{2d}\right)\cap\SSS'^*\left(\RR^{2d}\right)$. Let $w\in\SSS^*\left(\RR^{2d}\right)$. Note that
\beqs
\langle \tilde{S}_{\beta,n}, w\rangle&=&\frac{1}{(2\pi)^d}\sum_{0\neq\delta\leq\beta}
{\beta\choose\delta}\frac{1}{\beta!}(\tau_1-\tau)^{|\beta|} \int_{\RR^{3d}}\frac{1}{P_l(\xi)}e^{i(x-y)\xi}\\
&{}&\,\cdot P_l(D_y)\left(\left(D^{\delta}_{\xi}\left(\chi_{n+1}- \chi_n\right)\right)(x',\xi)D^{\beta-\delta}_{\xi}\partial^{\beta}_x a(x',\xi)w(x,y)\right)d\xi dxdy,
\eeqs
in the $(M_p)$ case, where $l>0$ will be chosen later, resp. the same but with $P_{l_p}$ in place of $P_l$ in the $\{M_p\}$ case, where $(l_p)\in\mathfrak{R}$ will be chosen later. We will consider first the $(M_p)$ case. Then there exists $m>0$ such that $a\in\Gamma_{A_p,B_p,\rho}^{(M_p),\infty}\left(\RR^{2d};m\right)$. Chose $l$ such that $\left|P_l(\xi)\right|\geq c'e^{4M(m|\xi|)}$ (cf. proposition \ref{orn}). On the other hand $P_l(\xi)=\sum_{\alpha}c_{\alpha} \xi^{\alpha}$ and there exist $C_0>0$ and $L_0>0$ such that $|c_{\alpha}|\leq C_0L_0^{|\alpha|}/M_{\alpha}$. Note that, when $\left(\chi_{n+1}-\chi_n\right)(x',\xi)\neq0$, $\langle (x',\xi)\rangle\geq Rm_n$. Using this, one easily obtains that
\beqs
&{}&\sum_{n=1}^{\infty}\sum_{0\neq|\beta|\leq n}\sum_{0\neq\delta\leq\beta}{\beta\choose\delta} \frac{(|\tau_1|+|\tau|)^{|\beta|}}{\beta!}\\
&{}&\quad\cdot\int_{\RR^{3d}} \left|\frac{e^{i(x-y)\xi}}{P_l(\xi)} P_l(D_y)\left(\left(D^{\delta}_{\xi}\left(\chi_{n+1}- \chi_n\right)\right)(x',\xi)D^{\beta-\delta}_{\xi}\partial^{\beta}_x a(x',\xi)w(x,y)\right)\right|d\xi dxdy<\infty,
\eeqs
for sufficiently large $R$ (from the definition of $\chi_n$). In the $\{M_p\}$ case, by lemma \ref{nn1} there exists $(k_p)\in\mathfrak{R}$ such that the estimate in that lemma holds (we can regard $a((1-\tau)x+\tau y,\xi)$ as an element of $\Pi_{A_p,B_p,\rho}^{\{M_p\},\infty}\left(\RR^{3d};h\right)$). Take $(l_p)\in\mathfrak{R}$ such that $\left|P_{l_p}(\xi)\right|\geq c' e^{4N_{k_p}(|\xi|)}$. One obtains the same estimate as above but with $P_{l_p}$ in place of $P_l$, for sufficiently large $R$ (from the definition of $\chi_n$). From this we obtain that $\ds\sum_{n=1}^{\infty}(S_{1,n}-S_{3,n})=\sum_{n=1}^{\infty}\sum_{0\neq|\beta|\leq n}\tilde{S}_{\beta,n}$ converges in $\SSS'^*\left(\RR^{2d}\right)$. Denote its limit by $\tilde{S}(x,y)$. Moreover, from the  above, we can change the order of summation and integration. The local finiteness of $\sum_n(\chi_{n+1}-\chi_n)$ implies
\beqs
\sum_{n\geq|\beta|}
D^{\delta}_{\xi}(\chi_{n+1}(x',\xi)-\chi_n(x',\xi))=D^{\delta}_{\xi}(1-\chi_{|\beta|}(x',\xi))=
-D^{\delta}_{\xi}\chi_{|\beta|}(x',\xi),
\eeqs
where the last equality follows from the fact that $\delta\neq 0$. In the $(M_p)$ case, we obtain\\
$\ds \sum_{n=1}^{\infty}\sum_{0\neq|\beta|\leq n}\langle \tilde{S}_{\beta,n}, w\rangle$
\beqs
&=& -\frac{1}{(2\pi)^d}\sum_{|\beta|=1}^{\infty}\sum_{0\neq\delta\leq\beta}
{\beta\choose\delta}\frac{1}{\beta!}(\tau_1-\tau)^{|\beta|}\\
&{}&\quad\cdot\int_{\RR^{3d}} \frac{1}{P_l(\xi)}e^{i(x-y)\xi} P_l(D_y)\left(D^{\delta}_{\xi}\chi_{|\beta|}(x',\xi)D^{\beta-\delta}_{\xi}\partial^{\beta}_x a(x',\xi)w(x,y)\right)d\xi dxdy\\
&=&-\frac{1}{(2\pi)^d}\sum_{|\beta|=1}^{\infty}\sum_{0\neq\delta\leq\beta}
{\beta\choose\delta}\frac{1}{\beta!}(\tau_1-\tau)^{|\beta|} \int_{\RR^{2d}}I_{\beta,\delta}(x,y)w(x,y)dxdy,
\eeqs
where we put $\ds I_{\beta,\delta}(x,y)=\int_{\RR^d}e^{i(x-y)\xi}D^{\delta}_{\xi}\chi_{|\beta|}(x',\xi) D^{\beta-\delta}_{\xi}\partial^{\beta}_x a(x',\xi)d\xi$. Similarly, in the $\{M_p\}$ case we obtain the same equality. Hence $\ds-\frac{1}{(2\pi)^d}\sum_{|\beta|=1}^{\infty}\sum_{0\neq\delta\leq\beta}
{\beta\choose\delta}\frac{1}{\beta!}(\tau_1-\tau)^{|\beta|} I_{\beta,\delta}(x,y)$ converges to $\tilde{S}(x,y)$ in $\SSS'^*\left(\RR^{2d}\right)$. Now we will prove that $\tilde{\theta}\tilde{S}$ is $\SSS^*$ function. Denote
\beq\label{zao}
T_n=\left\{(x,\xi)\in\RR^{2d}||x|\leq 3Rm_n\mbox{ and }|\xi|\leq 3Rm_n\right\}
\eeq
and put $T_{\xi,n}$ to be the projection of $T_n$ on $\RR^d_{\xi}$. By construction $\supp\chi_{|\beta|}\subseteq T_{|\beta|}$. So, for the derivatives of $I_{\beta,\delta}(x,y)$ when $(x,y)\in \RR^{2d}\backslash\Omega_{r}\supseteq\mathrm{supp\,}\tilde{\theta}$, we have\\
$\ds\left|D^{\beta'}_xD^{\gamma'}_yI_{\beta,\delta}(x,y)\right|$
\beqs
&\leq&\sum_{\substack{\alpha\leq\beta'\\ \nu\leq\gamma'}}\sum_{\substack{\alpha'+\alpha''=\alpha\\ \nu'+\nu''=\nu}}
{\gamma'\choose\nu}{\beta'\choose\alpha}{\alpha\choose\alpha'}{\nu\choose\nu'} (1+|\tau|)^{|\alpha'|+|\nu'|}(1+|\tau|)^{|\beta'|-|\alpha|+|\gamma'|-|\nu|}\\
&{}&\hspace{30 pt}\cdot\int_{T_{\xi,|\beta|}}
|\xi|^{|\alpha''|+|\nu''|}\left|D^{\delta}_{\xi}D^{\alpha'+\nu'}_x\chi_{|\beta|}(x',\xi)\right|
\left|D^{\beta-\delta}_{\xi} D^{\beta+\gamma'-\nu+\beta'-\alpha}_x a(x',\xi)\right|d\xi\\
&\leq&C_1\sum_{\substack{\alpha\leq\beta'\\ \nu\leq\gamma'}}\sum_{\substack{\alpha'+\alpha''=\alpha\\ \nu'+\nu''=\nu}}
{\gamma'\choose\nu}{\beta'\choose\alpha}{\alpha\choose\alpha'}{\nu\choose\nu'} (1+|\tau|)^{|\beta'|+|\gamma'|-|\alpha''|-|\nu''|}\\
&{}&\hspace{30 pt}\cdot\int_{T_{\xi,|\beta|}}|\xi|^{|\alpha''|+|\nu''|}
\frac{h_1^{|\delta|+|\alpha'|+|\nu'|}A_{\delta}B_{\alpha'+\nu'}}
{(Rm_{|\beta|})^{|\delta|+|\alpha'|+|\nu'|}}\\
&{}&\hspace{30 pt}\cdot\frac{h^{|2\beta-\delta+\beta'+\gamma'-\alpha-\nu|}
A_{\beta-\delta}B_{\beta+\beta'+\gamma'-\alpha-\nu}e^{M(m|\xi|)}e^{M(m|x'|)}}
{\langle (x',\xi)\rangle^{\rho|2\beta-\delta+\beta'+\gamma'-\alpha-\nu|}}d\xi.
\eeqs
Because $\delta\neq 0$, $D^{\delta}_{\xi}D^{\alpha'+\nu'}_x\chi_{|\beta|}(x',\xi)=0$ when $\chi_{|\beta|}(x',\xi)=1$, hence when $|x'|\leq Rm_{|\beta|}$ and $|\xi|\leq Rm_{|\beta|}$. So, when $D^{\delta}_{\xi}D^{\alpha'+\nu'}_x\chi_{|\beta|}(x',\xi)\neq0$ we have $\langle(x',\xi)\rangle \geq Rm_{|\beta|}$. We obtain
\beqs
R^{|\delta|+|\alpha'|+|\nu'|}m_{|\beta|}^{|\delta|+|\alpha'|+|\nu'|}\langle (x',\xi)\rangle^{\rho|2\beta-\delta+\beta'+\gamma'-\alpha-\nu|}\geq
\left(Rm_{|\beta|}\right)^{\rho|2\beta+\beta'+\gamma'-\alpha''-\nu''|}.
\eeqs
By assumption, there exists $c,L\geq 1$ such that $A_p\leq cL^p M_p^{\rho}$ and $B_p\leq cL^p M_p^{\rho}$. Hence\\
$\ds\frac{A_{\delta}B_{\alpha'+\nu'}A_{\beta-\delta}B_{\beta+\beta'+\gamma'-\alpha-\nu}}
{\left(Rm_{|\beta|}\right)^{\rho|2\beta+\beta'+\gamma'-\alpha''-\nu''|}}$
\beqs
&\leq& \frac{A_{\beta}B_{\beta+\beta'+\gamma'-\alpha''-\nu''}}
{\left(Rm_{|\beta|}\right)^{\rho|2\beta+\beta'+\gamma'-\alpha''-\nu''|}}\leq
\frac{c^2L^{|2\beta+\beta'+\gamma'-\alpha''-\nu''|}M_{2\beta+\beta'+\gamma'-\alpha''-\nu''}^{\rho}}
{\left(Rm_{|\beta|}\right)^{2\rho|\beta|}
\left(Rm_{|\beta|}\right)^{\rho|\beta'+\gamma'-\alpha''-\nu''|}}\\
&\leq& \frac{C''(LH^2)^{|2\beta+\beta'+\gamma'-\alpha''-\nu''|}M_{\beta}^{2\rho}M_{\beta'+\gamma'}}
{R^{2\rho|\beta|}m_{|\beta|}^{2\rho|\beta|}\left(RM_1\right)^{\rho|\beta'+\gamma'-\alpha''-\nu''|}
M_{\alpha''+\nu''}}
\leq \frac{C''(LH^2)^{|2\beta+\beta'+\gamma'-\alpha''-\nu''|}M_{\beta'+\gamma'}}
{R^{2\rho|\beta|}\left(RM_1\right)^{\rho|\beta'+\gamma'-\alpha''-\nu''|}M_{\alpha''+\nu''}},
\eeqs
where, in the last inequality, we used that $m_n^n\geq M_n$. Also, note that when $(x,y)\in\RR^{2d}\backslash\Omega_r$ and $\chi_{|\beta|}((1-\tau)x+\tau y,\xi)\neq 0$, we have the following inequalities
\beqs
|x'|&=&|(1-\tau)x+\tau y|\leq 3R m_{|\beta|},\\
|x|^2+|y|^2&\leq& 2|x|^2+|x-y|^2+2|x||x-y|\leq 2|x|^2+r^2\langle x\rangle^2+2r|x|\langle x\rangle\\
&\leq& (2+r)^2\langle x\rangle^2,\\
1+\left|(1-\tau)x+\tau y\right|^2&\geq& 1+|x|^2+|\tau|^2|x-y|^2-2|\tau||x||x-y|\geq \langle x\rangle^2-\frac{1}{4}\langle x\rangle^2\geq \frac{1}{4}\langle x\rangle^2,
\eeqs
(remember, $r=1/(8(1+|\tau|+|\tau_1|))$). Put $s=2+r$ for shorter notation. Combining these inequalities we get $|(x,y)|\leq 2s\langle x'\rangle\leq 8sRm_{|\beta|}$. Using this and proposition 3.6 of \cite{Komatsu1}, for arbitrary $m'>0$, we obtain
\beqs
e^{M(m|x'|)}\leq e^{M(3mRm_{|\beta|})}e^{M(8sm'Rm_{|\beta|})}e^{-M(m'|(x,y)|)}\leq c_0e^{M\left(8s(m+m')HRm_{|\beta|}\right)}e^{-M(m'|(x,y)|)},
\eeqs
when $(x,y)\in\RR^{2d}\backslash\Omega_r$ and $\chi_{|\beta|}((1-\tau)x+\tau y,\xi)\neq 0$. Using these inequalities in the estimate for $D^{\beta'}_xD^{\gamma'}_yI_{\beta,\delta}(x,y)$, for $(x,y)\in \RR^{2d}\backslash\Omega_r$, we get\\
$\ds \left|D^{\beta'}_xD^{\gamma'}_yI_{\beta,\delta}(x,y)\right|$
\beqs
&\leq&C_3\sum_{\substack{\alpha\leq\beta'\\ \nu\leq\gamma'}}\sum_{\substack{\alpha'+\alpha''=\alpha\\ \nu'+\nu''=\nu}}
{\gamma'\choose\nu}{\beta'\choose\alpha}{\alpha\choose\alpha'}{\nu\choose\nu'} (1+|\tau|)^{|\beta'|+|\gamma'|-|\alpha''|-|\nu''|}\\
&{}&\hspace{5 pt}\cdot\int_{T_{\xi,|\beta|}}|\xi|^{|\alpha''|+|\nu''|}
\frac{(LH^2)^{|2\beta+\beta'+\gamma'-\alpha''-\nu''|}h_1^{|\delta|+|\alpha'|+|\nu'|}
h^{|2\beta-\delta+\beta'+\gamma'-\alpha-\nu|}M_{\beta'+\gamma'}}
{R^{2\rho|\beta|}\left(RM_1\right)^{\rho|\beta'+\gamma'-\alpha''-\nu''|}M_{\alpha''+\nu''}}\\
&{}&\hspace{110 pt}\cdot e^{M(m|\xi|)}e^{M\left(8s(m+m')HR m_{|\beta|}\right)}e^{-M(m'|(x,y)|)}d\xi\\
&\leq&C_4\frac{M_{\beta'+\gamma'}}{e^{M(m'|(x,y)|)}}
\sum_{\substack{\alpha\leq\beta'\\ \nu\leq\gamma'}}\sum_{\substack{\alpha'+\alpha''=\alpha\\ \nu'+\nu''=\nu}}
{\gamma'\choose\nu}{\beta'\choose\alpha}{\alpha\choose\alpha'}{\nu\choose\nu'}\\
&{}&\cdot\int_{T_{\xi,|\beta|}}
\frac{(LH^2)^{2|\beta|}h_1^{|\delta|+|\alpha'|+|\nu'|}h^{|2\beta-\delta+\beta'+\gamma'-\alpha-\nu|}
e^{2M((m+m')R|\xi|)}e^{M\left(8s(m+m')HRm_{|\beta|}\right)}}
{(m'R)^{|\alpha''|+|\nu''|}R^{2\rho|\beta|}}d\xi,
\eeqs
where, in the last inequality, we used that
\beqs
\frac{(1+|\tau|)^{|\beta'|+|\gamma'|-|\alpha''|-|\nu''|}
(LH^2)^{|\beta'|+|\gamma'|-|\alpha''|-|\nu''|}}{\left(RM_1\right)^{\rho|\beta'+\gamma'-\alpha''-\nu''|}}\leq 1,
\eeqs
for large enough $R$. Moreover, on $T_{\xi,|\beta|}$, by proposition 3.6 of \cite{Komatsu1}, we have $2M((m+m')R|\xi|)\leq M(3(m+m')HR^2m_{|\beta|})+\ln c_0$. Lemma \ref{69} implies\\
$\ds e^{M\left(3(m+m')HR^2m_{|\beta|}\right)}e^{M\left(8s(m+m')HR m_{|\beta|}\right)}$
\beqs
&\leq& c_0^2H^{2(3c_0(m+m')HR^2+2)|\beta|}H^{2\left(8c_0s(m+m')HR+2\right)|\beta|}\leq c_0^2H^{4\left(8c_0s(m+m')HR^2+2\right)|\beta|}.
\eeqs
Similarly as in the proof for proposition \ref{72}, we have $\left|T_{\xi,|\beta|}\right| \leq C_5R^d H^{d|\beta|}$, for some $C_5>0$. For the $(M_p)$ case, $m$ is fixed. It is clear that, without losing generality, we can assume that $m\geq 1$. Choose $R$ such that $R\geq 4$ and $R^{2\rho}\geq 2(1+|\tau|+|\tau_1|)L^2H^{d+4}$. For arbitrary but fixed $m'>0$, choose $h$ such that $hH^{4\left(8c_0s(m+m')HR^2+2\right)}\leq 1$ and $2h\leq 1/(4m')$. Moreover, choose $h_1$ such that $h_1\leq h$. Then we obtain
\beqs
\left|D^{\beta'}_xD^{\gamma'}_yI_{\beta,\delta}(x,y)\right|&\leq&C_6 R^d\frac{M_{\beta'+\gamma'}h^{|\beta|}}{e^{M(m'|(x,y)|)}}\left(\frac{L^2H^{d+4}}{R^{2\rho}}\right)^{|\beta|}
\left(2h+\frac{1}{m'R}\right)^{|\beta'|+|\gamma'|}\\
&\leq&C_6 R^d\frac{M_{\beta'+\gamma'}h^{|\beta|}}
{e^{M(m'|(x,y)|)}}\cdot\frac{1}{(2(1+|\tau|+|\tau_1|))^{|\beta|}}\cdot\frac{1}{(2m')^{|\beta'|+|\gamma'|}},
\eeqs
when $(x,y)\in \RR^{2d}\backslash\Omega_{r}$. Note that the choice of $R$ (and hence of $\chi_n$, $n\in\NN$) depends only on $A_p$, $B_p$, $M_p$, $\tau$, $\tau_1$ and $a$, but not on $m'$. By the definition of $\tilde{\theta}$ it follows that there exists $C'>0$ such that
\beqs
\left|D^{\beta'}_xD^{\gamma'}_y\left(\tilde{\theta}(x,y)I_{\beta,\delta}(x,y)\right)\right|\leq C' R^d\frac{M_{\beta'+\gamma'}h^{|\beta|}}{e^{M(m'|(x,y)|)}}\cdot\frac{1}{(2(1+|\tau|+|\tau_1|))^{|\beta|}} \cdot\frac{1}{m'^{|\beta'|+|\gamma'|}},
\eeqs
for all $(x,y)\in\RR^{2d}$ and $\beta',\gamma'\in\NN^d$. Hence
\beqs
\sum_{|\beta|=1}^{\infty}\sum_{0\neq\delta\leq\beta}
{\beta\choose\delta}\frac{1}{\beta!}(|\tau_1|+|\tau|)^{|\beta|}
\left|D^{\beta'}_xD^{\gamma'}_y\left(\tilde{\theta}(x,y)I_{\beta,\delta}(x,y)\right)\right|
\leq C\frac{M_{\beta'+\gamma'}}{m'^{|\beta'|+|\gamma'|}e^{M(m'|(x,y)|)}},
\eeqs
for all $(x,y)\in\RR^{2d}$ and $\beta',\gamma'\in\NN^d$. From the arbitrariness of $m'$ it follows that $\tilde{\theta}\tilde{S}\in\SSS^{(M_p)}$. Now we consider the $\{M_p\}$ case. Then $h$ and $h_1$ are fixed. Choose $R$ such that $R^{2\rho}\geq 2(1+|\tau|+|\tau_1|)(h+h_1)hL^2H^{d+16}$ and then choose $m$ and $m'$ such that $8c_0s(m+m')HR^2\leq 1$. Then $H^{4\left(8c_0s(m+m')HR^2+2\right)|\beta|}\leq H^{12|\beta|}$. Then we have
\beqs
\left|D^{\beta'}_xD^{\gamma'}_yI_{\beta,\delta}(x,y)\right|&\leq&C_6 R^d\frac{M_{\beta'+\gamma'}\left(hL^2H^{d+16}\right)^{|\beta|}h_1^{|\delta|} h^{|\beta-\delta|}}{e^{M(m'|(x,y)|)}R^{2\rho|\beta|}} \left(h+h_1+\frac{1}{m'R}\right)^{|\beta'|+|\gamma'|}\\
&\leq&C_6R^d\frac{M_{\beta'+\gamma'}}{e^{M(m'|(x,y)|)}(2(1+|\tau|+|\tau_1|))^{|\beta|}} \left(h+h_1+\frac{1}{m'R}\right)^{|\beta'|+|\gamma'|},
\eeqs
when $(x,y)\in\RR^{2d}\backslash\Omega_r$. By the definition of $\tilde{\theta}$ it follows that there exist $C'>0$ and $\tilde{h}>0$ such that
\beqs
\left|D^{\beta'}_xD^{\gamma'}_y\left(\tilde{\theta}(x,y)I_{\beta,\delta}(x,y)\right)\right|\leq C'R^d\frac{M_{\beta'+\gamma'}\tilde{h}^{|\beta'|+|\gamma'|}}{e^{M(m'|(x,y)|)}(2(1+|\tau|+|\tau_1|))^{|\beta|}},
\eeqs
for all $(x,y)\in\RR^{2d}$ and $\beta',\gamma'\in\NN^d$. Hence
\beqs
\sum_{|\beta|=1}^{\infty}\sum_{0\neq\delta\leq\beta}
{\beta\choose\delta}\frac{1}{\beta!}(|\tau_1|+|\tau|)^{|\beta|}
\left|D^{\beta'}_xD^{\gamma'}_y\left(\tilde{\theta}(x,y)I_{\beta,\delta}(x,y)\right)\right|\leq C\frac{M_{\beta'+\gamma'}\tilde{h}^{|\beta'|+|\gamma'|}}{e^{M(m'|(x,y)|)}},
\eeqs
for all $(x,y)\in\RR^{2d}$ and $\beta',\gamma'\in\NN^d$, from what we obtain $\tilde{\theta}\tilde{S}\in\SSS^{\{M_p\}}$.\\
\indent It remains to prove that $\ds\sum_{n=0}^{\infty}\tilde{\theta}(x,y)S_{2,n}(x,y)\in\SSS^*$. Note that
\beqs
S_{2,n}(x,y)&=&\frac{n+1}{(2\pi)^d}\sum_{|\beta|=n+1}\sum_{\delta\leq\beta}{\beta\choose\delta} \frac{(-1)^{|\beta|}}{\beta!}(\tau-\tau_1)^{|\beta|} \int_{\RR^d}e^{i(x-y)\xi}D^{\delta}_{\xi}\left(\chi_{n+1}-\chi_n\right)(x',\xi)\\
&{}&\hspace{110 pt}\cdot\int_0^1(1-t)^n D^{\beta-\delta}_{\xi}\partial^{\beta}_x a(x'+t(\tau-\tau_1)y',\xi)dtd\xi.
\eeqs
For brevity in notation, put
\beqs
\tilde{I}_{\beta,\delta,n}(x,y)=\int_{\RR^d}e^{i(x-y)\xi}D^{\delta}_{\xi}\left(\chi_{n+1}-\chi_n\right)(x',\xi) \int_0^1(1-t)^n D^{\beta-\delta}_{\xi}\partial^{\beta}_x a(x'+t(\tau-\tau_1)y',\xi)dtd\xi.
\eeqs
We will estimate $\left|D^{\beta'}_x D^{\gamma'}_y \tilde{I}_{\beta,\delta,n}(x,y)\right|$ when $(x,y)\in\RR^{2d}\backslash\Omega_r\supseteq \mathrm{supp\,}\tilde{\theta}$.\\
$\ds\left|D^{\beta'}_x D^{\gamma'}_y \tilde{I}_{\beta,\delta,n}(x,y)\right|$
\beqs
&\leq& C_1\sum_{\substack{\alpha\leq\beta'\\ \nu\leq\gamma'}}\sum_{\substack{\alpha'+\alpha''=\alpha\\ \nu'+\nu''=\nu}}{\beta'\choose\alpha}{\gamma'\choose\nu}{\alpha\choose\alpha'}{\nu\choose\nu'} (2(1+|\tau|+|\tau_1|))^{|\beta'|-|\alpha''|+|\gamma'|-|\nu''|}\\
&{}&\hspace{5 pt}\cdot\int_{T_{\xi,n+1}}|\xi|^{|\alpha''|+|\nu''|} \frac{h_1^{|\delta|+|\alpha'|+|\nu'|}A_{\delta}B_{\alpha'+\nu'}}{(Rm_n)^{|\alpha'|+|\nu'|+|\delta|}}\\
&{}&\hspace{15 pt}\cdot\int_0^1(1-t)^n\frac{h^{2|\beta|-|\delta|+|\beta'|-|\alpha|+|\gamma'|-|\nu|} A_{\beta-\delta}B_{\beta+\beta'-\alpha+\gamma'-\nu}e^{M(m|\xi|)}e^{M(m|x'+t(\tau-\tau_1)y'|)}}{\langle (x'+t(\tau-\tau_1)y',\xi)\rangle^{\rho(2|\beta|-|\delta|+|\beta'|-|\alpha|+|\gamma'|-|\nu|)}}dtd\xi.
\eeqs
Above, we already proved that on $\RR^{2d}\backslash\Omega_r$, $\langle x\rangle \leq 2\langle x'\rangle$. Using this, by similar technic as there, one easily proves that $\langle(x'+t(\tau-\tau_1)y',\xi)\rangle\geq Rm_n$ when $(x,y)\in\RR^{2d}\backslash\Omega_r$ and $\chi_{n+1}(x',\xi)-\chi(x',\xi)\neq 0$. Also, for such $x$, $y$ and $\xi$ we have $|x'+t(\tau-\tau_1)y'|\leq |x'|+(|\tau|+|\tau_1|)|y'|\leq \langle x'\rangle+2r(|\tau|+|\tau_1|)\langle x'\rangle\leq 8Rm_{n+1}$ and $|\xi|\leq 3Rm_{n+1}$. Now, the proof continues analogously as above and one obtains $\sum_n\tilde{\theta}S_{2,n}\in\SSS^*$. We already pointed out that from this it follows that $K\in\SSS^*$.
\end{proof}

\begin{theorem}
Let $\tau\in\RR$ and $a\in\Gamma_{A_p,B_p,\rho}^{*,\infty}\left(\RR^{2d}\right)$. The transposed operator, ${}^t\Op_{\tau}(a)$, is still a pseudo-differential operator and it is equal to $\Op_{1-\tau}(a(x,-\xi))$. Moreover, there exist $b\in\Gamma_{A_p,B_p,\rho}^{*,\infty}\left(\RR^{2d}\right)$ and *-regularizing operator $T$ such that ${}^t\Op_{\tau}(a)=\Op_{\tau}(b)+T$ and
\beqs
b(x,\xi)\sim\sum_{\alpha}\frac{1}{\alpha!}(1-2\tau)^{|\alpha|}(-\partial_{\xi})^{\alpha}D^{\alpha}_x a(x,-\xi) \mbox{ in } FS_{A_p,B_p,\rho}^{*,\infty}\left(\RR^{2d}\right).
\eeqs
\end{theorem}

\begin{proof} In the observation after theorem \ref{npr} we proved that ${}^t\Op_{\tau}(a(x,\xi))=\Op_{1-\tau}(a(x,-\xi))$. The rest follows from theorem \ref{200}.
\end{proof}

\begin{theorem}\label{cef}
Let $a,b\in\Gamma_{A_p,B_p,\rho}^{*,\infty}\left(\RR^{2d}\right)$. There exist $f\in\Gamma_{A_p,B_p,\rho}^{*,\infty}\left(\RR^{2d}\right)$ and *-regularizing operator $T$ such that $a(x,D)b(x,D)=f(x,D)+T$ and $f$ has the asymptotic expansion
\beq\label{ezz}
f(x,\xi)\sim\sum_{\alpha}\frac{1}{\alpha!}\partial^{\alpha}_{\xi}a(x,\xi)D^{\alpha}_x b(x,\xi) \mbox{ in } FS_{A_p,B_p,\rho}^{*,\infty}\left(\RR^{2d}\right).
\eeq
\end{theorem}

\begin{proof} By the above theorem ${}^{t}b(x,D)=b_1(x,D)+T'$ where $T'$ is *-regularizing operator and $b_1\in\Gamma_{A_p,B_p,\rho}^{*,\infty}\left(\RR^{2d}\right)$ with asymptotic expansion
\beq\label{240}
b_1(x,\xi)\sim\sum_{\alpha}\frac{1}{\alpha!}(-\partial_{\xi})^{\alpha}D^{\alpha}_x b(x,-\xi).
\eeq
Again, by the above theorem, ${}^{t}b_1(x,D)=\Op_1(b_1(x,-\xi))$ and $b(x,D)={}^t\Op_1(b(x,-\xi))={}^{t}\left({}^{t}b(x,D)\right)$. Put $b_2(x,\xi)=b_1(x,-\xi)$. Then we have
\beqs
b(x,D)={}^{t}\left({}^{t}b(x,D)\right)={}^{t}b_1(x,D)+{}^{t}T'=\Op_1(b_2)+{}^{t}T'.
\eeqs
We have $a(x,D)b(x,D)=a(x,D)\Op_1(b_2)+T_1$, where we put $T_1=a(x,D){}^{t}T'$, which is *-regularizing. Because $\ds\mathcal{F}\left(\Op_1(b_2)u\right)(\xi)=\int_{\RR^d}e^{-iy\xi}b_2(y,\xi)u(y)dy$ and $\Op_1(b_2)u\in\SSS^*$,
\beqs
a(x,D)\Op_1(b_2)u(x)=\frac{1}{(2\pi)^d}\int_{\RR^d}\int_{\RR^d}e^{i(x-y)\xi}a(x,\xi)b_2(y,\xi)u(y)dyd\xi
\eeqs
and this is well defined as iterated integral by theorem \ref{17}. $\tilde{a}(x,y,\xi)=a(x,\xi)b_2(y,\xi)$ is an element of $\Pi_{A_p,B_p,\rho}^{*,\infty}\left(\RR^{3d}\right)$. To prove that one only has to use the inequalities $2\langle (x,\xi)\rangle\langle x-y\rangle \geq \langle(x,y,\xi)\rangle$ and $2\langle (y,\xi)\rangle\langle x-y\rangle \geq \langle(x,y,\xi)\rangle$ in the estimates for the derivatives of $\tilde{a}$. The operator $\tilde{A}$ corresponding to this $\tilde{a}$ is the same as $a(x,D)\Op_1(b_2)$. Let
\beqs
p_j(x,\xi)=\sum_{|\beta|=j}\frac{1}{\beta!}\partial^{\beta}_{\xi} \left(a(x,\xi)D^{\beta}_x b_2(x,\xi)\right).
\eeqs
Obviously $\sum_j p_j\in FS_{A_p,B_p,\rho}^{*,\infty}\left(\RR^{2d}\right)$. Let $\chi_j(x,\xi)$, $j\in\NN$, be the sequence constructed in the proof of theorem \ref{150}, such that $f=\sum_j (1-\chi_j)p_j$ is an element of $\Gamma_{A_p,B_p,\rho}^{*,\infty}\left(\RR^{2d}\right)$ and $f\sim\sum_j p_j$. By the observations after theorem \ref{npr}, the operator $f(x,D)$ coincide with the operator $F$ corresponding to $f$ when we observe $f(x,\xi)$ as elements of $\Pi_{A_p,B_p,\rho}^{*,\infty}\left(\RR^{3d}\right)$. We will prove that the kernel of $\tilde{A}-F$ is in $\SSS^*\left(\RR^{2d}\right)$, i.e. $\tilde{A}-F$ is *-regularizing. Similarly as in the proof of theorem \ref{200}, $\ds \tilde{a}(x,y,\xi)-f(x,\xi)=\sum_{n=0}^{\infty}\tilde{a}_n(x,y,\xi)$
where we put
\beqs
\tilde{a}_n(x,y,\xi)=\left(\chi_{n+1}(x,\xi)-\chi_n(x,\xi)\right)\left(\tilde{a}(x,y,\xi)-\sum_{j=0}^n p_j(x,\xi)\right),
\eeqs
which is obviously an element of $\Pi_{A_p,B_p,\rho}^{*,\infty}\left(\RR^{3d}\right)$. Denote by $\tilde{A}_n$ its corresponding operator. Similarly as in the proof of theorem \ref{200}, we have $K(x,y)=\sum_n K_n(x,y)$, where $K$ is the kernel of $\tilde{A}-F$, $K_n$ is the kernel of $\tilde{A}_n$ and the convergence holds in $\SSS'^*$. Observe that
\beqs
K_n(x,y)=\frac{1}{(2\pi)^d}\int_{\RR^d}e^{i(x-y)\xi}\left(\chi_{n+1}-\chi_n\right)(x,\xi)\left(a(x,\xi)b_2(y,\xi) -\sum_{j=0}^n p_j(x,\xi)\right)d\xi,
\eeqs
for all $n\in\NN$. Let $r=1/8$. Take $\theta\in\EE^*\left(\RR^{2d}\right)$ as in lemma \ref{pll} and put $\tilde{\theta}=1-\theta$. $\theta$ and $\tilde{\theta}$ are obviously multipliers for $\SSS'^*$. By proposition \ref{72} and the properties of $\theta$, $\theta K\in\SSS^*\left(\RR^{2d}\right)$. It is enough to prove that $\tilde{\theta} K\in \SSS^*\left(\RR^{2d}\right)$. Note that $\tilde{\theta} K=\sum_n\tilde{\theta} K_n$. Our goal is to prove that $\sum_n \tilde{\theta}K_n\in\SSS^*$. Taylor expand $b_2(y,\xi)$ in the first variable to obtain
\beqs
b_2(y,\xi)=\sum_{|\beta|\leq n}\frac{1}{\beta!}(y-x)^{\beta}\partial^{\beta}_x b_2(x,\xi)+W_{n+1}(x,y,\xi),
\eeqs
where $W_{n+1}$ is the remainder of the expansion:
\beqs
W_{n+1}(x,y,\xi)=(n+1)\sum_{|\beta|=n+1}\frac{1}{\beta!}(y-x)^{\beta}\int_0^1 (1-t)^n \partial^{\beta}_x b_2(x+t(y-x),\xi)dt.
\eeqs
If we insert this in the expression for $K_n$, keeping in mind the definition of $p_j$, we have $K_n(x,y)=S_{1,n}(x,y)+S_{2,n}(x,y)$ where we put
\beqs
S_{1,n}(x,y)&=&\frac{1}{(2\pi)^d}\sum_{0\neq |\beta|\leq n}\sum_{0\neq\delta\leq \beta}{\beta\choose\delta}\frac{1}{\beta!}\\
&{}&\hspace{30 pt}\cdot\int_{\RR^d}e^{i(x-y)\xi}D^{\delta}_{\xi}\left(\chi_{n+1}-\chi_n\right)(x,\xi) D^{\beta-\delta}_{\xi}\left(a(x,\xi)\partial^{\beta}_x b_2(x,\xi)\right)d\xi\\
S_{2,n}(x,y)&=&\frac{1}{(2\pi)^d}\int_{\RR^d}e^{i(x-y)\xi}\left(\chi_{n+1}-\chi_n\right)(x,\xi) a(x,\xi)W_{n+1}(x,y,\xi)d\xi.
\eeqs
Our goal is to prove that $\sum_n \tilde{\theta}S_{1,n}$ and $\sum_n \tilde{\theta}S_{2,n}$ are $\SSS^*$ functions. Similarly as in the proof of theorem \ref{200}, $\sum_n S_{1,n}$ converges in $\SSS'^*$ to $\tilde{S}$ and $\ds\tilde{S}=-\frac{1}{(2\pi)^d}\sum_{|\beta|=1}^{\infty}\sum_{0\neq\delta\leq\beta}
{\beta\choose\delta}\frac{1}{\beta!}I_{\beta,\delta}$, where the convergence is in $\SSS'^*$, where we put
\beqs
I_{\beta,\delta}(x,y)=\int_{\RR^d}e^{i(x-y)\xi}D^{\delta}_{\xi}\chi_{|\beta|}(x,\xi) D^{\beta-\delta}_{\xi}\left(a(x,\xi)\partial^{\beta}_x b_2(x,\xi)\right)d\xi.
\eeqs
To prove that $\ds-\frac{1}{(2\pi)^d}\sum_{|\beta|=1}^{\infty}\sum_{0\neq\delta\leq\beta}
{\beta\choose\delta}\frac{1}{\beta!}\tilde{\theta}I_{\beta,\delta}$ is in $\SSS^*$ we have to estimate the derivatives of $I_{\beta,\delta}$ when $(x,y)\in\RR^{2d}\backslash\Omega_r\supseteq \mathrm{supp\,}\tilde{\theta}$. Note that, we can choose $m$ such that $a,b_2\in\Gamma_{A_p,B_p,\rho}^{(M_p),\infty}\left(\RR^{2d};m\right)$ in the $(M_p)$ case, resp. we can choose $h$ such that $a,b_2\in\Gamma_{A_p,B_p,\rho}^{\{M_p\},\infty}\left(\RR^{2d};h\right)$ in the $\{M_p\}$ case. Let $T_n$ be as in (\ref{zao}) and put $T_{\xi,n}$ to be the projection of $T_n$ on $\RR^d_{\xi}$. By the way we constructed $\chi_n$, it follows that $\supp\chi_{|\beta|}\subseteq T_{|\beta|}$.\\
$\ds\left|D^{\beta'}_xD^{\gamma'}_yI_{\beta,\delta}(x,y)\right|$
\beqs
&\leq&\sum_{\kappa\leq\beta-\delta}\sum_{\alpha\leq\beta'}\sum_{\alpha'+\alpha''=\alpha}\sum_{\alpha'''\leq\beta'-\alpha} {{\beta-\delta}\choose\kappa}
{\beta'\choose\alpha}{\alpha\choose\alpha'}{{\beta'-\alpha}\choose\alpha'''}\\
&{}&\hspace{30 pt}\cdot\int_{T_{\xi,|\beta|}}
|\xi|^{|\alpha''|+|\gamma'|}\left|D^{\delta}_{\xi}D^{\alpha'}_x\chi_{|\beta|}(x,\xi)\right| \left|D^{\beta-\delta-\kappa}_{\xi} D^{\beta'-\alpha-\alpha'''}_x a(x,\xi)D^{\kappa}_{\xi}D^{\beta+\alpha'''}_x b_2(x,\xi)\right|d\xi\\
&\leq&C_1\sum_{\kappa\leq\beta-\delta}\sum_{\alpha\leq\beta'}\sum_{\alpha'+\alpha''=\alpha}{{\beta-\delta}\choose\kappa}
{\beta'\choose\alpha}{\alpha\choose\alpha'}2^{|\beta'|-|\alpha|}\\
&{}&\hspace{30 pt}\cdot\int_{T_{\xi,|\beta|}}|\xi|^{|\alpha''|+|\gamma'|} \frac{h_1^{|\delta|+|\alpha'|}A_{\delta}B_{\alpha'}h^{|2\beta-\delta+\beta'-\alpha|}
A_{\beta-\delta}B_{\beta+\beta'-\alpha}e^{2M(m|\xi|)}e^{2M(m|x|)}}
{(Rm_{|\beta|})^{|\delta|+|\alpha'|}\langle (x,\xi)\rangle^{\rho|2\beta-\delta+\beta'-\alpha|}}d\xi.
\eeqs
Because $\delta\neq 0$, $D^{\delta}_{\xi}D^{\alpha'}_x\chi_{|\beta|}(x,\xi)=0$ when $\chi_{|\beta|}(x,\xi)=1$, hence when $|x|\leq Rm_{|\beta|}$ and $|\xi|\leq Rm_{|\beta|}$. So, when $D^{\delta}_{\xi}D^{\alpha'}_x\chi_{|\beta|}(x,\xi)\neq0$ we have $\langle (x,\xi)\rangle\geq Rm_{|\beta|}$. Now the proof continues analogously as for theorem \ref{200}.\\
\indent Next we will prove that $\sum_n\tilde{\theta}(x,y)S_{2,n}(x,y)\in\SSS^*$. Note that
\beqs
S_{2,n}(x,y)&=&\frac{n+1}{(2\pi)^d}\sum_{|\beta|=n+1}\sum_{\delta\leq\beta}\sum_{\kappa\leq\beta-\delta} {\beta\choose\delta}{{\beta-\delta}\choose\kappa}\frac{1}{\beta!} \int_{\RR^d}e^{i(x-y)\xi}D^{\delta}_{\xi}\left(\chi_{n+1}-\chi_n\right)(x,\xi)\\
&{}&\hspace{90 pt}\cdot D^{\kappa}_{\xi}a(x,\xi)\int_0^1(1-t)^n D^{\beta-\delta-\kappa}_{\xi}\partial^{\beta}_x b_2(x+t(y-x),\xi)dtd\xi.
\eeqs
For brevity in notation, put
\beqs
\tilde{I}_{\beta,\delta,n}(x,y)&=&\sum_{\kappa\leq\beta-\delta}{{\beta-\delta}\choose\kappa} \int_{\RR^d}e^{i(x-y)\xi}D^{\delta}_{\xi}\left(\chi_{n+1}-\chi_n\right)(x,\xi) D^{\kappa}_{\xi}a(x,\xi)\\
&{}&\hspace{30 pt}\cdot\int_0^1(1-t)^n D^{\beta-\delta-\kappa}_{\xi}\partial^{\beta}_x b_2(x+t(y-x),\xi)dtd\xi.
\eeqs
We will estimate $\left|D^{\beta'}_x D^{\gamma'}_y \tilde{I}_{\beta,\delta,n}(x,y)\right|$ when $(x,y)\in\RR^{2d}\backslash\Omega_r\supseteq \mathrm{supp\,}\tilde{\theta}$.\\
$\ds\left|D^{\beta'}_x D^{\gamma'}_y \tilde{I}_{\beta,\delta,n}(x,y)\right|$
\beqs
&\leq& C_1\sum_{\substack{\alpha\leq\beta'\\ \nu\leq\gamma'}} \sum_{\alpha'+\alpha''=\alpha}\sum_{\kappa\leq\beta-\delta}\sum_{\alpha'''\leq\beta'-\alpha} {\beta'\choose\alpha}{\gamma'\choose\nu}{\alpha\choose\alpha'}{{\beta-\delta}\choose\kappa} {{\beta'-\alpha}\choose\alpha'''}\\
&{}&\hspace{5 pt}\cdot\int_{T_{\xi,n+1}}|\xi|^{|\alpha''|+|\nu|} \frac{h_1^{|\delta|+|\alpha'|}A_{\delta}B_{\alpha'}}{(Rm_n)^{|\alpha'|+|\delta|}}\\
&{}&\hspace{15 pt}\cdot\int_0^1(1-t)^n\frac{h^{2|\beta|-|\delta|+|\beta'|-|\alpha|+|\gamma'|-|\nu|} A_{\beta-\delta}B_{\beta+\beta'-\alpha+\gamma'-\nu}e^{2M(m|\xi|)}e^{2M(m(|x|+|y|))}} {\langle(x,\xi)\rangle^{\rho(|\beta'|-|\alpha|-|\alpha'''|+|\kappa|)} \langle(x+t(y-x),\xi)\rangle^{\rho(2|\beta|-|\delta|-|\kappa|+|\alpha'''|+|\gamma'|-|\nu|)}}dtd\xi.
\eeqs
When $(\chi_{n+1}-\chi_n)(x,\xi)\neq 0$ and $(x,y)\in\RR^{2d}\backslash\Omega_r$, the inequalities $\langle (x,\xi)\rangle\geq Rm_m$ and $\langle (x+t(y-x),\xi)\rangle\geq Rm_m$ hold. Also $|x|+|y|\leq 2|x|+|x-y|\leq s\langle x\rangle\leq 4sRm_{n+1}$, where we put $s=2+r$. Hence\\
$\ds\left|D^{\beta'}_x D^{\gamma'}_y \tilde{I}_{\beta,\delta,n}(x,y)\right|$
\beqs
&\leq& \frac{C_2}{n+1}\sum_{\substack{\alpha\leq\beta'\\ \nu\leq\gamma'}} \sum_{\alpha'+\alpha''=\alpha} {\beta'\choose\alpha}{\gamma'\choose\nu}{\alpha\choose\alpha'}2^{|\beta|-|\delta|+|\beta'|-|\alpha|} \int_{T_{\xi,n+1}}|\xi|^{|\alpha''|+|\nu|}d\xi\\
&{}&\hspace{15 pt}\cdot\frac{h_1^{|\delta|+|\alpha'|}h^{2|\beta|-|\delta|+|\beta'|-|\alpha|+|\gamma'|-|\nu|} A_{\beta}B_{\beta+\beta'-\alpha''+\gamma'-\nu}e^{M(3mHRm_{n+1})}e^{M(4smHRm_{n+1})}} {(Rm_n)^{\rho(2|\beta|+|\beta'|-|\alpha''|+|\gamma'|-|\nu|)}}.
\eeqs
The proof continues in analogous fashion as for theorem \ref{200} and one obtains that $\sum_n\tilde{\theta}S_{2,n}\in\SSS^*$. Hence, we proved that $a(x,D)b(x,D)=a(x,D)\Op_1(b_2)+T_1=f(x,D)+T_2$, where $T_2$ is *-regularizing operator. It remains to prove (\ref{ezz}). Obviously, it is enough to prove that $\ds \sum_{\beta}\frac{1}{\beta!}\partial^{\beta}_{\xi} \left(a(x,\xi)D^{\beta}_x b_2(x,\xi)\right)\sim\sum_{\beta}\frac{1}{\beta!}\partial^{\beta}_{\xi}a(x,\xi)D^{\beta}_x b(x,\xi)$. For $N\in\ZZ_+$ we have\\
$\ds \sum_{j=0}^{N-1}\sum_{|\beta|=j}\frac{1}{\beta!}\partial^{\beta}_{\xi} \left(a\cdot D^{\beta}_x b_2\right)$
\beqs
&=&\sum_{j=0}^{N-1}\sum_{|\alpha+\gamma|=j}\frac{1}{\alpha!\gamma!}\partial^{\gamma}_{\xi} a\cdot \left(\partial^{\alpha}_{\xi} D^{\alpha+\gamma}_xb_2-\sum_{s=0}^{N-j-1}\sum_{|\delta|=s}\frac{(-1)^{|\delta|}}{\delta!}\partial^{\alpha+\delta}_{\xi} D^{\alpha+\gamma+\delta}_xb\right)\\
&{}&\hspace{50 pt}+\sum_{j=0}^{N-1}\sum_{s=0}^{N-j-1}\sum_{|\alpha+\gamma|=j}\sum_{|\delta|=s} \frac{(-1)^{|\delta|}}{\alpha!\gamma!\delta!}\partial^{\gamma}_{\xi} a\cdot\partial^{\alpha+\delta}_{\xi} D^{\alpha+\gamma+\delta}_x b.
\eeqs
Note that\\
$\ds\sum_{j=0}^{N-1}\sum_{s=0}^{N-j-1}\sum_{|\alpha+\gamma|=j}\sum_{|\delta|=s} \frac{(-1)^{|\delta|}}{\alpha!\gamma!\delta!}\partial^{\gamma}_{\xi} a\cdot\partial^{\alpha+\delta}_{\xi} D^{\alpha+\gamma+\delta}_x b$
\beqs
&=&\sum_{j=0}^{N-1}\sum_{s=0}^{N-j-1}\sum_{k=0}^j\sum_{\substack{|\alpha|=k\\ |\gamma|=j-k}}\sum_{|\delta|=s} \frac{(-1)^{|\delta|}}{\alpha!\gamma!\delta!}\partial^{\gamma}_{\xi} a\cdot\partial^{\alpha+\delta}_{\xi} D^{\alpha+\gamma+\delta}_x b\\
&=&\sum_{s=0}^{N-1}\sum_{j=0}^{N-s-1}\sum_{k=0}^j\sum_{\substack{|\alpha|=k\\ |\gamma|=j-k}}\sum_{|\delta|=s} \frac{(-1)^{|\delta|}}{\alpha!\gamma!\delta!}\partial^{\gamma}_{\xi} a\cdot\partial^{\alpha+\delta}_{\xi} D^{\alpha+\gamma+\delta}_x b\\
&=&\sum_{s=0}^{N-1}\sum_{k=0}^{N-s-1}\sum_{j=k}^{N-s-1}\sum_{\substack{|\alpha|=k\\ |\gamma|=j-k}}\sum_{|\delta|=s} \frac{(-1)^{|\delta|}}{\alpha!\gamma!\delta!}\partial^{\gamma}_{\xi} a\cdot\partial^{\alpha+\delta}_{\xi} D^{\alpha+\gamma+\delta}_x b\\
&=&\sum_{t=0}^{N-1}\sum_{s+k=t}\sum_{j=k}^{N-s-1}\sum_{|\gamma|=j-k}\sum_{\substack{|\alpha|=k\\ |\delta|=s}} \frac{(-1)^{|\delta|}}{\alpha!\gamma!\delta!}\partial^{\gamma}_{\xi} a\cdot\partial^{\alpha+\delta}_{\xi} D^{\alpha+\gamma+\delta}_x b\\
&=&\sum_{t=0}^{N-1}\sum_{s+k=t}\sum_{j=k}^{N-s-1}\sum_{|\gamma|=j-k}\sum_{|\beta|=t}\sum_{\alpha+\delta=\beta} \frac{(-1)^{|\delta|}}{\alpha!\gamma!\delta!}\partial^{\gamma}_{\xi} a\cdot\partial^{\beta}_{\xi} D^{\beta+\gamma}_x b\\
&=&\sum_{j=0}^{N-1}\sum_{|\gamma|=j} \frac{1}{\gamma!}\partial^{\gamma}_{\xi} a\cdot D^{\gamma}_x b.
\eeqs
Hence, we have to estimate the derivatives of
\beqs
\sum_{j=0}^{N-1}\sum_{|\alpha+\gamma|=j}\frac{1}{\alpha!\gamma!}\partial^{\gamma}_{\xi} a\cdot \partial^{\alpha}_{\xi} D^{\alpha+\gamma}_x\left(b_2-\sum_{s=0}^{N-j-1}\sum_{|\delta|=s}\frac{(-1)^{|\delta|}}{\delta!}\partial^{\delta}_{\xi} D^{\delta}_xb\right).
\eeqs
By construction $\ds b(x,\xi)\sim\sum_{j=0}^{\infty}\sum_{|\delta|=j}\frac{(-1)^{|\delta|}}{\delta!}\partial_{\xi}^{\delta}D^{\delta}_x b(x,\xi)$. So, for $(x,\xi)\in Q_{Bm_N}^c$, we have\\
$\ds \left|D^{\alpha'}_{\xi}D^{\beta'}_x\sum_{j=0}^{N-1}\sum_{|\alpha+\gamma|=j}\frac{1}{\alpha!\gamma!}\partial^{\gamma}_{\xi} a(x,\xi)\partial^{\alpha}_{\xi} D^{\alpha+\gamma}_x \left(b_2(x,\xi)-\sum_{s=0}^{N-j-1}\sum_{|\delta|=s}\frac{(-1)^{|\delta|}}{\delta!}\partial^{\delta}_{\xi} D^{\delta}_x b(x,\xi)\right)\right|$
\beqs
&\leq&C_1\sum_{j=0}^{N-1}\sum_{|\alpha+\gamma|=j}\sum_{\substack{\alpha''\leq\alpha'\\ \beta''\leq\beta'}} {\alpha'\choose\alpha''}{\beta'\choose\beta''} \frac{h^{|\alpha'|+|\beta'|+2N}A_{|\alpha'|+j}B_{|\beta'|+j}A_{N-j}B_{N-j}e^{2M(m|\xi|)}e^{2M(m|x|)}} {\alpha!\gamma!\langle(x,\xi)\rangle^{\rho(|\alpha'|+|\beta'|+2N)}}\\
&\leq&C\frac{(4Hh)^{|\alpha'|+|\beta'|+2N}A_{\alpha'}B_{\beta'}A_{N}B_{N}e^{M(mH|\xi|)}e^{M(mH|x|)}} {\langle(x,\xi)\rangle^{\rho(|\alpha'|+|\beta'|+2N)}},
\eeqs
which gives the desired asymptotic expansion.
\end{proof}

For the next corollary we need the following technical lemma.

\begin{lemma}\label{ase}
Let $a,b\in\Gamma_{A_p,B_p,\rho}^{*,\infty}\left(\RR^{2d}\right)$ are such that $a\sim\sum_j a_j$ and $b\sim\sum_j b_j$. Then $\ds ab\sim\sum_{j=0}^{\infty}\sum_{s+k=j}a_sb_k$ and
\beq\label{zzs}
\partial^{\alpha}_{\xi}a(x,\xi)\partial^{\alpha}_x b(x,\xi)\sim\underbrace{0+...+0}_{|\alpha|}+\sum_{j=|\alpha|}^{\infty}\sum_{s+k+|\alpha|=j}
\partial^{\alpha}_{\xi} a_s(x,\xi)\partial^{\alpha}_x b_k(x,\xi)
\eeq
in $FS_{A_p,B_p,\rho}^{*,\infty}\left(\RR^{2d}\right)$, for each $\alpha\in\NN^d$. Moreover, there exist $B>0$ and $m>0$ such that, for every $h>0$, there exists $C>0$; resp. there exist $B>0$ and $h>0$ such that, for every $m>0$, there exists $C>0$; such that
\beqs
\sup_{\alpha}\sup_{N>|\alpha|}\sup_{\gamma,\delta}\sup_{(x,\xi)\in Q_{Bm_N}^c}\left|D^{\gamma}_{\xi}D^{\delta}_x \left(\partial^{\alpha}_{\xi}a(x,\xi)\partial^{\alpha}_x b(x,\xi)-\sum_{j=|\alpha|}^{N-1}\sum_{s+k+|\alpha|=j}
\partial^{\alpha}_{\xi} a_s(x,\xi)\partial^{\alpha}_x b_k(x,\xi)\right)\right|\\
\cdot \frac{\langle(x,\xi)\rangle^{\rho|\gamma|+\rho|\delta|+2N\rho}e^{-M(m|\xi|)}e^{-M(m|x|)}}
{h^{|\gamma|+|\delta|+2N}A_{\gamma}B_{\delta}A_NB_N}\leq C.
\eeqs
\end{lemma}

\begin{proof} By the conditions in the lemma, there exist $B>0$ and $m>0$ such that, for every $h>0$, there exists $\tilde{C}>0$; resp. there exist $B>0$ and $h>0$ such that, for every $m>0$, there exists $\tilde{C}>0$; such that
\beqs
\sup_{j\in\NN}\sup_{\gamma,\delta}\sup_{(x,\xi)\in Q_{Bm_j}^c}\frac{\left|D^{\gamma}_{\xi}D^{\delta}_x a_j(x,\xi)\right|
\langle (x,\xi)\rangle^{\rho|\gamma|+\rho|\delta|+2j\rho}e^{-M(m|\xi|)}e^{-M(m|x|)}}
{h^{|\gamma|+|\delta|+2j}A_{\gamma}B_{\delta}A_jB_j}\leq \tilde{C},\\
\sup_{N\in\ZZ_+}\sup_{\gamma,\delta}\sup_{(x,\xi)\in Q_{Bm_N}^c}\frac{\left|D^{\gamma}_{\xi}D^{\delta}_x \left(a(x,\xi)-\sum_{j<N}a_j(x,\xi)\right)\right|\langle(x,\xi)\rangle^{\rho|\gamma|+\rho|\delta|+2N\rho}}
{h^{|\gamma|+|\delta|+2N}A_{\gamma}B_{\delta}A_NB_N}\cdot\\
\cdot e^{-M(m|\xi|)}e^{-M(m|x|)}\leq \tilde{C}
\eeqs
and the same estimate for $D^{\gamma}_{\xi}D^{\delta}_x b_j$ and $D^{\gamma}_{\xi}D^{\delta}_x \left(b-\sum_{j<N}b_j\right)$. One easily checks that $\ds \underbrace{0+...+0}_{|\alpha|}+\sum_{j=|\alpha|}^{\infty}\sum_{s+k+|\alpha|=j}
\partial^{\alpha}_{\xi} a_s\partial^{\alpha}_x b_k\in FS_{A_p,B_p,\rho}^{*,\infty}\left(\RR^{2d}\right)$, for each fixed $\alpha\in\NN^d$. For $N>|\alpha|$ and $(x,\xi)\in Q^c_{Bm_N}$, observe that
\beqs
\partial^{\alpha}_{\xi} a\cdot\partial^{\alpha}_x b&=&\partial^{\alpha}_{\xi} a\cdot\left(\partial^{\alpha}_x b-\sum_{k=0}^{N-|\alpha|-1}\partial^{\alpha}_x b_k\right)
+\sum_{k=0}^{N-|\alpha|-1}\left(\partial^{\alpha}_{\xi}a-\sum_{s=0}^{N-|\alpha|-k-1}\partial^{\alpha}_{\xi} a_s\right)\cdot\partial^{\alpha}_x b_k\\
&{}&\hspace{150 pt}+\sum_{j=|\alpha|}^{N-1}\sum_{s+k=j-|\alpha|}\partial^{\alpha}_{\xi} a_s\partial^{\alpha}_x b_k.
\eeqs
Using this, one verifies that\\
$\ds\left|D^{\gamma}_{\xi}D^{\delta}_x\sum_{k=0}^{N-|\alpha|-1}\left(\partial^{\alpha}_{\xi}a(x,\xi)-
\sum_{s=0}^{N-|\alpha|-k-1}\partial^{\alpha}_{\xi} a_s(x,\xi)\right)\cdot\partial^{\alpha}_x b_k(x,\xi)\right|$
\beqs
\leq c_0^4\tilde{C}^2\frac{(4hH)^{|\gamma|+|\delta|+2N}A_{\gamma}B_{\delta}A_{N}B_{N}
e^{M(mH|\xi|)}e^{M(mH|x|)}}{\langle(x,\xi)\rangle^{\rho|\gamma|+\rho|\delta|+2N\rho}},
\eeqs
for all $(x,\xi)\in Q^c_{Bm_N}$, $\gamma,\delta\in\NN^d$ and the estimates are uniform for $\alpha$ and $N$, $N>|\alpha|$. Analogously, one obtains similar estimates for the derivatives of $\ds \partial^{\alpha}_{\xi} a\cdot\left(\partial^{\alpha}_x b-\sum_{k=0}^{N-|\alpha|-1}\partial^{\alpha}_x b_k\right)$. Now we can estimate the derivatives of $\ds \partial^{\alpha}_{\xi} a\cdot\partial^{\alpha}_x b -\sum_{j=|\alpha|}^{N-1}\sum_{s+k=j-|\alpha|}\partial^{\alpha}_{\xi} a_s\partial^{\alpha}_x b_k$ and obtain the inequality in the lemma. Moreover, for fixed $\alpha\in\NN^d$, to obtain (\ref{zzs}) it only remains to consider the case when $N\leq |\alpha|$ (we already consider the case when $N>|\alpha|$ above). But then $\ds\sum_{j=|\alpha|}^{N-1}\sum_{s+k=j-|\alpha|}\partial^{\alpha}_{\xi} a_s\partial^{\alpha}_x b_k$ is empty and we only have to estimate the derivatives of $\partial^{\alpha}_{\xi} a\cdot\partial^{\alpha}_x b$ which is easy and we omit it ($a,b\in\Gamma_{A_p,B_p,\rho}^{*,\infty}\left(\RR^{2d}\right)$ and $\alpha$ is fixed).
\end{proof}

\begin{corollary}
Let $a,b\in\Gamma_{A_p,B_p,\rho}^{*,\infty}\left(\RR^{2d}\right)$ with asymptotic expansions $a\sim\sum_j a_j$ and $b\sim\sum_j b_j$. Then there exists $c\in\Gamma_{A_p,B_p,\rho}^{*,\infty}\left(\RR^{2d}\right)$ and *-regularizing operator $T$ such that $a(x,D)b(x,D)=c(x,D)+T$ and $c$ has the following asymptotic expansion
\beqs
c(x,\xi)\sim\sum_{j=0}^{\infty}\sum_{s+k+l=j}\sum_{|\alpha|=l}\frac{1}{\alpha!}\partial^{\alpha}_{\xi}
a_s(x,\xi)D^{\alpha}_x b_k(x,\xi).
\eeqs
\end{corollary}
\begin{proof} It is easy to check that the above formal sum is an element of $FS_{A_p,B_p,\rho}^{*,\infty}\left(\RR^{2d}\right)$. By theorem \ref{cef}, we only have to prove that
\beqs
\sum_{j=0}^{\infty}\sum_{|\alpha|=j}\frac{1}{\alpha!}\partial^{\alpha}_{\xi}a(x,\xi)D^{\alpha}_x b(x,\xi)\sim\sum_{j=0}^{\infty}\sum_{s+k+l=j}\sum_{|\alpha|=l}\frac{1}{\alpha!}\partial^{\alpha}_{\xi}
a_s(x,\xi)D^{\alpha}_x b_k(x,\xi).
\eeqs
For $N\in\ZZ_+$ and $(x,\xi)\in Q_{Bm_N}^c$, we have\\
$\ds \sum_{j=0}^{N-1}\sum_{|\alpha|=j}\frac{1}{\alpha!}\partial^{\alpha}_{\xi}a(x,\xi)D^{\alpha}_x b(x,\xi)-\sum_{j=0}^{N-1}\sum_{s+k+l=j}\sum_{|\alpha|=l}\frac{1}{\alpha!}\partial^{\alpha}_{\xi}
a_s(x,\xi)D^{\alpha}_x b_k(x,\xi)$
\beqs
&=&\sum_{j=0}^{N-1}\sum_{|\alpha|=j}\frac{1}{\alpha!}\partial^{\alpha}_{\xi}a(x,\xi)D^{\alpha}_x b(x,\xi)-\sum_{j=0}^{N-1}\sum_{l=0}^{j}\sum_{s+k=j-l}\sum_{|\alpha|=l}\frac{1}{\alpha!}\partial^{\alpha}_{\xi}
a_s(x,\xi)D^{\alpha}_x b_k(x,\xi)\\
&=&\sum_{j=0}^{N-1}\sum_{|\alpha|=j}\frac{1}{\alpha!}\partial^{\alpha}_{\xi}a(x,\xi)D^{\alpha}_x b(x,\xi)-\sum_{l=0}^{N-1}\sum_{j=l}^{N-1}\sum_{s+k=j-l}\sum_{|\alpha|=l}\frac{1}{\alpha!}\partial^{\alpha}_{\xi}
a_s(x,\xi)D^{\alpha}_x b_k(x,\xi)\\
&=&\sum_{j=0}^{N-1}\sum_{|\alpha|=j}\frac{1}{\alpha!}\left(\partial^{\alpha}_{\xi}a(x,\xi)D^{\alpha}_x b(x,\xi)-\sum_{l=j}^{N-1}\sum_{s+k=l-j}\partial^{\alpha}_{\xi}a_s(x,\xi)D^{\alpha}_x b_k(x,\xi)\right).
\eeqs
By lemma \ref{ase}, the derivatives of $\ds \partial^{\alpha}_{\xi}a(x,\xi)D^{\alpha}_x b(x,\xi)-\sum_{l=j}^{N-1}\sum_{s+k=l-j}\partial^{\alpha}_{\xi}a_s(x,\xi)D^{\alpha}_x b_k(x,\xi)$ can be uniformly estimated, as in the lemma, for all $\alpha$, $N$ and $(x,\xi)\in Q_{Bm_N}^c$, such that $|\alpha|<N$, from what the desired equivalence follows.
\end{proof}

\end{document}